\documentclass[11pt,leqno]{amsart}

\usepackage[all,cmtip]{xy}
\usepackage{a4,latexsym}
\usepackage[utf8]{inputenc}
\usepackage{amsfonts}
\usepackage[english]{babel}
\usepackage{amsxtra,amsmath,amsfonts,amscd, amssymb, mathrsfs, amsthm}
\usepackage{color, mathtools}
\usepackage{enumitem}
\usepackage{xcolor}
\usepackage{graphicx}
\usepackage{tikz-cd}
\usepackage{braket}
\usepackage[toc,page,title,titletoc,header]{appendix}

\definecolor{darkgreen}{rgb}{0,0.5,0}
\definecolor{darkblue}{rgb}{0,0,0.8}
\definecolor{darkred}{rgb}{0.8,0,0}
\definecolor{lightblue}{rgb}{0,0.6,0.8}
\usepackage[pdfencoding=auto,colorlinks,citecolor=darkgreen,linkcolor=darkblue,urlcolor=darkred]{hyperref}
\usepackage[capitalize]{cleveref}
\usepackage{array,longtable}

\numberwithin{paragraph}{section}
\newcommand{\splitcond}[1]{%
	\begin{tabular}[t]{@{}l@{}}#1\end{tabular}%
}

\setlist[enumerate]{label=\it{(\roman*)},
	ref=\it{(\roman*)}}

\newlist{enumerate1}{enumerate}{1}
\setlist[enumerate1]{resume,leftmargin=*,label=\it(\arabic*),ref=\it{(\arabic*)}}

\newlist{enumeratea}{enumerate}{1}
\setlist[enumeratea]{resume,leftmargin=*,label=\it(\alph*),ref=\it{(\alph*)}}

\newcommand{\C}{{\mathbb C}}

\renewcommand{\P}{{\mathbb P}}
\newcommand{\Q}{{\mathbb Q}}
\newcommand{\Z}{{\mathbb Z}}
\newcommand{\N}{{\mathbb N}}

\newcommand{\Span}{\mathrm{Span}}

\renewcommand{\Im}{\mathrm{Im}}

\newcommand{\GL}{\mathrm{GL}}
\newcommand{\Tr}{\mathrm{Tr}}
\newcommand{\Frac}{\mathrm{Frac}}
\newcommand{\Char}{\mathrm{Char}}

\newcommand{\Spec}{\mathrm{Spec}}

\newcommand{\Sp}{\mathrm{Sp}}
\newcommand{\Supp}{\mathrm{Supp}}

\newcommand{\Pic}{\mathrm{Pic}}

\newcommand{\OO}{\mathcal{O}}

\newcommand{\ord}{\mathrm{ord}}
\newcommand{\Gal}{\mathrm{Gal}}

\renewcommand{\mod}{\operatorname{mod}\nolimits}

\newcommand{\id}{\operatorname{id}\nolimits}

\renewcommand{\Im}{\operatorname{Im}\nolimits}
\newcommand{\Ker}{\operatorname{Ker}\nolimits}
\newcommand{\Coker}{\operatorname{Coker}\nolimits}
\renewcommand{\dim}{\operatorname{dim}\nolimits}

\newcommand{\rk}{\mathrm{rk}}

\newcommand{\red}{\mathrm{red}}

\newcommand{\an}{\mathrm{an}}

\newtheorem{theorem}{Theorem}[section]

\newtheorem{algorithm}[theorem]{Algorithm}

\newtheorem{corollary}[theorem]{Corollary}

\newtheorem{lemma}[theorem]{Lemma}
\newtheorem{proposition}[theorem]{Proposition}
\newtheorem{lemma*}{Lemma}

\newtheorem{theointro}{Theorem}

\theoremstyle{definition}
\newtheorem{definition}[theorem]{Definition}
\newtheorem{proposition&definition}[theorem]{Proposition\&Definition}
\newtheorem{lemma&definition}[theorem]{Lemma\&Definition}
\newtheorem{theorem&definition}[theorem]{Theorem\&Definition}
\newtheorem{example}[theorem]{Example}

\newtheorem{example*}{Example}
\newtheorem{remark}[theorem]{Remark}

\newtheorem{question*}{Question}

\newtheorem{art}[theorem]{}

\crefname{plemma}{Lemma}{lemmas}
\crefname{pdef}{Definition}{definitions}
\crefname{prmk}{Remark}{remarks}
\crefname{pprop}{Proposition}{propositions}
\crefname{palgm}{Algorithm}{algorithms}
\crefname{pcor}{Corollary}{corollaries}

\begin{document}
	
\title[Explicit Coleman integration over finite extensions of $\Q_{\lowercase{p}}$]
	{Explicit Coleman integration for curves\\ over finite extensions of $\Q_{\lowercase{p}}$}

\author{Yulin Cai}
\address{Yulin Cai, Hangzhou International Innovation Institute, Beihang University, Hangzhou 311115, China}
\email{ylcai5388339@gmail.com}

\author{Timo Keller}
\address{
Timo Keller,
IBM Deutschland Research \& Development, IBM Campus 1, 71139 Ehningen, Germany;
Institut für Mathematik, Universität Würzburg, Emil-Fischer-Strasse 30, 97074,
Würz\-burg, Germany; 
Rijksuniversiteit Groningen, Bernoulli Institute, Bernoulliborg, Nijenborgh 9, 9747 AG Groningen, The Netherlands; 
Leibniz Universität Hannover, Institut für Algebra, Zahlentheorie und Diskrete Mathematik, Welfengarten 1, 30167 Hannover, Germany}
\email{math@kellertimo.de}
\urladdr{\url{https://www.timo-keller.de}}
	
	
\begin{abstract}
	Coleman integration is an important computational ingredient of the Chabauty--Coleman method and its extensions. We extend the Balakrishnan--Tuitman algorithm to general smooth projective curves with good reduction over arbitrary finite extensions of $\Q_p$, allowing both ramification and nontrivial residue degree. Given a suitable plane model and integral bases, we provide explicit precision bounds and an accompanying
	{\sffamily Magma} implementation.
	As an application, a single-prime computation at $p=19$ over the unramified quadratic extension $\Q_{19^2}/\Q_{19}$ recovers all $35$ affine sporadic quadratic divisors on $X(s3,ns5)$ over $\Q(\sqrt5)$. Completeness is proved in an appendix using symmetric Chabauty and a Mordell--Weil sieve.
\end{abstract}
	

\maketitle

\setcounter{tocdepth}{1}

\tableofcontents


\section{Introduction}

\subsection*{Background}


Coleman integration is a theory of $p$-adic line integration on curves with good reduction. One of its principal applications is the Chabauty--Coleman method for studying rational points.
For a smooth projective curve $X/\Q$ of genus $g$ with Jacobian $J$, this method uses Coleman integrals to construct $p$-adic analytic functions vanishing at the rational points when $\operatorname{rank}J(\Q)<g$. Computing these functions explicitly requires algorithms for evaluating Coleman integrals.


Explicit Coleman integration has been developed in several settings.
Besser and de Jeu \cite{besser2008lip} gave an algorithm for
iterated integrals on $\P^1\setminus\{0,1,\infty\}$.
Balakrishnan, Bradshaw and Kedlaya \cite{balakrishnan2010explicit}
developed an algorithm for hyperelliptic curves.
Balakrishnan--Tuitman \cite{balakrishnan2020explicit} developed
a general algorithm for curves with good reduction over $\Q_p$,
with an accompanying {\sffamily Magma} implementation.
Best \cite[Chapter~4]{best2021tools} developed algorithms for
superelliptic curves over unramified extensions of $\Q_p$.

\subsection*{Motivation}

Two kinds of Diophantine computations naturally require Coleman
integration over local fields with nontrivial residue degree.
These applications motivate extending the available algorithms
to general curves admitting a suitable model over arbitrary
finite extensions of $\Q_p$.
\begin{itemize}
	\item \emph{Rational points over number fields.}
	Let $X/F$ be a smooth projective curve of genus $g$ over a number
	field $F$, and let $J$ be its Jacobian. Siksek's restriction-of-scalars
	variant of the Chabauty--Coleman method
	\cite[\S 4]{siksek2013explicit} can apply beyond the classical range
	$\operatorname{rank}J(F)<g$. It uses a rational prime $p$ unramified
	in $F$ and requires Coleman integrals over the completions $F_v$
	for all $v\mid p$.
	
	If $p$ splits completely in $F$, then
	\[
	F\otimes_\Q\Q_p\cong\Q_p^{[F:\Q]},
	\]
	so these integrals can be computed using an existing $\Q_p$
	implementation. However, completely split primes can be sparse:
	their density is
	\[
	\frac{1}{\#\Gal(F^{\Gal}/\Q)},
	\]
	where $F^{\Gal}$ is the Galois closure of $F$, see \cite[Exercise~I.9.4, Theorem~VII.13.4]{neukirch1999algebraic}.
	Large primes also increase the cost of treating residue classes.
	At a completely split prime of good reduction, the full residue
	space has cardinality
	\[
	\prod_{v\mid p}\#\widetilde X_v(\mathbb F_p)
	=p^{[F:\Q]}\bigl(1+O_{g,[F:\Q]}(p^{-1/2})\bigr),
	\]
	where $\widetilde X_v$ is the reduction at $v$ and the implicit
	constant depends only on $g$ and $[F:\Q]$.

For example, the remaining cases in Freitas--Stoll's study of
$x^2+y^3=z^{25}$ lead to five genus-$2$ curves
$Y^2=X^5+A_i$ over number fields $F_i$ of degree $12$, each a
quadratic extension of a sextic field; see
\cite[Theorem~1.1 and the following discussion]{stoll2025z25}.
An exhaustive treatment at a completely split prime involves
roughly $p^{12}$ residue tuples. Integration over extensions permits
the use of smaller primes that need not split completely.
The residue space still has size on the order of $p^{12}$, so
the benefit here is the greater choice of usable primes.
This superelliptic example illustrates the need for integration
over extensions, already addressed by Best's algorithms under
their hypotheses. The same need arises for curves outside
the scope of those algorithms.
	
	\item \emph{Points on symmetric powers.}
	Even for a curve $X/\Q$, studying effective divisors of degree $d\geq2$ can require integration over extensions of $\Q_p$.
	Write $X^{(d)}\coloneq\operatorname{Sym}^d(X)$ for the variety parametrizing these divisors, and let $\widetilde X/\mathbb F_p$ be the reduction at a prime of good reduction, and put $\widetilde X^{(d)}\coloneq\operatorname{Sym}^d(\widetilde X)$.
	
	Already for $d=2$, a divisor in
	$\widetilde X^{(2)}(\mathbb F_p)$ can consist of two
	$\mathbb F_p$-rational points or of two conjugate points defined
	over $\mathbb F_{p^2}$. Lifts of the latter points are naturally
	defined over the unramified quadratic extension $\Q_{p^2}$ of
	$\Q_p$. Direct Coleman-integral computations at such lifts therefore
	require integration over that extension.
	
	More generally, put $a_1=\#\widetilde X(\mathbb F_p)$. The number of effective degree-$d$ divisors supported entirely on $\mathbb F_p$-rational points is $\binom{a_1+d-1}{d}$.
	For fixed $g$ and $d$, as $p\to\infty$, the zeta-function enumeration of effective divisors gives
	\[\frac{\binom{a_1+d-1}{d}}{\#\widetilde X^{(d)}(\mathbb F_p)}=\frac{1}{d!}+O_{g,d}(p^{-1/2}),\]
	see \cite[\S 3, especially (3.1)--(3.2) and (3.8)]{Bruin2013}.
	Thus approximately half the quadratic residue classes, and a proportion tending to $1-1/d!$ in degree $d$, involve points whose residue fields are larger than $\mathbb F_p$.
	
	Box, Gajovi\'c and Goodman \cite[\S 2.3 and~3]{box2023cubic} address this difficulty by combining tiny integrals with a Mordell--Weil sieve. Their method uses Mordell--Weil information to restrict attention to residue classes containing known rational divisors. Integration over extensions also makes direct Coleman-integral computations possible in the other residue classes.
\end{itemize}

In both situations, the missing capability is integration over local fields with nontrivial residue-field extension.  Implementations over $\Q_p$, or even over totally ramified extensions of $\Q_p$, do not address this problem.

\subsection*{Main result}
We extend the explicit Coleman integration algorithm of Balakrishnan--Tuitman \cite{balakrishnan2020explicit} from $\Q_p$ to arbitrary finite extensions $K/\Q_p$, allowing both ramification and nontrivial residue degree.
We first describe the input. A polynomial $Q(x,y)$ gives a plane model of a smooth projective curve $X$ and a nonconstant map $x\colon X\to\P_K^1$. The matrices $W^0,W^\infty$ describe changes
from the power basis to integral function bases over $K[x]$ and
$K[1/x]$, respectively. The good-reduction conditions require
integrality of these changes of basis, compatibility of the bases
with reduction, and separation of the boundary points in reduction.
These explicit hypotheses on the chosen equation, map, and bases
are stated in \cref{def:triplehasgoodreduction}.

A differential is of the \emph{second kind} if its residue at every
pole is zero. The computational result is as follows.

\begin{theointro}[Coleman integration over finite extensions]
	\label{thm:main-integration}
	Let $F$ be a number field, let $\mathfrak P$ be a nonzero prime ideal
	of $\OO_F$, and put $K=F_{\mathfrak P}$. Normalize $v_K$ by
	$v_K(\pi)=1$ for a uniformizer $\pi$ of $K$.
	Let $X/F$ be a smooth projective geometrically irreducible curve
	of genus $g$, equipped with a triple $(Q,W^0,W^\infty)$ having
	good reduction over $K$ in the sense of
	\cref{def:triplehasgoodreduction}.
	Let $P,R\in X(K)$, and let $\omega_1,\dots,\omega_{2g}$ be
	differentials of the second kind whose classes form a basis of
	$H^1_{\mathrm{dR}}(X_K)$ and which are regular at $P$ and $R$.
	
	For every integer $N\geq1$, there is an explicit algorithm computing
	$I_{i,N}\in K$, for $1\leq i\leq2g$, such that
	\[
	v_K\left(I_{i,N}-\int_P^R\omega_i\right)\geq N.
	\]
\end{theointro}

The cohomology representatives can be chosen using \cref{algthm:cohomologybasis}; the integration algorithms
are given in
\cref{algthm:tinyintegral,algthm:semilinear-and-linear-frobenius-matrices,algthm:frobeniusofapoint,algthm:integralonbasis}, and their precision bounds are proved in \cref{precision bounds}.

As an application, we recover all $35$ affine sporadic quadratic divisors of the modular curve $C=X(s3,ns5)$ over $F=\Q(\sqrt5)$ using a single prime above $19$.
Our choice of $C$ and $F$ is motivated by Box's work on modularity of elliptic curves over totally real quartic fields: he proves modularity for fields not containing $\sqrt5$, while the remaining cases lead to quadratic points over $F$ on modular curves, some of which admit $C$ as a quotient after forgetting the level-$7$ structure \cite[Section~7.1]{box2022elliptic}.
Since $19$ splits in $F$, we compute over the unramified quadratic extension $\Q_{19^2}$ of $F_v\cong\Q_{19}$, recovering $12$ divisors in split residue classes and $23$ in Frobenius-conjugate classes; the latter require this extension for their constituent points.
The divisors are recognized over $F$ and verified exactly on $C$. The completeness is established separately by symmetric Chabauty and the Mordell--Weil sieve, see \cref{ex:app-singleprime,appendix:a sc and mw}.

\subsection*{Algorithm and implementation}

The algorithm combines local integration with Frobenius on cohomology. First, we compute the action of the $q$-power Frobenius on a cohomology basis, where $q$ is the cardinality of the residue field, and retain the exact-differential correction functions.
Next, we evaluate these functions at the endpoints and compute \emph{tiny integrals}, whose endpoints lie in the same residue disc and which can be evaluated using local power series.
Frobenius functoriality then gives a finite linear system for the desired integrals. The convergence and precision estimates control each step, including computations at bad endpoints.

Our accompanying {\sffamily Magma} package \cite{ColemanIntegralsLocalFields}, available at \[\text{\url{https://github.com/TimoKellerMath/ColemanIntegralsLocalFields},}\] implements these constructions for curves over completions of number fields.
We record three features of the algorithm and implementation.
\begin{enumerate}
	\item The implementation takes a number field $F$ and a prime ideal $\mathfrak P$ as input. To improve the running time, the main algebraic computations at working precision $N'$ use the exact quotient $\OO_F/\mathfrak P^{N'}$, rather than in a finite-precision local field.
	\item A direct $q$-power Frobenius lift is available over every $K$. When $K/\Q_p$ is unramified, we may instead use a $p$-power lift and iterate it $f_K$ times, where $f_K$ is the residue degree.
	This can reduce the computational cost.
	\item The implementation returns a certified output precision. With the curve and auxiliary endpoint choices fixed, this precision tends to infinity as the working precision $N_{\max}$ increases. Thus any prescribed output precision can be attained by increasing $N_{\max}$.
\end{enumerate}

The algorithm has two practical limitations. First, endpoints on the boundary of the chosen affine open are treated using tiny integrals and suitable auxiliary endpoints in the same residue discs. Evaluating the correction functions at these auxiliary endpoints can lose a large fraction of the working precision, making such computations expensive. Second, as in \cite{balakrishnan2020explicit}, the algorithm requires a supplied triple $(Q,W^0,W^\infty)$ having good reduction; constructing such a triple is outside its scope.

\subsection*{Organization of the paper}

\cref{section:Effective convergence bounds} extends the effective convergence bounds of \cite{kedlaya2012effective} to finite extensions of $\Q_p$.
\cref{section:good models} introduces good models and studies their properties.
\cref{section:computing a frobenius lift} constructs a Frobenius lift. 
\cref{section:computing a basis of cohomology} computes de Rham cohomology of the relevant affine open of the curve and describes the reduction of differential forms.
These three sections follow the framework of \cite[Sections~2--3]{tuitman2017counting}, which treats unramified extensions of $\Q_p$ of arbitrary residue degree; we formulate the constructions over general finite $K/\Q_p$ and retain the proofs to justify their use in the ramified case.
\cref{computing Coleman integrals} gives the algorithms for Coleman integration, and 
\cref{precision bounds} establishes their precision bounds. 
\cref{section: implementation} presents the implementation and examples.
\cref{appendix:a sc and mw} determines $X(s3,ns5)^{(2)}(\Q(\sqrt5))$ using symmetric Chabauty and a Mordell--Weil sieve.


\subsection*{Notation}

Throughout the paper, $K$ is a finite extension of $\Q_p$.
We fix a uniformizer $\pi$ and normalize the valuation $v_K$ by
$v_K(\pi)=1$. We denote its ring of integers, maximal ideal,
and residue field by $K^\circ$, $K^{\circ\circ}$, and
$\widetilde K$, respectively. Write
\[
e_K=v_K(p),\qquad f_K=[\widetilde K:\mathbb F_p],\qquad
q=q_K=\#\widetilde K=p^{f_K}.
\]
Thus $[K:\Q_p]=e_Kf_K$.
We fix an algebraic closure $\overline K$ and extend $v_K$ to it.
The absolute value is $|a|=p^{-v_K(a)/e_K}$, with $v_K(0)=+\infty$.
We use $F$ for a number field and $F_v$ for its completion at a
finite place $v$.

We fix an automorphism $\tau\in\operatorname{Aut}(K/\Q_p)$.
Let $1\leq\ell_\tau\leq f_K$ be determined by
$\widetilde\tau(a)=a^{p^{\ell_\tau}}$ for $a\in\widetilde K$.
We assume $\ell_\tau\mid f_K$ and set
\begin{equation*}\label{eq:frobenius-exponents}
	q_\tau=p^{\ell_\tau},\qquad
	n_\tau=\frac{f_K}{\ell_\tau},\qquad q_\tau^{n_\tau}=q.
\end{equation*}
We require
\begin{equation*}\label{eq:residue-frobenius-compatibility}
	\tau^{n_\tau}=\id_K.
\end{equation*}
In particular, $n_\tau=\operatorname{ord}(\tau)$.
For a polynomial, rational function, series, or matrix $h$, we
write $h^\tau$ for coefficientwise application of $\tau$,
leaving the variables fixed. We write
\[
\sigma_\tau\colon K(t)\longrightarrow K(t),\qquad
h(t)\longmapsto h^\tau(t^{q_\tau}),
\]
and denote its associated scheme morphism by
$\sigma_{\tau,\flat}\colon\P_K^1\to\P_K^1$, lying over the
automorphism of $\Spec(K)$ induced by $\tau$.
The map $\sigma_\tau$ is a ring homomorphism and is $K$-linear
if and only if $\tau=\id_K$.

By an algebraic variety over a field $k$, we mean a geometrically integral, separated scheme of finite type over $k$. A curve over $k$ is an algebraic variety of dimension one.
For an algebraic variety $X/K$, we denote by $\Omega^1_{X/K}$ its sheaf of relative differentials and by $X^\an$ its Berkovich analytification by $X^\an$, see \cite[Theorem 3.4.1]{berkovich1990spectral}.

For a smooth curve $X/K$ and $P\in X(\overline K)$, let $K(P)$ denote the smallest extension of $K$ over which $P$ is defined.
For $f\in K(X)$, we write $\ord_P(f)$ for its order of vanishing at $P$, with negative values indicating poles. If $f$ is regular at $P$, we write $f(P)\in K(P)$ for its value at $P$. We use the same evaluation notation for $K$-analytic functions defined on a neighborhood of $P$.

We write $D(a,r)$ and $\overline D(a,r)$ for the open and closed discs of radius $r>0$ centered at $a\in K$ in $(\mathbb A_K^1)^\an$. The Tate algebra $K\langle t\rangle$ consists of power series converging on $\overline D(0,1)$.
We write
\[
K\{t\}=\left\{\sum_{n\geq0}a_nt^n\in K[[t]]
\;\middle|\; |a_n|r^n\to0\text{ for every }0<r<1\right\}
\]
for the ring of analytic functions on $D(0,1)$.
For $0<\alpha<\beta$, the notation $K\langle\alpha/t,t/\beta\}$ denotes the ring of Laurent series
converging on the half-open annulus $\alpha\leq|t|<\beta$.

For a ring $R$, we write $M_{a\times b}(R)$ for the set of
$a\times b$ matrices over $R$, and $M_d(R)=M_{d\times d}(R)$.
Basis tuples such as $\boldsymbol{b}^0$ and
$\boldsymbol{b}^\infty$ are row vectors; their coefficient vectors
are columns. For $w\in R^{\oplus d}$, the components $w_i$ are
indexed by $0\leq i\leq d-1$, whereas matrix indices start at $1$.
For a nonzero polynomial $h$, we denote its leading coefficient by $\operatorname{lc}(h)$ and use $\deg(0)=-\infty$.

For a polynomial over $K$, $v_K$ denotes the minimum valuation of its coefficients. We extend this valuation to $K(t)$ by $v_K(a/b)=v_K(a)-v_K(b)$ for nonzero polynomials $a,b$.
For a convergent Laurent expansion $f=\sum_{i\in\Z}a_it^i$, we set
\[v_K(f)=\inf_{i\in\Z}v_K(a_i).\]
For a matrix $A=(A_{ij})$, $v_K(A)$ is the minimum of the valuations of its entries. Similarly, for a matrix of rational functions, $\ord_z(A)$ is the minimum of the orders of its entries at $z$.

\section{Effective convergence bounds}
\label{section:Effective convergence bounds}
In this section, we follow the idea of \cite{kedlaya2012effective} and extend their main result to the one over any finite extension of $\Q_p$.

\begin{art}
	Let $X$ be an algebraic variety over $K$, and $(\Omega_{X/K}^1,d\colon \OO_X\to \Omega_{X/K}^1)$ the relative differential sheaf of $X$ over $K$. Let $\mathcal{E}$ be an $\OO_X$-module. A \emph{connection} on $\mathcal{E}$ is a $K$-linear map
	\[\nabla: \mathcal{E}\rightarrow \Omega_{X/K}^1\otimes_{\OO_X}\mathcal{E}\]
	satisfying the Leibniz rule: for any local sections $f$ and $s$ of $\OO_X$ and $\mathcal{E}$, respectively, we have
	\[\nabla(fs)=f\nabla(s)+ df\otimes s.\] 
\end{art}

In this section, we fix an open subset $U\subset \P_K^1$ such that $Z\coloneq\P_K^1\setminus U$ is nonempty. We also fix a connection $(\mathcal{E},\nabla)$ on $U$ with $\mathcal{E}$ a free sheaf. Let $v_1,\dots, v_n$ be a basis of $\mathcal{E}$. 
Recall that $\sigma_{\tau,\flat}$ is induced by
$f(t)\mapsto f^\tau(t^{q_\tau})$. We denote by $V\subset U^\an$
the complement of the union of open discs of radius $1$ around the points
of $Z$. We assume that $\infty\not\in V$. Hence
$V\subset \overline{D}(0,1)\subset (\mathbb{A}^1_K)^\an$ and
$\Omega_{V/K}^1 = \OO_Vdt$.

\begin{art}
A \emph{strict neighborhood} $W$ of $V$ is an analytic domain of $\P^{1,\an}_K$ such that $\{W, \P^{1,\an}_K\setminus V\}$ is an admissible covering of $\P^{1,\an}_K$ in the sense of \cite[\S 1.6]{berkovich1993etale}. We set $\OO^\dagger(U)$ the ring of functions that converge on some strict neighborhood of $V$.
\end{art}   

\begin{art}\label{q-frobenius structure def}
A \emph{$q_\tau$-Frobenius structure} on
$(\mathcal{E},\nabla)$ is an isomorphism
$F\colon\sigma_{\tau,\flat}^*\mathcal{E}\rightarrow\mathcal{E}$
of free sheaves with connection defined on some strict neighborhood of $V$.
\end{art}

The following lemma is from definition, see also \cite[Observation~5.1.3]{kedlaya2012effective1} and \cite[Section~2]{kedlaya2012effective}.

\begin{lemma}\label[plemma]{lemma:matricesofbasischange}
	Assume that $\infty\not\in V$, i.e. $V\subset (\mathbb{A}^1_K)^\an$. Let $(\mathcal{E},\nabla,F)$ be a $q_\tau$-Frobenius structure on $V$ with respect to $\sigma_{\tau,\flat}$. 
	We set $N=(N_{ij})\in M_n(\OO_U(U))$, $\Phi=(\Phi_{ij})\in M_n(\OO^\dagger(U))$ such that \[\nabla(v_j)=\sum_{i=1}^nN_{ij}v_i\otimes dt \ \ (j=1,\dots, n),\]
	\[F(v_j) = \sum_{i=1}^n\Phi_{ij}v_i \ \ (j=1,\dots, n).\]
	\begin{enumerate1}
		\item \label{matrix of connection by base change} If $w_1,\dots, w_n$ is another local basis of $\mathcal{E}$, i.e. there is a matrix $W=(W_{ij})\in \GL_n(K(t))$ such that $w_j=\sum_{i=1}^nW_{ij}v_i$, then the matrix of $\nabla$ with respect to $w_1,\dots, w_n$ is 
		\[W^{-1}NW+W^{-1}\dfrac{d}{dt}W,\]
		and the matrix of $F$ with respect to $w_1,\dots, w_n$ is
		\[W^{-1}\Phi\sigma_\tau(W).\]
		\item \label{matrix of Frobenius by base change} The matrices $N$ and $\Phi$ satisfy the following compatibility equation:
		\begin{align}
			N\Phi+\frac{d}{dt}(\Phi)
			=q_\tau t^{q_\tau-1}\Phi\sigma_\tau(N).
			\label{equationoffrobeniusstructure}
		\end{align}
	\end{enumerate1}
\end{lemma}

\begin{definition} 	\label[pdef]{def:connectionregularatz}
	Let $z\in \overline{K}$, and $N$ the matrix of $\nabla$ with respect to the basis $v_1,\dots, v_n$ as in \cref{lemma:matricesofbasischange}. We say that $v_1,\dots, v_n$ on $U$ is \emph{regular at $z$} (or \emph{Fuchsian at $z$}) if the matrix $(t-z)N$ is holomorphic in a neighborhood of $z$. We say that $\nabla$ is \emph{regular at $z$} if it admits a regular basis on some neighborhood of $z$.
\end{definition}

\begin{definition} 	\label[pdef]{def:exponentofmatrix}	
	Keep the notation in \cref{def:connectionregularatz}. Suppose that the basis $v_1,\dots, v_n$ is regular at $z$. The \emph{residue matrix at $z$} of this basis is $((t-z)N)|_{t=z}$ (notice that if $z\in U$ then $(t-z)N|_{t=z}=0$). The \emph{exponents of $\nabla$} at $z$ with respect to $v_1,\dots, v_n$ are the eigenvalues of the residue matrix $((t-z)N)|_{t=z}$.
\end{definition}

\begin{remark}
\cref{def:exponentofmatrix} and \cref{def:connectionregularatz} are invariant under automorphisms of $\P^1_K$, so they can be defined at $\infty\in \P_K^1$: let $s=t^{-1}$, then $N$ is regular at $z=\infty$ if $-s^{-1}N(s^{-1})$ is holomorphic in a neighborhood of $s=0$. Similarly, $-tN|_{t=\infty}$ is the residue matrix at $z=\infty$.

		To show our claim above, let $t=\gamma(s)\in K(s)$ be a M\"obius transform over $K$ such that $w\coloneq\gamma^{-1}(z)\in \overline{K}$. The matrix of $\nabla$ becomes $N(\gamma(s))\gamma'(s)$. Notice that $\frac{s-w}{\gamma(s)-\gamma(w)}\gamma'(s)|_{s=w} = 1$ and $t=\gamma(s)$ is an automorphism, so $\gamma'(w)\not=0$, and $(s-w)N(\gamma(s))\gamma'(s)$ is holomorphic if and only if $(t-z)N$ is so. Moreover, 
		\[(s-w)N(\gamma(s))\gamma'(s)|_{s=w} = (\gamma(s)-\gamma(w))N(\gamma(s))|_{s=w}=(t-z)N(t)|_{t=z}.\]
		In our case, we take $\gamma(s)=s^{-1}$ and $w=0$.
\end{remark}

The main result of this section is the following theorem.

\begin{theorem} 	\label{thm:degreeoffrobeniusstructuresatapoint}
	Let $z\in Z(\overline{K})$. Let $F$ be a Frobenius structure on $(\mathcal{E},\nabla)$ with respect to $\sigma_{\tau,\flat}$, and $\Phi=(\Phi_{ij})\in M_n(\OO^\dagger(U))$ be the matrix of $F$ with respect to the basis $v_1,\dots, v_n$. Assume that 
	\begin{itemize}
		\item $Z$ does not contain any other point with the same reduction as $z$ modulo $K^{\circ\circ}$ (the model of $\P_K^1$ is $\P_{K^\circ}^1$);
		\item $v_1,\dots, v_n$ is regular at $z$, and the exponents $\{\lambda_1,\dots, \lambda_n\}$ of $\nabla$ at $z$ are contained in $\Q\cap K^\circ$.
	\end{itemize}  For $i\in \N$, set 
\begin{equation*}
	f(i)\coloneq \begin{cases} 
	
		0 & \text{ if $i=0$};\\[4ex]
		\!\begin{aligned}[t]
			\max&\{(v_K(\Phi)+v_K(\Phi^{-1}))(\lfloor\log_{q_\tau}i\rfloor+1), \\&(n-1)\min\{0,v_K(N)\}-e_K(n-1)\left(\lfloor\log_pi\rfloor-\ell_\tau\lfloor\log_{q_\tau}i\rfloor\right)\\
			&\quad+(v_K(\Phi)+v_K(\Phi^{-1}))\lfloor\log_{q_\tau}i\rfloor\}
		\end{aligned} & \text{ otherwise,}
	\end{cases} 
\end{equation*}
	and define
	\[c\coloneq\begin{cases}
		0& \text{ if $e_K\leq p-1$ and $v_K(N)\geq 0$,}\\
		\min\{0, i+f(i)\mid i\in \N\}& \text{ otherwise}.
	\end{cases}\]
For $m\in \N$, put
\[\varphi(m)\coloneq \max\{i\in\N\mid i+v_K(\Phi)+c+ f(i)<m\},\]
with the convention that this maximum is $0$ when the set is empty, and define
	\begin{align*}
		\alpha_1\coloneq&\lfloor -q_\tau\min_i\{\lambda_i\}+\max_i\{\lambda_i\}\rfloor,\\
		\alpha_2\coloneq&\begin{cases}
			0 & \text{ if $N$ does not have a pole at $z$,}\\
			0& \text{ if $z\in \{0,\infty\}$,}\\
			\varphi(m)& \text{ otherwise.}
		\end{cases}
	\end{align*}
	Then for any $1\leq i,j\leq n$, there is a rational function $h(t)/g(t)\in K(t)$ with no pole other than $z$ having the same reduction as $z$, of order $\geq -(\alpha_1+q_\tau\alpha_2)$ at $z$, such that $\Phi_{ij}-h(t)/g(t)\in (K^{\circ\circ})^m[[ t-z, (t-z)^{-1}]]$.
\end{theorem}

\begin{remark}\label{remark:degreeoffrobeniusstructuresatapoint}
 Keep the notation in \cref{thm:degreeoffrobeniusstructuresatapoint}. In this paper, only the case where $v_K(\Phi), v_K(\Phi^{-1})\geq 0$ (which forces $v_K(\Phi)=v_K(\Phi^{-1})= 0$ since $v_K(\Phi)+v_K(\Phi^{-1})\leq 0$) will be considered. In this case,  we have $f(i)\geq 0$ and $c=0$, so $\varphi(m)\leq m-1$ for $m\geq1$ and 
		$-(\alpha_1+q_\tau\alpha_2)\geq -(\alpha_1+q_\tau(m-1))$. If moreover $\lambda_1,\dots, \lambda_n\in[k,k+1)$ for $k\in\Z$, then $$-(q_\tau-1)(k+1)< -q_\tau\min_i\{\lambda_i\}+\max_i\{\lambda_i\}< (1-q_\tau)k+1,$$ so $-(q_\tau-1)(k+1)\leq \alpha_1\leq (1-q_\tau)k$ and 
		\[-(\alpha_1+q_\tau\alpha_2)\geq-q_\tau(m-k-1)-k.\]
\end{remark}

\begin{lemma}	\label{lemma:orderoffrobeniusstructureat0}
	Let $N=\sum_{i=-1}^\infty N_it^i\in M_n(K[[ t,t^{-1}]])$ such that $tN\in M_n(K\{ t\})$ and the eigenvalues $\lambda_1,\dots, \lambda_n$ of $N_{-1}$ are rational numbers with denominator coprime to $p=\Char(\widetilde{K})$. Let $\Phi = \sum_{i=-\infty}^\infty\Phi_it^i\in M_n(K\langle\alpha_0/t, t\})$ for some $0<\alpha_0<1$. Suppose that $N, \Phi$ satisfy \eqref{equationoffrobeniusstructure}. Then $\Phi_i=0$ wherever
	\[i<q_\tau\min_j\{\lambda_j\} - \max_j\{\lambda_j\}.\]
\end{lemma}
\begin{proof}
	The proof for $K$ unramified over $\Q_p$ and $q_\tau=p$ (that is, over $\Q_q$ with the semilinear $p$-power Frobenius lift) is given in \cite{kedlaya2012effective}, but it can be easily extended to general $K$ and $q_\tau$.
	
	If $\lambda_1=\dots=\lambda_n=0$, this is from \cite[Proposition~17.5.1]{kedlaya2010p-adic}. Indeed, $tN$ defines a differential operator $D$ on $K\{ t\}^{\oplus n}$ with respect to $td/dt$: for $v\in K\{ t\}^{\oplus n}$, 
	\[D(v)\coloneq  tNv+t\frac{d}{dt}(v),\]
	then the equality
	\[tN\Phi+t\dfrac{d}{dt}(\Phi)=\frac{t}{\sigma_\tau(t)}\dfrac{d\sigma_\tau(t)}{dt}\Phi\sigma_\tau(tN),\]
	shows that $D, \Phi$ define a differential operator and a compatible Frobenius structure on $K\langle\alpha_0/t, t\}^{\oplus n}$. By \cite[Proposition~17.5.1]{kedlaya2010p-adic}, we know that $\Phi\in M_n(K\{ t\})$, in particular, $\Phi_i=0$ for $i<0$.
	
In general, we first reduce to the case $\lambda_1,\dots,\lambda_n\in\Z$. For $k\in\N_{\geq1}$ coprime to $p$, consider the homomorphism $\alpha\colon K[[ t,t^{-1}]]\to K[[ s,s^{-1}]]$, $t\mapsto s^k$. It satisfies $(s\frac{d}{ds})\circ\alpha=k\cdot\alpha\circ(t\frac{d}{dt})$.

    After the base change $\alpha^*$, matrices $N, \Phi$ are replaced by 
	\[ks^{-1}\alpha(tN) = \sum_{i=0}^\infty kN_{i-1}s^{ik-1}, \ \ \alpha(\Phi)=\sum_{i=-\infty}^\infty\Phi_is^{ki}\]
	respectively.
	Notice that the eigenvalues of $kN_{-1}$ are $k\lambda_1,\dots, k\lambda_n$; hence, if the statement holds for $ks^{k-1}\alpha(N), \alpha(\Phi)$, then it also holds for $N, \Phi$. So we can assume that $\lambda_1,\dots, \lambda_n\in \Z$.
	
	Set $a\coloneq\min_{j}\{\lambda_j\}, b\coloneq\max_{j}\{\lambda_j\}$. If $\lambda_1,\dots, \lambda_n\in \Z$, by shearing transformations, i.e. \cite[Lemma~5.1.6]{kedlaya2012effective1}, we can find a matrix $W=(W_{ij})\in \GL_n(K(t))$ such that $t^bW$ and $t^{-a}W^{-1}$ are regular at $0$, and the basis $w_1, \dots, w_n$ of $\mathcal{E}$ defined by $w_j= \sum_{i=1}^nW_{ij}v_i$ is regular at $0$ with all exponents equal to $0$. Notice that the matrix of $F$ with respect to $w_1,\dots, w_n$ is $\Phi'=\sum_{i=-\infty}^\infty\Phi'_it^i\coloneq W^{-1}\Phi\sigma_\tau(W)$ by \cref{lemma:matricesofbasischange}~\ref{matrix of connection by base change}, and $\Phi_i'=0$ for $i<0$ from our discussion above. Since $\Phi = W\Phi'\sigma_\tau(W^{-1})$, we have that $\Phi_i=0$ for all $i< q_\tau a-b$.
\end{proof}

In the following, let $D:\mathcal{E}\rightarrow \mathcal{E}$ denote the differential operator on $\mathcal{E}$ defined by $\nabla v=Dv\otimes dt$. 

\begin{lemma}
	\label[plemma]{lemma:changeoffrobeniusstructures}
Let $\sigma_{1,\flat}, \sigma_{2,\flat}\colon \P_K^1\rightarrow \P_K^1$ be morphisms induced by $q_\tau$-power Frobenius lifts $\sigma_1, \sigma_2\colon K(t)\to K(t)$, respectively, both restricting to the same automorphism $\tau$ on $K$. Suppose that $\mathcal{E}$ admits a Frobenius structure $F_1\colon \sigma_{1,\flat}^*\mathcal{E}\overset{\sim}{\rightarrow}\mathcal{E}$ with respect to $\sigma_{1,\flat}$. Then $\mathcal{E}$ also admits a Frobenius structure $F_2\colon \sigma_{2,\flat}^*\mathcal{E}\overset{\sim}{\rightarrow}\mathcal{E}$ with respect to $\sigma_{2,\flat}$ defined by 
\[F_2(v)=\sum_{i=0}^\infty(\sigma_2(t)-\sigma_1(t))^iF_1\left(\frac{D^i}{i!}(v)\right).\]
Notice that $F_1(v)=\sum_{i=0}^\infty(\sigma_1(t)-\sigma_2(t))^iF_2\left(\frac{D^i}{i!}(v)\right)$.
\end{lemma}
\begin{proof}
See \cite[Proposition~17.3.1]{kedlaya2010p-adic}.
\end{proof}

\begin{lemma}	\label{lemma:valuationboundofdifferentials}
	Let $\Delta^{(i)}\in M_n(K[[ t-z,(t-z)^{-1}]])$ be the matrix of the differential operator $\dfrac{D^i}{i!}$ with respect to the basis $v_1, \dots, v_n$, i.e. 
	$\left(\dfrac{D^i}{i!}\right)v_k=\sum_{j=1}^n\Delta_{jk}^{(i)}v_j.$ Then we have
	\[v_K(\Delta^{(i)})\geq f(i),\]
	where $f(i)$ is defined as in \cref{thm:degreeoffrobeniusstructuresatapoint}.
\end{lemma}
\begin{proof}
Let $\eta$ be the generic point of $\overline{D}(z,1)$, i.e. there is a complete field extension $L/K$ such that $\eta\in L$ with $|\eta-z|=1$ and $|w-z-\eta|\geq 1$ for any $w\in \overline{K}\subset \overline{L}$. We have an isometric map $K\langle t-z\rangle\hookrightarrow L\langle t-\eta\rangle, \ \ t-z\mapsto (\eta-z)+(t-\eta)$. 

The Taylor series 
\[T(-t+\eta,v_j)\coloneq\sum_{i=0}^\infty(-t+\eta)^i\frac{D^i(v_j)}{i!}=\sum_{i=0}^\infty(-t+\eta)^i\sum_{k=1}^n\Delta_{kj}^{(i)}v_k\]
converges around $\eta$, so $T(-t+\eta,v_1),\dots, T(-t+\eta,v_n)$ form a horizontal basis of $\mathcal{E}$ around $\eta$. Let $M=\sum_{i=0}^\infty M_i(t-\eta)^i\coloneq(\sum_{i=0}^\infty(-t+\eta)^i\Delta_{jk}^{(i)})_{j,k}\in M_n(L[[ t-\eta]])$ be the matrix such that 
\[\begin{bmatrix}
T(-t+\eta,v_1)&\dots& T(-t+\eta,v_n)
\end{bmatrix}=\begin{bmatrix}
v_1&\dots& v_n
\end{bmatrix}M.\]
Notice that $M_0=I_n$ is the identity matrix.
Since the open disc of radius $1$ around $\eta$ contains
no point of $Z(\overline K)$, the connection has no
singularity in this disc. Hence
$\Delta^{(i)}\in M_n(L[[ t-\eta]])$.

By the argument of
\cite[Theorem~18.3.3]{kedlaya2010p-adic},
with the corrected iteration count in
\cite[Lemma~18.3.2]{kedlaya2010p-adic},%
\footnote{The recurrence in Lemma~18.3.2 replaces an index
	$j$ by $\lfloor j/q_\tau\rfloor$. Starting from $i\geq1$,
	it takes $\lfloor\log_{q_\tau}i\rfloor+1$ iterations to
	reach $0$, rather than the printed
	$\lceil\log_{q_\tau}i\rceil$.}
for every $i\in\N_{\geq1}$, we obtain
\[
\min\{v_K(M_0),\dots,v_K(M_i)\}
\geq
\bigl(v_K(\Phi)+v_K(\Phi^{-1})\bigr)
\bigl(\lfloor\log_{q_\tau}i\rfloor+1\bigr),
\]
where $v_K$ denotes also its extension to $L$.
By \cite[Remark~18.3.4]{kedlaya2010p-adic}, we also have
\begin{align*}
\min\{v_K(M_0),\dots, v_K(M_i)\}\geq& (n-1)\min\{0,v_K(N)\}-e_K(n-1)\left(\lfloor\log_pi\rfloor-\ell_\tau\lfloor\log_{q_\tau}i\rfloor\right)+\\
&(v_K(\Phi)+v_K(\Phi^{-1}))\lfloor\log_{q_\tau}i\rfloor.
\end{align*}
These two inequalities show that
\[\min\{v_K(M_0),\dots, v_K(M_i)\}\geq f(i).\]
Notice that $v_K(\Delta^{(i)}) = v_K(\Delta^{(i)}(\eta))$ (see argument in \cite[Remark~9.7.2]{kedlaya2010p-adic}), and 
\[M_i= (-1)^i\Delta^{(i)}(\eta)+(\text{terms coming from $\Delta^{(j)}$ with $j<i$}).\]
We can prove that $v_K(\Delta^{(j)}(\eta))\geq f(i)$ for $j\leq i$ by induction on $j$, this implies the lemma. We fix an $i\in\N$. For $j=0$, we have that $\Delta^{(0)}=I_n=M_0$, so $v_K(\Delta^{(0)}(\eta))=v_K(M_0)\geq f(i)$. For $1\leq j\leq i$, assume that $v_K(\Delta^{(k)}(\eta))\geq f(i)$ (in particular $v_K(\Delta^{(k)})\geq f(i)$) for $0\leq k< j$. If $v_K(\Delta^{(j)}(\eta))<f(i)$, then 
\[v_K(M_j)=v_K(\Delta^{(j)}(\eta))<f(i),\]
contradiction. Hence $v_K(\Delta^{(j)}(\eta))\geq f(i)$. This completes the proof of the lemma.
\end{proof}

\begin{proof}[Proof of \cref{thm:degreeoffrobeniusstructuresatapoint}]
Let \(I_K\subseteq\Gal(\overline K/K)\) be the inertia group. Its fixed field \(\overline K^{I_K}=K^{\mathrm{ur}}\) is the maximal unramified extension of \(K\). For every \(\rho\in I_K\), the points \(\rho(z)\) and \(z\) have the same reduction. Since \(Z\) is defined over \(K\), we have \(\rho(z)\in Z(\overline K)\). The hypothesis on distinct reductions therefore implies \(\rho(z)=z\). Hence \(K(z)\subseteq K^{\mathrm{ur}}\), so \(K(z)/K\) is unramified.
Since \(K(z)/K\) is finite unramified, by \cite[Tag~04GK]{stacks-project}, we can extend the fixed automorphism \(\tau\) to \(K(z)\) with residue action \(a\mapsto a^{q_\tau}\). The valuations used in the local bounds are unchanged under this extension.

For the descent, suppose $z\notin K$, put $L=K(z)$, and let $P_z$ be the monic minimal polynomial of $z$.
Write $b=\alpha_1+q_\tau\alpha_2$ and $f=P_z^b\Phi_{ij}$.
An approximant has order at least $-b$ at $z$ precisely when multiplying it by $P_z^b$ makes it regular there.

The local calculation below shows that every negative-power coefficient in the Laurent expansion of $f$ at $z$ is divisible by $\pi^m$.
The same holds at each conjugate of $z$.
Using the Mittag--Leffler decomposition \cite[Proposition~2.2.6]{fresnel2004rigid}, subtract these negative-power series at all the conjugates.
The resulting function $f_0$ is regular across their residue discs and still satisfies $f_0\equiv f\pmod{\pi^m}$.
It is defined over $K$, since Galois permutes the subtracted series and their sum is therefore invariant.

Let $V'=V\cup\bigcup_{\gamma\in\Gal(L/K)}D(\gamma(z),1)$. Then $f_0$ is an analytic function on $V'$.
Rational approximation on affinoid subsets of \(\mathbf P_K^1\) gives a rational function \(r\in K(t)\), with no poles on \(V'\), such that
\[\sup_{n\ge0}|a_n|\le\|r-f_0\|_{V'}\le|\pi|^m,\]
where $r-f_0=\sum_{n\geq0}a_n(t-z)^n$ is the Taylor expansion at $z$, and $\|\cdot\|_{V'}$ is the supremum norm on $V'$.
Then $r/P_z^b\in K(t)$ is the desired approximant: its order at $z$ is at least $-b$, because $P_z$ has a simple zero there, and its precision is unchanged, because $|P_z|=1$ on $V$.
Thus we may assume $z\in K$ in the local calculation.
 
When $z=0$ or $z=\infty$, this is from \cref{lemma:orderoffrobeniusstructureat0}.

Suppose that $N$ has no pole at $z$, so $\alpha_1=\alpha_2=0$. We consider the other $q_\tau$-power Frobenius lift $\sigma'$ with $\sigma'(t-z)=(t-z)^{q_\tau}$. Let $F'\colon \sigma_\flat'^*\mathcal{E}\simeq \mathcal{E}$ be the corresponding Frobenius structure with respect to $\sigma_\flat'\colon \P_K^1\to\P_K^1$ (which is induced by $\sigma'$) in \cref{lemma:changeoffrobeniusstructures} (with $\sigma_1=\sigma', \sigma_2=\sigma_\tau$), and $\Phi'$ the matrix of $F'$. By \cref{lemma:orderoffrobeniusstructureat0}, we know that $\Phi'$ is holomorphic at $z$. Then \cref{lemma:changeoffrobeniusstructures} implies that $\Phi$ is also holomorphic at $z$. This proves \cref{thm:degreeoffrobeniusstructuresatapoint} in this case.

Suppose that $N$ has a pole at $z$. For $\sigma', F', \Phi'$ as above, $\Phi'$ has order $\geq-\alpha_1$ at $z$ by \cref{lemma:orderoffrobeniusstructureat0}. As before, we use \cref{lemma:changeoffrobeniusstructures} to have the equality
\[\Phi=\sum_{i=0}^\infty \pi^iu^i\Phi'\sigma'(\Delta^{(i)}),\]
where $\pi u=(t^{q_\tau}-\tau(z))-(t-z)^{q_\tau}$ with $v_K(u)\geq0$, and $\Delta^{(i)}$ is  the matrix of the differential operator $\dfrac{D^i}{i!}$ with respect to the basis $v_1, \dots, v_n$. Notice that the $i$-th term in the summation has order $\geq -\alpha_1-q_\tau i$ at $z$ and no other pole with the same reduction as $z$, and valuation $\geq i+v_K(\Phi')+f(i)$ by \cref{lemma:valuationboundofdifferentials}. The theorem follows from the inequality $v_K(\Phi')\geq v_K(\Phi)+c$. In fact, by \cref{lemma:changeoffrobeniusstructures}, we have
\[\Phi'=\sum_{i=0}^\infty\frac{\pi^i}{i!}(-u)^i\Phi\sigma_\tau(i!\Delta^{(i)}).\]
If $e_K\leq p-1$ and $v_K(N)\geq 0$, notice that $v_K(\pi^i/i!)=i-e_Kv_p(i!)\geq i-\frac{ie_K}{p-1}$ and $v_K(i!\Delta^{(i)})$ are both nonnegative, so $v_K(\Phi')\geq v_K(\Phi)+c$ with $c=0$; otherwise, we directly use \cref{lemma:valuationboundofdifferentials}.

\end{proof}

\begin{corollary}
	\label[pcor]{cor:degreeoffrobeniusstructureatnonregularpoints}
	Keep the notation and hypotheses of \cref{thm:degreeoffrobeniusstructuresatapoint}. Let $w_1,\dots, w_n$ be another basis of $\mathcal{E}$ such that $v_j=\sum_{i=1}^nW_{ij}w_i$ with $W=(W_{ij})\in \GL_n(K(t))$, and $\Phi'=(\Phi_{ij}')\in M_n(\OO^\dagger(U))$ the matrix of $F$ with respect to $w_1,\dots, w_n$. Then for any $1\leq i,j\leq n$ and $m\in \N$, there is a rational function $f(t)/g(t)\in K(t)$ with 
	\[\ord_z(f(t)/g(t))\geq -(\alpha_1+q_\tau\alpha_2)+\ord_z(W)+\ord_z\bigl(\sigma_\tau(W^{-1})\bigr)\]
	such that $\Phi_{ij}'-f(t)/g(t)\in (K^{\circ\circ})^{m+v_K(W)+v_K(W^{-1})}((t-z))$. For $z\in\{0,\infty\}$, we have $$\ord_z\bigl(\sigma_\tau(W^{-1})\bigr)=q_\tau\cdot\ord_z(W^{-1}).$$
\end{corollary}
\begin{proof}
	This is from \cref{thm:degreeoffrobeniusstructuresatapoint} and the fact that
	\[\Phi'=W\Phi\sigma_\tau(W)^{-1}. \qedhere\]
\end{proof}

\section{Good models}

\label{section:good models}

We use the integral-basis framework of \cite[Section~2, Assumption~1]{tuitman2017counting}, stating it over $K^\circ$. 

\subsection{Plane models $(Q,W^0, W^\infty)$}

\label{subsection:plane models}

We fix a smooth projective geometrically irreducible curve $X$ of genus $g$ over $K$. We also fix a plane model $Q(x,y)=0$ of $X$ satisfying the following properties:
\begin{itemize}
	\item  $Q(x,y)\in K^\circ[x,y]$ is a polynomial that is irreducible and monic in $y$;
	\item the reduction $\widetilde{Q}(x,y)\in \widetilde{K}[x,y]$ is irreducible (not necessarily smooth).
\end{itemize} 
Set $d_x\coloneq\deg_yQ$ and $d_y\coloneq\deg_xQ$.
We denote by $\Delta(x)\in K^\circ[x]$ the discriminant of $Q$ with respect to $y$, and put
\[
r(x)\coloneq
\frac{\Delta(x)}{\gcd(\Delta(x),\Delta'(x))},
\]
where the gcd is taken to be monic. Thus $r$ is the squarefree polynomial with the same zeros and leading coefficient as $\Delta$. We also put $r_{\mathrm{mon}}\coloneq r/\operatorname{lc}(r),$ the monic polynomial with the same zeros as $r$. 
For further discussion, let $m\in \N$ be the least positive integer such that there is a polynomial $g(x)\in {K}[x]$ that satisfies $r(x)^m=g(x)\Delta(x)$. 

We write $\boldsymbol{b}^0=\begin{bmatrix}
	b_0^0&
	b_1^0&
	\dots&
	b_{d_x-1}^0
\end{bmatrix}$ and $\boldsymbol{b}^\infty=\begin{bmatrix}
b_0^\infty&
b_1^\infty&
\dots&
b_{d_x-1}^\infty
\end{bmatrix}$ as row vectors.

Furthermore, we fix an integral basis $(b_0^0, \dots, b_{d_x-1}^0)$ of $K(X)$ over $K[x]$, and an integral basis $(b_0^\infty, \dots, b_{d_x-1}^\infty)$ of $K(X)$ over $K[x^{-1}]$. Set $W^0 = (W_{i,j}^0)_{1\leq i,j\leq d_x}\in \GL_{d_x}(K[x,1/r])$ and $W^\infty = (W_{i,j}^\infty)_{1\leq i,j\leq d_x}\in \GL_{d_x}(K[x,x^{-1},1/r])$ such that
\[b_j^0=\sum_{i=0}^{d_x-1}W_{i+1,j+1}^0y^i \ \text{ and  } \  b_j^\infty=\sum_{i=0}^{d_x-1}W_{i+1,j+1}^\infty y^i\]
for all $0 \leq j\leq d_x-1$, i.e.
\[\begin{bmatrix}
	b_0^0&
	b_1^0&
	\dots&
	b_{d_x-1}^0
\end{bmatrix} =\begin{bmatrix}
	1&
	y&
	\dots&
	y^{d_x-1}
\end{bmatrix} W^0\]
and \[\begin{bmatrix}
	b_0^\infty&
	b_1^\infty&
	\dots&
	b_{d_x-1}^\infty
\end{bmatrix} = \begin{bmatrix}
	1&
	y&
	\dots&
	y^{d_x-1}
\end{bmatrix}W^\infty.\] 
We also set $W= (W^0)^{-1}W^\infty\in \GL_{d_x}(K[x,x^{-1},1/r])$ (Here notice that $K[x,1/r,y]/(Q(x,y))$ is smooth over $K[x,1/r]$, so it is the integral closure of $K[x,1/r]$ in $K(X)$). Notice that $b^0_i$ only has poles above $x=\infty$ and $b^\infty_j$ only has poles above $x=0$, so $\ord_0(W)<0$ and $\ord_\infty(W^{-1})<0$. 

We have a covering $\P_K^1=\Spec(K[x])\cup \Spec(K[x^{-1}])$. Via the morphism $x\colon X\rightarrow \P_K^1$, we have $x\colon X^0\rightarrow \Spec(K[x])$, $x\colon X^\infty\rightarrow \Spec(K[x^{-1}])$, where $X^0 = \Spec(A^0)$, $X^\infty=\Spec(A^\infty)$ and $A^0= 
\bigoplus_{i=0}^{d_x-1}K[x]b_i^{0}$, $A^\infty= \bigoplus_{i=0}^{d_x-1}K[x^{-1}]b_i^{\infty}$. 

For any $P\in X$, we denote by $e_P$ the ramification index of $P$ via $x: X\rightarrow \P_K^1$.

In our setting, only $X$, $Q$, and integral bases (or matrices $W^0, W^\infty$) are fixed, the others are determined by them.

\begin{lemma}
	\label[plemma]{lemma:derivativeatyisinvertible}
We have $r,g\in K^\circ[x]$, and the quotient
	\[
	\Delta/\frac{\partial Q}{\partial y}
	\]
	exists in $K^\circ[x,y]/(Q)$.
	In particular, $\Delta$ is a unit in
	$K^\circ[x,1/r]$, and $\partial Q/\partial y$ is
	a unit in $K^\circ[x,1/r,y]/(Q)$.
\end{lemma}

\begin{proof}
	If $\Delta$ is constant, then $r=\Delta$, $m=1$ and $g=1$, so $r,g\in K^\circ[x]$.
	Otherwise, write
	\[\Delta=c\prod_i(x-a_i)^{e_i}, \qquad c=\operatorname{lc}(\Delta), \]
	with $e_i\in\N_{\geq 1}$ and distinct $a_i\in\overline K$. Since the gcd is monic, we have
	\[
	r=c\prod_i(x-a_i),\qquad
	m=\max_i e_i\geq 1,\qquad
	g=c^{m-1}\prod_i(x-a_i)^{m-e_i}.
	\]
	Extend $v_K$ to $\overline K[x]$ by taking the minimum of the valuations of the coefficients. This valuation is additive by Gauss's lemma. Put
	\[
	\mu_i=v_K(x-a_i)=\min\{0,v_K(a_i)\}\leq0.
	\]
	Since $\Delta\in K^\circ[x]$, we have
	\[
	v_K(c)+\sum_i e_i\mu_i=v_K(\Delta)\geq0.
	\]
	Consequently,
	\[
	v_K(r)=v_K(c)+\sum_i\mu_i
	\geq v_K(\Delta)\geq0,
	\]
	\begin{align*}
		v_K(g)
		&=(m-1)v_K(c)+\sum_i(m-e_i)\mu_i\geq m\sum_i(1-e_i)\mu_i \geq0.
	\end{align*}
	Thus $r,g\in K^\circ[x]$.
	
	Since $Q$ is monic in $y$, the resultant identity
	gives $a,b\in K^\circ[x,y]$ such that
	\[
	aQ+b\frac{\partial Q}{\partial y}=\Delta.
	\]
	Hence $b\, \frac{\partial Q}{\partial y}=\Delta$ in
	$K^\circ[x,y]/(Q)$, proving the quotient assertion.
	
	Finally, $\Delta g=r^m$ gives $\Delta^{-1}=\frac{g}{r^m}\in K^\circ[x,1/r],$
	and the same resultant identity gives
	\[
	\left(\frac{\partial Q}{\partial y}\right)^{-1}
	=\frac{bg}{r^m}
	\in K^\circ[x,1/r,y]/(Q).\]
\end{proof}
\begin{remark}
	In general, given two polynomials $f, g \in R[x]$ over a ring $R$, the resultant $\mathrm{Res}(f,g)\in (f,g)\subset R[x]$.
\end{remark}
\begin{lemma}
	The morphism
	\[
	x\colon
	\mathcal U\coloneq\Spec(K^\circ[x,1/r,y]/(Q))
	\longrightarrow
	\mathcal V\coloneq\Spec(K^\circ[x,1/r])
	\]
	is \'etale, and
	$\mathcal U_K\subset X$,
	$\mathcal V_K\subset\P_K^1$.
\end{lemma}

\begin{proof}
	We claim that $K^\circ[x,y]/(Q)$ is finite free	over $K^\circ[x]$ of rank $d_x$. Indeed, by assumption, $y$ is integral over $K^\circ[x]$ of degree $d_x$ and $Q$ is the minimal polynomial, then $1, y, \dots, y^{d_x-1}$ form a basis of $K^\circ[x,y]/(Q)$.
	Thus $\mathcal U\to\mathcal V$ is finite and flat.
	
	By \cref{lemma:derivativeatyisinvertible}, $\partial Q/\partial y$ is a unit in $K^\circ[x,1/r,y]/(Q)$.
	The Jacobian criterion therefore shows that $\mathcal U\to\mathcal V$ is smooth. Since it is finite, it is \'etale.
	
	Finally, $\mathcal V_K$ is an open subset of $\P_K^1$. Since $\mathcal U_K$ is normal and finite over $\mathcal V_K$, with function field $K(X)$, it identifies with the normalization of $\mathcal V_K$ in $K(X)$, namely $x^{-1}(\mathcal V_K)\subset X$.
\end{proof}

\begin{example}
	Let $X$ be a hyperelliptic curve, then we can take $Q(x,y)=y^2-f(x)\in K^\circ[x,y]$ for $X$ such that $f$ has no multiple roots, then $\Delta(x)=4f(x)$, and $r_{\mathrm{mon}}(x)=f(x)/\operatorname{lc}(f)$.
\end{example}

\subsection{Good reduction of plane models}

\begin{definition}[Good-reduction assumptions] 	\label[pdef]{def:triplehasgoodreduction}
	We say that the triple $(Q, W^0, W^\infty)$ has \emph{good reduction} if the conditions below are satisfied.
	\begin{enumerate1}
		\item \label{good reduction: roots of r} The leading coefficient and discriminant of $r(x)$ are contained in $(K^\circ)^\times$.

		\item \label{good reduction: basis of reduction} Let $\widetilde{K}(x,y)$ be the field of fractions of $\widetilde{K}[x,y]/(\widetilde{Q})$, then:
		\begin{enumeratea}
			\item  The reduction modulo $K^{\circ\circ}$ of $(b_0^0,\dots, b_{d_x-1}^0)$ in $\widetilde{K}(x,y)$ is an integral basis of $\widetilde{K}(x,y)$ over $\widetilde{K}[x]$.
			\item The reduction modulo $K^{\circ\circ}$ of $(b_0^\infty,\dots, b_{d_x-1}^\infty)$ in $\widetilde{K}(x,y)$ is an integral basis of $\widetilde{K}(x,y)$ over $\widetilde{K}[x^{-1}]$.
		\end{enumeratea}
		\item \label{good reduction: W are integral} $W^0 \in \GL_{d_x}(K^\circ[x,1/r])$ and $W^\infty\in \GL_{d_x}(K^\circ[x,x^{-1},1/r])$.
		\item \label{good reduction: b are basis of model} Denote
		\begin{align*}
			\mathcal{R}^0&\coloneq K^\circ[x]b_0^0+\cdots+K^\circ[x]b_{d_x-1}^0\subset \sum_{i=0}^{d_x-1}K[x]b_i^0,\\
			\mathcal{R}^\infty&\coloneq K^\circ[x^{-1}]b_0^\infty+\cdots+K^\circ[x^{-1}]b_{d_x-1}^\infty\subset \sum_{i=0}^{d_x-1}K[x^{-1}]b_i^\infty.
		\end{align*}
		Then $\mathcal{R}^0, \mathcal{R}^\infty$ are rings and the discriminants of the finite $K^\circ$-algebras $(\mathcal{R}^0/(r(x)))_{\mathrm{red}}$ and $(\mathcal{R}^\infty/(x^{-1}))_{\mathrm{red}}$ are contained in $(K^\circ)^\times$.
	\end{enumerate1}
\end{definition}
\begin{remark} 
	We have the following observation:
	\begin{itemize}
		\item The property~\ref{good reduction: roots of r} is equivalent to the condition that $\deg\widetilde r=\deg r$ and the image $\widetilde{r}(x)$ of $r(x)$ in $\widetilde{K}[x]$ has distinct roots. This fact implies that the roots of $r(x)$ live in some unramified extension of $K$.
		\item The property~\ref{good reduction: W are integral} implies that $b^0_i \in K^\circ[x, 1/r, y]$, $b^\infty_i\in K^\circ[x,x^{-1},1/r, y]$, and \ref{good reduction: b are basis of model} implies that $b_i^0$ (resp. $b_i^\infty$) is integral over $K^{\circ}[x]$ (resp. $K^\circ[x^{-1}]$).
	\end{itemize}
\end{remark}
\begin{remark}
Note that if $X$ is over a number field $F$, then for any triple $(Q,W^0, W^\infty)$ of $X$ over $F$, we can always  find a prime $\mathfrak{p}$ such that $(Q,W^0, W^\infty)$ has good reduction over $F_{\mathfrak{p}}$. Indeed, each condition holds outside finitely many primes.
\end{remark}

Geometrically, the fact that $(Q, W^0, W^\infty)$ has good reduction implies that the \'etale morphism $x: \Spec(K^\circ[x,1/r,y]/(Q))\rightarrow \Spec(K^\circ[x,1/r])$ has a good compactification. 

\begin{proposition}	\label[pprop]{prop:smoothmodelassociatedtotriple}
	Let $(Q, W^0, W^\infty)$ be a triple with good reduction for $X$. Set $\mathcal{U}\coloneq \Spec(K^\circ[x,1/r, y]/(Q))$, $\mathcal{V}\coloneq \Spec(K^\circ[x,1/r])$. Then the following statements hold:
	\begin{enumerate1}
		\item \label{sm model of P^1} There is a smooth relative (Weil) divisor $\mathcal{Z}_{\P^1}$ on $\P_{K^\circ}^1$ such that $\mathcal{V}=\P^1_{K^\circ}\setminus\mathcal{Z}_{\P^1}$.
		\item \label{sm model of X} There is a smooth proper curve $\mathcal{X}$ over $K^\circ$ and a smooth relative (Weil) divisor $\mathcal{Z}_{\mathcal{X}}$ on $\mathcal{X}$ such that $\mathcal{U}=\mathcal{X}\setminus\mathcal{Z}_{\mathcal{X}}$.
	\end{enumerate1}
\end{proposition}

\begin{proof}
	We adapt the proof of \cite[Proposition 2.3]{tuitman2017counting} to the complete discrete valuation ring $K^\circ$.

	Fix the notation in \cref{def:triplehasgoodreduction}.
	We take $\mathcal X$ to be the normalization of $\P^1_{K^\circ}$ in $K(X)$.
	
	Since $(b_0^0,\dots,b_{d_x-1}^0)$ is linearly independent over $K[x]$, the ring $\mathcal R^0$ is finite free over $K^\circ[x]$, hence flat over $K^\circ$.
	By our choice of integral basis and
	\cref{def:triplehasgoodreduction} \ref{good reduction: basis of reduction}, the generic and special fibers of $\Spec(\mathcal R^0)\to\Spec(K^\circ)$ are normal curves, hence regular.
	Since $K$ and $\widetilde K$ are perfect, these fibers are smooth. Thus $\Spec(\mathcal R^0)$ is smooth over
	$K^\circ$. The same holds for $\Spec(\mathcal R^\infty)$.
	These smooth schemes are normal, so their coordinate rings, being finite over $K^\circ[x]$ and
	$K^\circ[x^{-1}]$ with fraction field $K(X)$, are the corresponding integral closures.
	Hence they are the two affine charts of $\mathcal X$.
	Consequently, $\mathcal X$ is smooth over $K^\circ$ and finite over $\P^1_{K^\circ}$, hence proper.
	
	The complement of $\mathcal V$ in $\P^1_{K^\circ}$ is the union of the zero locus of $r(x)$ and $\infty$.
	By \cref{def:triplehasgoodreduction} \ref{good reduction: roots of r}, these are disjoint and finite \'etale over $K^\circ$.
	Their union is therefore a smooth relative divisor $\mathcal Z_{\P^1}$. This proves \ref{sm model of P^1}.
	
	By \cref{def:triplehasgoodreduction} \ref{good reduction: W are integral}, $\mathcal R^0[1/r]=K^\circ[x,1/r,y]/(Q)$.
	Thus the reduced complement of $\mathcal U$ in $\mathcal X$ is the disjoint union of the reduced zero locus of $r(x)$ and the reduced fiber above $\infty$. Their coordinate rings are
	\[(\mathcal R^0/(r(x)))_{\mathrm{red}} \quad\text{and}\quad (\mathcal R^\infty/(x^{-1}))_{\mathrm{red}}.\]
	These are finite free over $K^\circ$: the unreduced algebras are finite free since the leading coefficient of $r$ is a unit, and their reductions remain torsion-free. Their discriminants are units by \cref{def:triplehasgoodreduction} \ref{good reduction: b are basis of model}, so they are finite \'etale over $K^\circ$.
	Their union is therefore a smooth relative Weil divisor $\mathcal Z_{\mathcal X}$ with $\mathcal U=\mathcal X\setminus\mathcal Z_{\mathcal X}$. This proves \ref{sm model of X}.
\end{proof}

\begin{definition} \label[pdef]{def:smooth model of X}
	Let $(Q, W^0, W^\infty)$ be a triple with good reduction for $X$. We call the model $\mathcal{X}\rightarrow \Spec(K^\circ)$ of $X$ in \cref{prop:smoothmodelassociatedtotriple}~\ref{sm model of X} the \emph{smooth model associated to $(Q,W^0, W^\infty)$}. 
\end{definition}

In the rest of this paper, we fix a triple $(Q, W^0, W^\infty)$ with good reduction for $X$, and denote by $\mathcal{X}$ the smooth model associated to $(Q,W^0, W^\infty)$. Via the model $\mathcal{X}$, we have a reduction map
\[\Sp\colon X^\an\rightarrow \mathcal{X}_s,\]
and residue discs in $X^\an$. By a residue disc, we mean a subset of $X^\an$ of the form $\Sp^{-1}(\overline{x})$ for some point $\overline{x}\in \mathcal{X}_s$ which is not necessarily $\widetilde{K}$-rational. We also denote $\mathcal{U}\coloneq\Spec(K^\circ[x,1/r, y]/(Q))$, $\mathcal{Z}\coloneq  \mathcal{X}\setminus \mathcal{U}$, $U\coloneq\mathcal{U}_K$ and $Z\coloneq\mathcal{Z}_K$. 

\subsection{Bad and good points}

\begin{definition}
	Let $(Q, W^0, W^\infty)$ be a triple with good reduction for $X$. A residue disc  $D\subset X^\an$
	\begin{itemize}
		\item is said to be \emph{infinite} if it is the preimage of an infinite point in $\mathcal{X}_s$;
		\item is said to be \emph{finitely bad} if it is the preimage of a point in $\mathcal{X}_s$ with $x$-coordinate a zero of $\widetilde{r}\in \widetilde{K}[x]$;
		\item is said to be \emph{bad} if it is infinite or finitely bad;
		\item is said to be \emph{good} if it is not bad;
		\item is said to be \emph{$K$-rational} if it contains a $K$-rational point.
	\end{itemize}
\end{definition}

\begin{definition}\label{classification of points def}
	Let $(Q, W^0, W^\infty)$ be a triple with good reduction for $X$. A rigid point $P\in X^\an$ (i.e. the residue field of $P$ is finite over $K$) 
	\begin{itemize}
		\item is said to be \emph{very infinite} if  
		its $x$-coordinate is $\infty$;
		\item is said to be \emph{bad} (resp. \emph{finitely bad}, resp. \emph{infinite}) if it is in a bad disc (resp. a finitely bad disc , resp. an infinite disc); 
		\item is said to be \emph{very bad} if it is very infinite or its $x$-coordinate is a zero of $r(x)$;
\item is said to be \emph{suitable bad} if either $P$ is finitely bad and
\[0<v_K(r(x(P)))<\frac1q,\]
or $P$ is infinite but not very infinite;
		\item is said to be \emph{good} if it is in a good disc.
	\end{itemize} 
\end{definition}


\begin{remark} \label{remark:defofbadresiduedisc}
 If a rigid point $P$ is very bad, then one of three cases holds:
		\begin{itemize}
			\item $x(P) = \infty$;
			\item the fiber of $X^\an$ above $x(P)$ contains a ramification point (the point satisfies $Q(P)=0$ and $r=0$);
			\item the fiber of $X^\an$ above $x(P)$ contains a point mapping to a singularity of the plane model $Q(x,y)=0$.
		\end{itemize}
\end{remark}
\begin{remark}\label{rem:unique-very-bad-point}
Since $(Q, W^0,W^\infty)$ has good reduction, by the Hensel lemma, we have that a bad residue disc contains only one very bad point and this point is defined over an unramified extension of $K$. This implies that the very bad point corresponding to a $K$-rational, bad residue disc is always $K$-rational. 

To show the implication, let $D$ be a bad residue disc containing a $K$-rational point $P$ and a very bad point $P'$. If $P'$ is not $K$-rational, we assume that $P' \in X(K')$ for some finite Galois extension $K'/K$. Then there is $\alpha\in \Gal(K'/K)$ such that $\alpha(P')\not=P'$. Notice that $\alpha(P')$ is also a very bad point, $\alpha(D)=\Sp^{-1}(\Sp(\alpha(P')))$ and $P\in \alpha(D)\cap D$. On the other hand, $\alpha(D)\cap D=\emptyset$ since every bad residue disc contains only one very bad point, hence we have a contradiction. 
\end{remark}

\section{Computing a Frobenius lift}
\label{section:computing a frobenius lift}
We retain the notation fixed above and construct a $\tau$-semilinear
lift $F_\tau^*$ with $F_\tau^*(x)=x^{q_\tau}$ and $F_\tau^*|_K=\tau$. Its $n_\tau$-fold iterate will be denoted by $F_q^*$; it is $K$-linear and sends $x$ to $x^q$.
We also fix a smooth projective geometrically irreducible curve $X$ of genus $g$ over $K$ with a triple $(Q, W^0, W^\infty)$ having good reduction.  The notations $\Delta(x), r(x), g(x), m, b^0_i, b^\infty_i$ are given in \cref{subsection:plane models}. We denote $\mathcal{U}= \Spec(K^\circ[x,1/r, y]/(Q))$, $\mathcal{V}= \Spec(K^\circ[x,1/r])$ and $U=\mathcal{U}_K, V=\mathcal{V}_K$ as \cref{prop:smoothmodelassociatedtotriple}.
We also set
\begin{align*}
	\mathcal{S}&=K^\circ\langle x, 1/r\rangle, & & \mathcal{S}^\dagger=K^\circ\langle x, 1/r\rangle^\dagger,\\
	\mathcal{R}&=K^\circ\langle x, 1/r,y\rangle/(Q),&& \mathcal{R}^\dagger=K^\circ\langle x, 1/r,y\rangle^\dagger/(Q).
\end{align*}
Here $\mathcal S$ is the $\pi$-adic completion of $K^\circ[x,1/r]$, and $\mathcal S^\dagger$ is its weak completion in the sense of \cite[Definition 1.1]{monsky1968formal}. Similarly, $\mathcal R$ and $\mathcal R^\dagger$ are the $\pi$-adic and weak completions of $K^\circ[x,1/r,y]/(Q)$.

\begin{art}
Modulo $K^{\circ\circ}$, let
$\widetilde{\mathcal R}\coloneq \widetilde K[x,1/\widetilde r,y]/(\widetilde Q)$.
For $a\in\widetilde K$, the rule
\[
 a\longmapsto a^{q_\tau},\qquad
 x\longmapsto x^{q_\tau},\qquad
 y\longmapsto y^{q_\tau}
\]
defines a ring endomorphism $\varphi\colon\widetilde{\mathcal R}\to
\widetilde{\mathcal R}$, the absolute $q_\tau$-power Frobenius.
Indeed, $\varphi(\widetilde Q)=\widetilde Q^{q_\tau}=0$ in
$\widetilde{\mathcal R}$.  This map is $\widetilde\tau$-semilinear:
$\varphi(af)=\widetilde\tau(a)\varphi(f)$.  Unless
$\widetilde\tau=\mathrm{id}_{\widetilde K}$, it is not a
$\widetilde K$-algebra homomorphism.  

Since \(\widetilde{\mathcal R}\) is smooth over \(\widetilde K\), the classical lifting result for
Monsky--Washnitzer dagger algebras gives a \(\tau\)-semilinear ring
endomorphism
\[
F_\tau^*:\mathcal R^\dagger\longrightarrow\mathcal R^\dagger
\]
whose reduction is the absolute \(q_\tau\)-power Frobenius $\varphi\colon\widetilde{\mathcal R}\to
\widetilde{\mathcal R}$, see \cite[Definition~3.2.3]{kedlaya2006finiteness}.
For the computations below, we construct explicitly such a lift
satisfying \(F_\tau^*(x)=x^{q_\tau}\), using a Newton--Hensel iteration, compare \cite[Theorem~2.6]{tuitman2017counting}.
\end{art}

Recall the following form of Hensel's lemma, which also applies to weak completions.

\begin{lemma}
	\label[plemma]{lemma:explicitHensellemma}
	Let $A$ be a ring and $I\subset A$ an ideal such that $A$ is complete (hence separated) for the $I$-adic topology.
	Denote reduction modulo $I$ by a bar.
	Let $f(x)\in A[x]$ and $\alpha\in A$ satisfy $\overline f(\overline\alpha)=0, \overline f'(\overline\alpha)\in(A/I)^\times.$
	Set $\alpha_0=\alpha$ and, for $n\geq0$, choose
	$t_n\in A$ such that
	\[
	t_n\equiv f'(\alpha_n)^{-1}\pmod{I^{2^n}},
	\qquad
	\alpha_{n+1}=\alpha_n-f(\alpha_n)t_n.
	\]
	Then
	\[
	f(\alpha_n)\in I^{2^n},
	\qquad
	\alpha_{n+1}-\alpha_n\in I^{2^n}.
	\]
	The limit $\widetilde\alpha=\lim_{n\to\infty}\alpha_n$
	exists in $A$ and is the unique root of $f$
	congruent to $\alpha$ modulo $I$.
	
	The same conclusion holds if $A$ is instead the weak
	completion of a finitely generated $K^\circ$-algebra
	and $I=\pi A$: the limit taken in the $\pi$-adic
	completion of $A$ belongs to $A$.
\end{lemma}

\begin{proof}
	The congruences follow by induction using the Taylor
	expansion of $f$. In the complete case, convergence
	and uniqueness follow from Hensel's lemma.
	
	In the weak-completion case, $(A,\pi A)$ is a
	henselian pair by
	\cite[Theorem~3]{etesse2002relevement}.
	Hence the simple root modulo $\pi$ lifts uniquely
	to a root in $A$. Its image in the $\pi$-adic
	completion coincides, by uniqueness, with the limit
	of the above iteration.
\end{proof}

\begin{proposition}\label{proposition:expressionoffrobeniuslift}
	There is a $\tau$-semilinear Frobenius lift
	$F_\tau^*\colon \mathcal{R}^\dagger\rightarrow \mathcal{R}^\dagger$ for which 
	\begin{align*}
		F_\tau^*(x)& = x^{q_\tau},\\
		F_\tau^*\left(\frac{1}{r}\right)& =
		\frac{1}{r^{\tau}(x^{q_\tau})},\\
		F_\tau^*(y)& = \beta,
	\end{align*}
	where $\beta\in \mathcal{R}^\dagger$ is the root, congruent to
	$y^{q_\tau}$ modulo $K^{\circ\circ}$, of
	$Q^\tau(x^{q_\tau},\beta)=0$.
\end{proposition}

\begin{proof}
	This is the Newton construction of \cite[Theorem~2.6]{tuitman2017counting}, with coefficient action $\tau$, Frobenius exponent $q_\tau$, and congruences measured in powers of $\pi$. Throughout this proof, we view $\mathcal S^\dagger$ as a subring of $\mathcal R^\dagger$. Unless another ring is specified, equalities are interpreted in $\mathcal R^\dagger$. For congruences, the modulus $\pi^j\mathcal S^\dagger$ or $\pi^j\mathcal R^\dagger$ specifies the ambient ring. Polynomial identities in $K^\circ[x,y]$ will be indicated explicitly.
	
	We will construct the solutions $\alpha, \beta\in \mathcal{R}^\dagger$ of 
	\begin{align*}
		r^\tau(x^{q_\tau})\cdot\alpha-1=0,\\ 
		Q^\tau(x^{q_\tau},\beta)=0,
	\end{align*} 
    respectively. This allows us to define a $\tau$-semilinear ring homomorphism $F_\tau^*\colon \mathcal{R}^\dagger\rightarrow \mathcal{R}^\dagger$  by
	\[F_\tau^*(x) = x^{q_\tau}, \ \ F_\tau^*(1/r) = \alpha, \ \ F_\tau^*(y) = \beta\]
	which is a Frobenius lift. In this proof, the equalities are in $\mathcal R^\dagger$.
    To construct $\alpha$ and $\beta$, we fix a lift $s(x,y)\in K^\circ[x,y]$ of $\Delta/\frac{\partial Q}{\partial y}\in K^\circ[x,y]/(Q)$ which exists by \cref{lemma:derivativeatyisinvertible}, and set 
\[\begin{cases}
    \alpha_0 = 1/r(x)^{q_\tau},\\
    \beta_0=y^{q_\tau},
\end{cases}\]
and 
\[\begin{cases}
    \alpha_{n} = \alpha_{n-1}-(r^\tau(x^{q_\tau})\alpha_{n-1}-1)\alpha_{n-1} = \alpha_{n-1}(2-r^\tau(x^{q_\tau})\alpha_{n-1}),\\
    \beta_{n}=\beta_{n-1}-Q^\tau(x^{q_\tau},\beta_{n-1})s^\tau(x^{q_\tau},\beta_{n-1})g^\tau(x^{q_\tau})\alpha_{n-1}^m.
\end{cases}\]
for $n\in\N_{\geq 1}$.
We will show the following congruences:
\begin{equation}\label{eq:frobenius-hensel-congruences}
	\left\{
	\begin{array}{@{}l@{}}
		Q^\tau(x^{q_\tau},\beta_n)
		\equiv 0 \pmod{\pi^{2^n}\mathcal R^\dagger},\\[3pt]
		\alpha_n r^\tau(x^{q_\tau})
		\equiv 1 \pmod{\pi^{2^n}\mathcal S^\dagger},\\[3pt]
		s^\tau(x^{q_\tau},\beta_n)g^\tau(x^{q_\tau})\alpha_n^m
		\dfrac{\partial Q^\tau}{\partial y}(x^{q_\tau},\beta_n)
		\equiv 1 \pmod{\pi^{2^n}\mathcal R^\dagger}
	\end{array}
	\right.
\end{equation}
by induction on $n$, where $m\in \N$ such that $r(x)^m=g(x)\Delta(x)$, see \cref{subsection:plane models}.
\begin{itemize}
    \item When $n=0$, we have that 
    \[Q^\tau(x^{q_\tau},\beta_{0})=Q^\tau(x^{q_\tau},y^{q_\tau}) \equiv Q(x,y)^{q_\tau} \equiv 0 \ \ \mod \pi\mathcal{R}^\dagger,\]
    \[\alpha_{0}r^{\tau}(x^{q_\tau}) = \frac{1}{r(x)^{q_\tau}}\cdot r^{\tau}(x^{q_\tau}) \equiv \frac{1}{r(x)^{q_\tau}}\cdot r(x)^{q_\tau}
	\equiv 1 \ \ \mod \pi\mathcal{S}^\dagger,\]
    \begin{align*}
     s^\tau(x^{q_\tau},\beta_{0})g^\tau(x^{q_\tau})\alpha_{0}^m\cdot \frac{\partial Q^\tau}{\partial y}(x^{q_\tau},\beta_{0}) & =s^\tau(x^{q_\tau},y^{q_\tau})g^\tau(x^{q_\tau})\alpha_0^m\cdot \frac{\partial Q^\tau}{\partial y}(x^{q_\tau},y^{q_\tau})\\
     & \equiv \Delta^\tau(x^{q_\tau})g^\tau(x^{q_\tau})\alpha_0^m\\
&=(r^\tau(x^{q_\tau}))^m\alpha_0^m\\
&\equiv 1 \ \ \mod \pi\mathcal{R}^\dagger.  
    \end{align*}
    So the congruences hold for $n=0$.
    \item When $n\geq 1$, we assume that \eqref{eq:frobenius-hensel-congruences} holds for $n-1$. Then 
    \[\beta_n-\beta_{n-1}\in \pi^{2^{n-1}}\mathcal R^\dagger\] 
    and there are $a\in \mathcal{S}^\dagger, b, b'\in \mathcal{R}^\dagger$ such that $$Q^\tau(x^{q_\tau},\beta_{n-1}) = \pi^{2^{n-1}}b\in \pi^{2^{n-1}}\mathcal{R}^\dagger,$$ $$\alpha_{n-1}r^\tau(x^{q_\tau})=1+\pi^{2^{n-1}}a,$$ $$s^\tau(x^{q_\tau},\beta_{n-1})g^\tau(x^{q_\tau})\alpha_{n-1}^m\cdot \frac{\partial Q^\tau}{\partial y}(x^{q_\tau},\beta_{n-1})=1+\pi^{2^{n-1}}b'.$$ By our choice of $s(x,y)$, we write $s(x,y)\frac{\partial Q}{\partial y}(x,y)g(x) = r(x)^m+Q(x,y)h(x,y)$ for some $h\in K^\circ[x,y]$. Then
    \begin{align*}
	Q^\tau(x^{q_\tau},\beta_{n}) = & Q^\tau(x^{q_\tau},\beta_{n-1})+\frac{\partial Q^\tau}{\partial y}(x^{q_\tau},\beta_{n-1})(\beta_n-\beta_{n-1}) + \sum_{i=2}^{d_x}\frac{1}{i!}\frac{\partial^i Q^\tau}{\partial y^i}(x^{q_\tau},\beta_{n-1})(\beta_n-\beta_{n-1})^i\\
    = & Q^\tau(x^{q_\tau},\beta_{n-1})-\frac{\partial Q^\tau}{\partial y}(x^{q_\tau},\beta_{n-1})Q^\tau(x^{q_\tau},\beta_{n-1})s^\tau(x^{q_\tau},\beta_{n-1})g^\tau(x^{q_\tau})\alpha_{n-1}^m\ + \pi^{2^{n}}c\\
    = & Q^\tau(x^{q_\tau},\beta_{n-1})-Q^\tau(x^{q_\tau},\beta_{n-1})(1+\pi^{2^{n-1}}b')+ \pi^{2^{n}}c\\
    =& -\pi^{2^n}bb'+\pi^{2^n}c\\
	\equiv& 0 \ \  \mod \pi^{2^{n}}\mathcal{R}^\dagger,
\end{align*}
where $c$ is some element in $\mathcal{R}^\dagger$ (notice that $\frac{1}{i!}\frac{\partial^i Q^\tau}{\partial y^i}(x^{q_\tau},\beta_{n-1})\in \mathcal{R}^\dagger$). We also have
\begin{align*}
	\alpha_{n}r^{\tau}(x^{q_\tau})= & r^\tau(x^{q_\tau})\alpha_{n-1}(2-r^\tau(x^{q_\tau})\alpha_{n-1})\\
	= & (1+\pi^{2^{n-1}}a)(1-\pi^{2^{n-1}}a)\\
	= & 1-\pi^{2^{n}}a^2\\
	\equiv & 1 \ \ \mod \pi^{2^{n}}\mathcal{S}^\dagger,
\end{align*}
\begin{align*}
s^\tau(x^{q_\tau},\beta_{n})g^\tau(x^{q_\tau})\alpha_{n}^m\cdot \frac{\partial Q^\tau}{\partial y}(x^{q_\tau},\beta_{n})= & (r^\tau(x^{q_\tau})\alpha_{n})^m + Q^\tau(x^{q_\tau},\beta_{n})h^{\tau}(x^{q_\tau},\beta_{n})\alpha_{n}^m\\
\equiv & 1 \ \ \mod \pi^{2^{n}}\mathcal{R}^\dagger. 
\end{align*}
Hence the equalities hold for $n$.

\end{itemize}
Hence the congruences hold for arbitrary $n\in\N$. By \cref{lemma:explicitHensellemma}, we can take 
\[\alpha\coloneq \lim_{n\to\infty}\alpha_n\in\mathcal S^\dagger, \qquad \beta \coloneq \lim_{n\to\infty}\beta_n \in \mathcal R^\dagger\]
($t_n$ for $\alpha$ (resp. $\beta$) is $\alpha_{n}$ (resp. $s^{\tau}(x^{q_\tau},\beta_n)g^\tau(x^{q_\tau})\alpha_n^m$)). This completes the proof of the proposition. 
\end{proof}
\begin{remark}
Also note that $F_\tau^*(1/r)- \alpha_n\in \pi^{2^n}\mathcal{S}^\dagger$ and $F_\tau^*(y)-\beta_n\in \pi^{2^n}\mathcal{R}^\dagger$.
The iterate $F_q^*\coloneq(F_\tau^*)^{n_\tau}$ agrees with the direct $q$-power lift with coefficient action $\id_K$: both send $x$ to $x^q$ and $y$ to a root of $Q(x^q,Y)$ congruent to $y^q$ modulo $\pi$. 
This root is unique: by \cref{lemma:derivativeatyisinvertible} and $\Delta(x)g(x)=r(x)^m$ (see \cref{subsection:plane models}), there are $s(x,y), h(x,y)\in K^\circ[x,y]$ such that
\[s(x,y)\frac{\partial Q}{\partial y}(x,y)g(x)=r(x)^m+Q(x,y)h(x,y).\]
Reducing modulo \(\pi\) and then passing to $\widetilde{\mathcal R}=\mathcal R^\dagger/\pi\mathcal R^\dagger$ gives 
$$\widetilde s(x,y) \frac{\partial\widetilde Q}{\partial y}(x,y) \widetilde{g}(x) = \widetilde{r}(x)^m$$
which implies that $\frac{\partial\widetilde Q}{\partial y}(x,y)\in\widetilde{\mathcal R}^{\times}$ since $\widetilde{r}(x)\in \widetilde{\mathcal R}^{\times}$.
Then
\[
\frac{\partial\widetilde Q}{\partial y}(x^q,y^q)
=
\left(\frac{\partial\widetilde Q}{\partial y}(x,y)\right)^q
\in\widetilde{\mathcal R}^{\times}.
\]
Hence Hensel's lemma gives a unique lift of $y^q$ as a root
of $Q(x^q,Y)$ in the $\pi$-adic completion, and therefore at
most one such root in $\mathcal R^\dagger$.
\end{remark}

We can apply \cref{thm:degreeoffrobeniusstructuresatapoint} to estimate the Frobenius lift $F_\tau^*$ given in \cref{proposition:expressionoffrobeniuslift}. We need more information about the differential map $d\colon \OO_X\to \Omega_{X/K}^1$ on $X$.

The sheaf of meromorphic (or fractional) differential $1$-forms on $X$ is denoted by $\Omega_{K(X)/K}^1$.  Notice that $\Omega_{K(X)/K}^1(U)$ is independent of the choice of the open subset $U\subset X$, and
\[\Omega^1_{X/K}(U) = \{\omega\in \Omega_{K(X)/K}^1(U)\mid \ord_{z}(\omega)\geq0 \text{ for any $z\in U$}\}.\]
The sheaf $\Omega^1_{X/K}$ is invertible on $X$ since $X$ is smooth.

\begin{lemma} 	\label[plemma]{lemma:differentialformonopenisfree}
	We have $\Omega^1_{\mathcal{U}/K^\circ}= \OO_{\mathcal{U}}dx$.
\end{lemma}
\begin{proof}
	By \cref{lemma:derivativeatyisinvertible}, we know that in $K^\circ[x,1/r,y]/(Q)=K^\circ[x,y,z]/(Q,rz-1), \frac{\partial Q}{\partial y}$ is invertible since $\Delta$ is invertible. So $dy = -(\frac{\partial Q}{\partial x}/\frac{\partial Q}{\partial y})dx$, $dz = -(zr'/r)dx$, and this implies that $\Omega^1_{\mathcal{U}/K^\circ}(\mathcal{U}) = K^\circ[x,y,z]/(Q,rz-1) dx$. This completes the proof.
\end{proof}
\begin{remark}\label[prmk]{rmk:differential form of Rdagger}
    We denote by $\Omega^1_{\mathcal{R}^\dagger/K^\circ}$ the \emph{module of continuous differentials} of $\mathcal{R}^\dagger$, see \cite[\S 2.3]{kedlaya2006fourier}. As before, we have $\Omega^1_{\mathcal{R}^\dagger/K^\circ}=\mathcal{R}^\dagger dx$.
\end{remark}
The morphism $x\colon U\rightarrow V$ is smooth, so it induces \emph{Gauss-Manin connections} $\nabla$ on $x_*\OO_U$ over $V$. Moreover, the Frobenius lift $F_\tau^*\colon\mathcal{R}^\dagger\rightarrow \mathcal{R}^\dagger$ induces a $q_\tau$-power Frobenius structure on $(x_*\OO_U,\nabla)$ around a strict neighborhood of good disks, see \cref{q-frobenius structure def} for definition. In fact, the connection on $x_*\OO_U$ is induced by the differential map $$d\colon\OO_X(U)\rightarrow \Omega_{U/K}^1(U),$$
or
$$\frac{d}{dx}\colon \OO_X(U)\to \OO_X(U)$$ (by \cref{lemma:differentialformonopenisfree}) as modules over $\OO_{\P_K^1}(V) = K[x,1/r]$.

Recall the regularity of a connection $\nabla$ at a point in \cref{def:connectionregularatz}, and the residue matrix and the exponents defined in \cref{def:exponentofmatrix}. Next we show that the connection $(x_*\OO_U, \nabla)$ (without the Frobenius structure) satisfies the assumptions in \cref{thm:degreeoffrobeniusstructuresatapoint}.

\begin{proposition} 	\label[pprop]{prop:matrixofdifferential}
	Let $G^0\in M_{d_x}(K^\circ[x,1/r])$ and $G^\infty\in M_{d_x}(K^\circ[x,x^{-1}, 1/r])$ be the matrices such that
	\[db_j^0=\sum_{i=0}^{d_x-1} G^0_{i+1,j+1}b_i^0 dx, \ \ \ \  db_j^\infty=\sum_{i=0}^{d_x-1} G^\infty_{i+1,j+1}b_i^\infty dx,\]
	for all $0\leq j\leq d_x-1$. Let $x_0\in \P_K^1(\overline{K})$. The following statements hold:
	\begin{enumerate1}
		\item \label{regularity of G} The connection matrix $G^0dx$ (resp. $G^\infty dx$ at $x_0=\infty$) has at most a simple pole at $x_0$. 
		\item \label{exponents of G} For any $P\in X(\overline{K})$ with $x(P)=x_0$, set $e_P$ the ramification index of $x$ at $P$. Then the exponents of $G^0$ (resp. $G^\infty$ if $x_0=\infty$) at $x_0$ are 
		\[0, 1/e_P,  \dots, (e_P-1)/e_P\]
		for $P$ going through $x^{-1}(x_0)$. In particular, the exponents of $G^0$ (resp. $G^\infty$) at $x_0$ are elements of $\Q\cap K^\circ$ and are contained in $[0,1)$.
	\end{enumerate1}
\end{proposition}

\begin{proof}
	The regularity and exponent arguments follow \cite[Propositions~2.8 and~2.10]{tuitman2017counting}, we include the local calculation over $K$.

	We assume that $x_0\in \P^1_K(\overline{K})\setminus\{\infty\}$, the proof for $x_0=\infty$ is similar.
	
	\ref{regularity of G} By \cref{lemma:differentialformonopenisfree}, we know that $G^0$ is uniquely determined. Let $t=x-x_0$ and $H(t)dt=G^0(x)dx$. For any point $P\in X(\overline{K})$ lying over $x_0$, notice that $\ord_P(dt/t)=-1$ and $\ord_P(db_i^0)\geq 0$ for all $0\leq i\leq d_x-1$, then 
	\[\ord_P\left(\sum_{i=0}^{d_x-1}tH_{i+1,j+1}(t)b_i^0\right) = \ord_P(tdb_j^0)-\ord_P(dt)\geq 1.\]
	This implies that $\sum_{i=0}^{d_x-1}tH_{i+1,j+1}(t)b_i^0\in A^0_{(t)}$, the localization of $A^0=\OO_X(X^0)$ at $(t)\subset K[x]$ (Notice that $A^0_{(t)}=\bigcap_{x(P)=x_0}\OO_{X,P}$, where $\OO_{X,P}$ is the local ring at $P$, since $A^0$ is the integral closure of $K[x]$ in $K(X)$). Since $b_0^0, \dots, b_{d_x-1}^0$ is a basis of $(K[x,y]/Q)_{(t)}$ (resp. $\Frac(K[x,y]/Q)$) over $K[x]_{(t)}$ (resp. $K(x)$), we have that $tH_{i,j}\in K[x]_{(t)}$ for any $1\leq i,j\leq d_x$, i.e. $H(t)$ has at most simple pole at $x_0$. This proved \ref{regularity of G}.

		\ref{exponents of G}
		We work over $\overline K$ for the local calculation.
		The residue matrix $G_{-1}^0(x_0)$ represents the endomorphism
		\[D\coloneq(x-x_0)\frac{d}{dx}\colon \OO_X(X^0)/(x-x_0)\to \OO_X(X^0)/(x-x_0)\]
		with respect to the basis $b_0^0,\dots,b_{d_x-1}^0$. We have
$$\OO_X(X^0)/(x-x_0)\simeq \bigoplus_{P\in X^0, x(P)=x_0}\OO_{X,P}/(t_P^{e_P})$$
where $t_P$ is a uniformizer of $\OO_{X,P}$. Write $x-x_0=ut_P^{e_P}$ with $u\in\OO_{X,P}^\times$. Then 
\[D(t_P)=(x-x_0)\frac{dt_P}{dx} = \frac{t_P}{e_P} - \frac{(x-x_0)t_P}{e_Pu}\frac{du}{dx}.\]
Since $u$ is regular at $P$ and $\ord_P=e_P$, we have that $$\ord_P\left(\frac{du}{dt_P}\right)\geq 0 \qquad \ord_P\left(\frac{dx}{dt_P}\right) = e_P-1,$$
and consequently
\[
\ord_P\left(\frac{du}{dx}\right)
=
\ord_P\left(\frac{du/dt_P}{dx/dt_P}\right)
\geq0-(e_P-1)=1-e_P.
\]
Hence,	for every $n\geq0$,
		\begin{align*}
			D(t_P^n)&=nt_P^{n-1}D(t_P)\\
			&=\frac{n}{e_P}t_P^n-\frac{n}{e_Pu}t_P^n(x-x_0)\frac{du}{dx}\\
			&\equiv\frac{n}{e_P}t_P^n\pmod{t_P^{n+1}}.
		\end{align*}
		Thus, on each summand, the matrix of $D$ in the basis
		$1,t_P,\dots,t_P^{e_P-1}$ is lower triangular, with
		diagonal entries
		\[
		0,\frac1{e_P},\dots,\frac{e_P-1}{e_P}.
		\]
		These are therefore precisely the eigenvalues of
		$G_{-1}^0(x_0)$, counted with multiplicity.
		
		It remains to prove that they are $p$-adically integral.
		If $e_P>1$, then $x_0$ is a root of $r$, since
		$x\colon U\to V$ is \'etale.
		By \ref{regularity of G}, write $G^0=M^0/r$ with
		$M^0\in M_{d_x}(K^\circ[x])$. Then
		\[
		G_{-1}^0(x_0)=\frac{M^0(x_0)}{r'(x_0)}
		\]
		has integral entries, because $x_0$ is integral and $r'(x_0)$ is a $p$-adic unit by the good-reduction assumptions
		on $r$. Its eigenvalues are therefore $p$-adically integral. This completed the proof of \ref{exponents of G}.
\end{proof}
\begin{remark}
	Notice that if $x_0\in \overline{K}$ is not a root of $r(x)$, then $G^0$ does not have a simple pole at $x_0$, i.e. residue matrix of $G^0$ at $x_0$ is $0$. So for the residue matrix of $G^0$, it suffices to consider that at the roots of $r(x)$. Write $G^0=M^0/r$ with $M^0\in M_{d_x}(K^\circ[x])$ and let $x_0$ be a root of $r(x)$, then the residue of $G^0dx$ at $x_0$ is $M^0(x_0)/r'(x_0)$.
\end{remark}

\begin{corollary}\label[pcor]{cor:regular basis of x}
	Let $x_0$ be a root of $r(x)$ or $x_0=\infty$. Then for every fixed $k\in\N$, the basis \[b_0^0/r^k,\dots,b_{d_x-1}^0/r^k \qquad \text{(resp. $b_0^\infty/r^k,\dots,b_{d_x-1}^\infty/r^k$)}\] of $(x_*\OO_U,\nabla)$ is regular at every root $x_0$ of $r(x)$ (resp. at $x_0=\infty$) in the sense of \cref{def:connectionregularatz} with corresponding exponents of $\nabla$ contained in \[\Q\cap K^\circ\cap[-k,-k+1) \qquad \text{(resp. $\Q\cap K^\circ\cap
	[k\deg(r),k\deg(r)+1)$)}.\]
\end{corollary}

\begin{proof}
	\cref{prop:matrixofdifferential} gives regularity and exponents in $\Q\cap K^\circ\cap[0,1)$
	for the integral bases $b_0^0,\dots,b_{d_x-1}^0$ and $b_0^\infty,\dots,b_{d_x-1}^\infty$.
By \cref{lemma:matricesofbasischange}~\ref{matrix of connection by base change}, multiplying a basis by $r^{-k}$ adds $\ord_{x_0}(r^{-k})$ to each exponent.
	Since $r$ is squarefree,
\[\ord_{x_0}(r^{-k})=-k\quad\text{if }r(x_0)=0,\qquad\ord_\infty(r^{-k})=k\deg(r).\]
This proves the assertion.
\end{proof}

The matrix $\Phi$ on the basis $1,y,\dots,y^{d_x-1}$ represents the
$\mathcal S^\dagger$-linear map
\[
 \Theta_\tau\colon
 \mathcal R^\dagger\otimes_{\mathcal S^\dagger,\sigma_\tau}
 \mathcal S^\dagger\longrightarrow\mathcal R^\dagger,
 \qquad h\otimes a\longmapsto aF_\tau^*(h).
\]
Here $\sigma_\tau=F_\tau^*|_{\mathcal S^\dagger}$.

\begin{lemma}\label{Phi-1 integral}
We have $\Phi^{-1}\in M_{d_x}(\mathcal S^\dagger)$.
In particular, $v_K(\Phi^{-1})\geq0$.
\end{lemma}
\begin{proof}
Modulo \(\pi\), the map \(\Theta_\tau\) is the relative \(q_\tau\)-power Frobenius of \(\widetilde{U}/\widetilde{V}\).
It is an isomorphism: relative \(p\)-power Frobenius is an isomorphism for an étale morphism \cite[Tag~0EBS]{stacks-project}, and relative \(q_\tau=p^{\ell_\tau}\)-power Frobenius is its \(\ell_\tau\)-fold iteration \cite[Tag~0CCG]{stacks-project}. Hence \(d=\det(\Phi)\) is a unit modulo \(\pi\). Choose \(e\in\mathcal S^\dagger\) lifting the inverse of \(\bar d\). Then \(de=1-\pi h\) for some \(h\in\mathcal S^\dagger\). By weak completeness, \(\sum_{n\geq0}(\pi h)^n\in\mathcal S^\dagger\); thus \(d^{-1}=e\sum_{n\geq0}(\pi h)^n\). The adjugate formula now gives \(\Phi^{-1}\in M_{d_x}(\mathcal S^\dagger)\).
\end{proof}
\begin{definition}
	For a geometric point $x_0\in\P_{\overline{K}}^1(\overline K)$, let $\ord_{x_0}$ denote the discrete valuation
	on $\overline K(x)$ corresponding to $x_0$, normalized by
	\[\ord_{x_0}(x-x_0)=1 \quad\text{if }x_0\ne\infty, \qquad \ord_\infty(1/x)=1.\]
	For $f\in\overline K(x)$, define
	\[\ord_{\ne\infty}(f)\coloneq\min_{x_0\in\P_{\overline{K}}^1(\overline K)\setminus\{\infty\}}\{\ord_{x_0}(f)\}.\]
	We extend these definitions to matrices over $\overline K(x)$ by taking the minimum over their entries.
\end{definition}
\begin{proposition}\label[pprop]{prop:approximationofFrobenius}
	Let $F_\tau^*\colon\mathcal R^\dagger
	\rightarrow\mathcal R^\dagger$ be the Frobenius
	lift given in
	\cref{proposition:expressionoffrobeniuslift},
	and let $N\in\N_{\geq1}$.
	\begin{enumerate1}
		\item \label{frobenius of r}
		In $\mathcal S^\dagger$, we have
		\[
		F_\tau^*\left(\frac1r\right)
		=\frac1{r^\tau(x^{q_\tau})}
		\equiv
		\sum_{i=q_\tau}^{q_\tau N}
		\frac{\rho_i}{r^i}
		\pmod{\pi^N\mathcal S^\dagger},
		\]
		where $\rho_i\in K^\circ[x]$ with
		$\deg(\rho_i)<\deg(r)$ for
		$q_\tau\leq i\leq q_\tau N$.
		
		\item \label{frobenius of y}
		In $\mathcal R^\dagger$, for every
		$0\leq i\leq d_x-1$, we have
		\[
		F_\tau^*(y^i)
		\equiv
		\sum_{j=0}^{d_x-1}\phi_{i,j}y^j
		\pmod{\pi^N\mathcal R^\dagger},
		\]
		where
		\[\phi_{i,j}=\sum_{k=0}^{q_\tau(N-1)-\ord_{\ne\infty}(W^0)}\frac{\phi_{i,j,k}}{r^k}\]
		and $\phi_{i,j,k}\in K^\circ[x]$ with
		\begin{align*}
			\deg(\phi_{i,j,0})
			&\leq
			-\ord_\infty(W^\infty)
			-q_\tau\ord_\infty((W^\infty)^{-1}),\\
			\deg(\phi_{i,j,k})
			&<\deg(r),
		\end{align*}
		for every $0\leq j\leq d_x-1$ and
		\(1\leq k\leq q_\tau(N-1)-\ord_{\ne\infty}(W^0)-q_\tau\ord_{\ne\infty}((W^0)^{-1}).\)
		
		\item \label{frobenius of b}
		In $\mathcal R^\dagger$, for every
		$0\leq i\leq d_x-1$, we have
		\[
		F_\tau^*\left(\frac{b_i^0}{r}\right)
		\equiv
		\sum_{j=0}^{d_x-1}
		\psi_{i,j}\frac{b_j^0}{r}
		\pmod{\pi^N\mathcal R^\dagger},
		\]
		where
		\[
		\psi_{i,j}
		=\sum_{k=0}^{q_\tau N-1}
		\frac{\psi_{i,j,k}}{r^k}
		\]
		and $\psi_{i,j,k}\in K^\circ[x]$ with
		\begin{align*}
			\deg(\psi_{i,j,0})
			&\leq
			-\ord_\infty(W)
			-q_\tau\ord_\infty(W^{-1})
			-(q_\tau-1)\deg(r),\\
			\deg(\psi_{i,j,k})
			&<\deg(r),
		\end{align*}
		for every $0\leq j\leq d_x-1$ and
		$1\leq k\leq q_\tau N-1$.
	\end{enumerate1}
\end{proposition}

\begin{proof}
	We follow \cite[Proposition~2.12]{tuitman2017counting}, using \cref{thm:degreeoffrobeniusstructuresatapoint} for the $q_\tau$-power bounds over $K$.

	\ref{frobenius of r}
	Notice that
	$r^\tau(x^{q_\tau})\equiv r^{q_\tau}\pmod\pi$,
	so we have the expansion
	\[
	\frac1{r^\tau(x^{q_\tau})}
	=
	\frac1{r^{q_\tau}}
	\sum_{n=0}^\infty
	\left(
	\frac{r^{q_\tau}-r^\tau(x^{q_\tau})}
	{r^{q_\tau}}
	\right)^n.
	\]
	The terms with $n\geq N$ belong to
	$\pi^N\mathcal S^\dagger$.
	Since
	\[
	\deg\bigl(r^{q_\tau}-r^\tau(x^{q_\tau})\bigr)
	\leq q_\tau\deg(r),
	\]
	repeated Euclidean division by $r$ expresses
	the $n$-th summand, including the factor
	$r^{-q_\tau}$, as
	\[
	\frac1{r^{q_\tau}}
	\left(
	\frac{r^{q_\tau}-r^\tau(x^{q_\tau})}
	{r^{q_\tau}}
	\right)^n
	=
	\sum_{i=q_\tau}^{q_\tau(n+1)}
	\frac{\rho_{n,i}}{r^i},
	\]
	where $\rho_{n,i}\in K^\circ[x]$ and
	$\deg(\rho_{n,i})<\deg(r)$.
	Collecting the terms for $0\leq n<N$
	proves the assertion.
	
		\ref{frobenius of y}
Notice that $F_\tau^*$ induces a
$q_\tau$-th power Frobenius structure on
$x_*\OO_U$.
Let $\Phi=(\Phi_{i,j})\in M_{d_x}(\mathcal S^\dagger)$
be its matrix with respect to the basis
$1,y,\dots,y^{d_x-1}$, so $v_K(\Phi)\geq0$.
Moreover, by \cref{Phi-1 integral},
$v_K(\Phi^{-1})\geq0$.

Let $\Phi_0$ be the matrix with respect to the integral basis
$b_0^0,\dots,b_{d_x-1}^0$.
Since $W^0$ and $(W^0)^{-1}$ have entries in
$\mathcal S^\dagger$, we also have
$v_K(\Phi_0),v_K(\Phi_0^{-1})\geq0$.
By \cref{cor:regular basis of x},
\cref{thm:degreeoffrobeniusstructuresatapoint} and
\cref{remark:degreeoffrobeniusstructuresatapoint},
modulo $\pi^N$, the entries of $\Phi_0$ admit
approximations of order at least $-q_\tau(N-1)$ at
every root $x_0$ of $r(x)$. 

Since the powers of $y$ lie in $\sum_iK[x]b_i^0$,
the matrix $(W^0)^{-1}$ has polynomial entries.
Hence $\sigma_\tau((W^0)^{-1})$ also has polynomial
entries, and
\[
\ord_{x_0}\bigl(\sigma_\tau((W^0)^{-1})\bigr)\geq0.
\]
Changing to the power basis by
\cref{cor:degreeoffrobeniusstructureatnonregularpoints},
the entries of $\Phi$ therefore admit approximations
modulo $\pi^N$ whose orders at $x_0$ are at least
\begin{align*}
	-q_\tau(N-1)+\ord_{x_0}(W^0)\geq
	-q_\tau(N-1)+\ord_{\ne\infty}(W^0).
\end{align*}
Thus, since $r$ is squarefree, we can write
\[
\Phi_{j+1,i+1}
\equiv
\sum_{k=0}^{q_\tau(N-1)-\ord_{\ne\infty}(W^0)}
\frac{\phi_{i,j,k}}{r^k}
\pmod{\pi^N\mathcal S^\dagger},
\]
with $\phi_{i,j,k}\in K^\circ[x]$ and
$\deg(\phi_{i,j,k})<\deg(r)$ for $k\geq1$,
by repeated Euclidean division.

As for $\deg(\phi_{i,j,0})$, we apply
\cref{cor:degreeoffrobeniusstructureatnonregularpoints}
at $\infty$, using the integral basis
$b_0^\infty,\dots,b_{d_x-1}^\infty$.
Its exponents lie in $[0,1)$ by \cref{prop:matrixofdifferential} \ref{exponents of G}, so its Frobenius matrix admits an approximation of order at least zero by \cref{thm:degreeoffrobeniusstructuresatapoint}.
Changing to the basis $1, y, \dots, y^{d_x-1}$ gives the order bound
\[
\ord_\infty(W^\infty)
+q_\tau\ord_\infty((W^\infty)^{-1}).
\]
Since the terms $\phi_{i,j,k}/r^k$ with $k\geq1$ vanish at $\infty$, the polynomial parts can therefore be chosen to satisfy
\[-\deg(\phi_{i,j,0})=\ord_\infty(\phi_{i,j,0})\geq \ord_\infty(W^\infty)+q_\tau\ord_\infty((W^\infty)^{-1}).\]

\ref{frobenius of b}
This is analogous to \ref{frobenius of y}, using the bases $\boldsymbol{b}^0/r$ and $\boldsymbol{b}^\infty/r$.
By \cref{cor:regular basis of x}, their exponents at the corresponding points lie in $[-1,0)$ and $[\deg(r),\deg(r)+1)$, respectively.
Thus \cref{remark:degreeoffrobeniusstructuresatapoint} gives the order bounds $-(q_\tau N-1)$ at
the roots of $r$ and $(q_\tau-1)\deg(r)$ at $\infty$.
Changing basis at $\infty$ using $\boldsymbol{b}^\infty/r=(\boldsymbol{b}^0/r)W$ gives the asserted polynomial-degree bound.
\end{proof}

\begin{remark}\label{rem:evaluation-of-beta-at-suitable-bad-points}
	The condition $\beta=F_\tau^*(y)\in\mathcal R^\dagger$ ensures that $\beta$ can be evaluated at good points, this fact can be deduced easily from 		\cref{prop:approximationofFrobenius}~\ref{frobenius of y}. However, this is not the case at every bad point. We claim that $\beta$ can be evaluated at suitable bad points. Let $P$ be suitable bad in the sense of \cref{classification of points def}.
	\begin{itemize}
		\item 	If $P$ is finitely bad, then $0<v_K(r(x(P)))<1/q\leq 1/q_\tau$.
		The pole-order bounds for $F_\tau^*(y)$ in \cref{prop:approximationofFrobenius}~\ref{frobenius of y} therefore imply that the expansion of $\beta$ converges at $P$.
\item If $P$ is infinite, then $x(P)\ne\infty$ and
$v_K(x(P))<0$. Put $d=\deg r$. The expansions in
\cref{prop:approximationofFrobenius}~\ref{frobenius of y}
have the form
\[
\beta=\sum_{j=0}^{d_x-1}
\left(
\phi_{1,j,0}(x)
+\sum_{k\geq1}\frac{\phi_{1,j,k}(x)}{r(x)^k}
\right)y^j,
\]
where the degrees of the polynomial parts $\phi_{1,j,0}$
are uniformly bounded, while
$\phi_{1,j,k}\in K^\circ[x]$ and
$\deg\phi_{1,j,k}<d$ for $k\geq1$.
Thus the polynomial parts can be evaluated at $x(P)$.
Moreover, the unit leading coefficient of $r$ gives
$v_K(r(x(P)))=d\,v_K(x(P))<0$, and
\[
v_K\!\left(
\frac{\phi_{1,j,k}(x(P))}{r(x(P))^k}
\right)
\geq (d-1-kd)v_K(x(P))
\longrightarrow+\infty.
\]
Hence the $1/r$-tails converge at $P$, so $\beta(P)$
is well defined.
	\end{itemize}
\end{remark}

\begin{algorithm}[Compute matrix of Frobenius on functions] \label[palgm]{algthm:matrixoffrobeniusonfunctions}
	\
	
	\textbf{Input}: The precision $N$ and the basis $b_0^0/r,\dots, b_{d_x-1}^0/r$, i.e. the matrix $W^0$, together with $\tau$ and $\ell_\tau$.
	
	\textbf{Output}: The matrix $M\in M_{d_x}(K^\circ[x,1/r])$ up to the precision $N$ such that 
	\[F_\tau^*\begin{bmatrix}
		b_0^0/r&\cdots& b_{d_x-1}^0/r
	\end{bmatrix} \equiv \begin{bmatrix}
	b_0^0/r&\cdots& b_{d_x-1}^0/r
\end{bmatrix}M \ \ (\mod \pi^N).\]
	\begin{enumerate1}
		\item Compute $F_\tau^*(1/r)$ and $F_\tau^*(y)$, this is from the proof of \cref{proposition:expressionoffrobeniuslift}: take $n_0$ such that $N\leq 2^{n_0}$, then calculate $\alpha_n$, $\beta_n$ for $n=1,\dots, n_0$, we have $F_\tau^*(1/r)\equiv \alpha_{n_0}, F_\tau^*(y)\equiv \beta_{n_0} \ \ (\mod \pi^N)$.

		\item Compute the matrix with respect to the basis $1/r,y/r,\dots, y^{d_x-1}/r$, i.e. represent $F_\tau^*(y^i/r)$ in terms of $1/r,\dots, y^{d_x-1}/r$: we use the equality $F_\tau^*(y^i/r) = F_\tau^*(y)\cdot F_\tau^*(y^{i-1}/r)$ for $i=1,\dots, d_x-1$ (e.g. $F_\tau^*(y/r)\equiv \beta_{n_0}\alpha_{n_0} \ \ (\mod \pi^N)$), then get the matrix $\widetilde{M}$ from the coefficients of $rF_\tau^*(y^i/r)$ with respect to $1,\dots, y^{d_x-1}$.

		\item Compute the matrix $M$ with respect to the basis $b_0^0/r,\dots, b_{d_x-1}^0/r$: we have 
        \[M= (W^0)^{-1}\widetilde{M}(W^0)^\tau(x^{q_\tau}).\]
	\end{enumerate1}
\end{algorithm}

This leads to an algorithm to compute the $q_\tau$-power Frobenius of a given differential form $\left(\sum_{i=0}^{d_x-1}w_ib_i^0\right)\frac{dx}{r}$ with $w\in K[x]^{\oplus d_x}$ up to a given precision $N$, using the following formula:
\begin{align}
F_\tau^*\left(\left(\sum_{i=0}^{d_x-1}w_ib_i^0\right)\frac{dx}{r}\right)& = F_\tau^*\begin{bmatrix}
	b_0^0/r&\cdots& b_{d_x-1}^0/r
\end{bmatrix}\begin{bmatrix}
	q_\tau w_0^\tau(x^{q_\tau})x^{q_\tau-1}\\\vdots\\q_\tau w_{d_x-1}^\tau(x^{q_\tau})x^{q_\tau-1}
\end{bmatrix}dx \nonumber\\ 
&\equiv\begin{bmatrix}
	b_0^0/r&\cdots& b_{d_x-1}^0/r
\end{bmatrix}M\begin{bmatrix}
	q_\tau w_0^\tau(x^{q_\tau})x^{q_\tau-1}\\\vdots\\q_\tau w_{d_x-1}^\tau(x^{q_\tau})x^{q_\tau-1}
\end{bmatrix}dx \ \ (\mod \pi^N),
\label{frobeniusofadifferentialform}
\end{align}
where $M$ is given by \cref{algthm:matrixoffrobeniusonfunctions}. The factors $w_i^\tau$ reflect the $\tau$-semilinearity of $F_\tau^*$.

\section[Computing a basis of H^1_dR(U)]%
      {Computing a basis of \texorpdfstring{$H^1_{\mathrm{dR}}(U)$}{H¹\_dR(U)}}
\label{section:computing a basis of cohomology}

As in \cref{section:computing a frobenius lift}, throughout this section, we fix a smooth projective geometrically irreducible curve $X$ of genus $g$ over $K$ with a triple $(Q, W^0, W^\infty)$ having good reduction. Let $\mathcal{X}$ be the corresponding model of $X$ over $K^\circ$, and $\Sp\colon X^\an \to \mathcal{X}_s$ the reduction map.  The notations $\Delta(x), r(x), b^0_i, b^\infty_i$ and $X^0$, $X^\infty$ are given as in \S~\ref{subsection:plane models}. We denote $\mathcal{U}= \Spec(K^\circ[x,1/r, y]/(Q))$, $\mathcal{V}= \Spec(K^\circ[x,1/r])$ as in \cref{prop:smoothmodelassociatedtotriple}, and write $U=\mathcal{U}_K, V=\mathcal{V}_K$, $Z=X\setminus U$, $\widetilde{U} = \mathcal{U}_{\widetilde{K}}=\Spec(\widetilde{K}[x,1/r, y]/(Q))$. Recall $\mathcal{R}^\dagger=K^\circ\langle x, 1/r,y\rangle^\dagger/(Q)$.

The reduction procedures follow \cite[Propositions~3.3--3.4, Remark~3.5]{tuitman2017counting}. The residue criteria below follow the approach of \cite[Propositions~3.12--3.13]{tuitman2017counting}. We retain the proofs over $K$ and use these procedures to compute a cohomology basis.

\subsection{Reduction of differential forms}

\begin{proposition} \label[pprop]{lemma:reductionprocedureforcohomologyclass}
	\begin{enumerate1}
		\item \label{reduction at finite points} For any $l\in \N_{\geq 1}$ and $w\in K[x]^{\oplus d_x}$, there exist vectors $u, v\in K[x]^{\oplus d_x}$ with $\deg(v)<\deg(r)$ such that in $\Omega_{X/K}^1(U)$, we have
		\[\frac{\sum_{i=0}^{d_x-1}w_ib_i^0}{r^l}\dfrac{dx}{r} = d\left(\frac{\sum_{i=0}^{d_x-1}v_ib_i^0}{r^l}\right)+\frac{\sum_{i=0}^{d_x-1}u_ib_i^0}{r^{l-1}}\dfrac{dx}{r}.\]
		\item \label{reduction at infty} For any $w\in K[x,x^{-1}]^{\oplus d_x}$ with $\ord_\infty(w) \leq -\deg(r)$, there exist vectors $u, v\in K[x,x^{-1}]^{\oplus d_x}$ with $\ord_\infty(u)>\ord_\infty(w)$ such that in $\Omega_{X/K}^1(U\cap X^\infty)$, we have
		\begin{align}\label{equalityforrepresentationatinfty}
			\left(\sum_{i=0}^{d_x-1}w_ib^\infty_i\right)\dfrac{dx}{r} = d\left(\sum_{i=0}^{d_x-1}v_ib_i^\infty\right)+\left(\sum_{i=0}^{d_x-1}u_ib_i^\infty\right)\dfrac{dx}{r}.
		\end{align}
		We can take $v\in K[x]^{\oplus d_x}$ of the form $\overline{v}x^\mu$ with $\overline{v}\in K^{\oplus d_x}$ and $\mu=-\ord_\infty(w)-\deg(r)+1\geq 1$.
		
		Moreover, if $\ord_\infty(w)\leq \ord_0(W)-\deg(r)+1$, we can take $v=\overline{v}x^\mu$ with $\overline{v}\in K^{\oplus d_x}$ such that $\sum_{i=0}^{d_x-1}v_ib_i^\infty\in \OO_X(X^0)$, i.e. only has poles lying above $\infty$.
	\end{enumerate1}
\end{proposition}
\begin{proof}
	\ref{reduction at finite points} Since $b_j^0\in K^\circ[x,1/r, y]$, by \cref{prop:matrixofdifferential}, there is $G^0\in M_{d_x}(K^\circ[x,1/r])$ such that
	\[db_j^0 = \sum_{i=0}^{d_x-1}G_{i+1,j+1}^0b_i^0dx,\]
	Then we can write $G^0=M^0/r$ with $M^0\in M_{d_x}({K}^\circ[x])$.
	We need to find $u,v\in K[x]^{\oplus d_x}$ satisfying
	\[\sum_{i=0}^{d_x-1}w_ib_i^0dx = r\sum_{i=0}^{d_x-1}(v_i'b_i^0dx+v_idb_i^0)-l\sum_{i=0}^{d_x-1}v_ib_i^0r'dx+r\sum_{i=0}^{d_x-1}u_ib_i^0dx,\]
	i.e.
	\[w=r\frac{dv}{dx}+M^0v-lr'v+ru.\]
	This implies that
	\[u=\frac{w-(M^0-lr'I)v}{r}-\frac{dv}{dx}.\]
	Then it suffices to find $v\in K[x]^{\oplus d_x}$ such that
	\[\left(\frac{M^0}{r'}-lI\right)v\equiv \frac{w}{r'} \ \ (\mod r)\]
	over $K[x]/(r)$. By  \cref{prop:matrixofdifferential}~\ref{exponents of G}, we know that $\det(M^0/r'-lI)$ is invertible in $K[x]/(r)$, so such $v$ exists. This completes the proof of \ref{reduction at finite points}. 
	
	\ref{reduction at infty} Denote $t=x^{-1}$, notice that $dx = -t^{-2}{dt}$. Keep the notation in \cref{prop:matrixofdifferential}. Around $x=\infty$, we can write
	\[G^\infty dx = \left(\frac{G_{-1}^\infty}{t}+G_0^\infty+\cdots\right)dt,\]
	i.e.
	\[G^\infty = -(G^\infty_{-1}t+G^\infty_0t^2+\cdots)\]
	for some $G_i^\infty\in M_{d_x}(K)$. Set $\mu= -\ord_\infty(w)-\deg(r)+1\geq 1$ and write
	\[w\frac{dx}{r} = \sum_{j=-(\mu+1)}^\infty w_j't^{j}dt\]
	for some $w_j'\in K^{\oplus d_x}$. Write 
	\[d\begin{bmatrix}
		b_0^\infty& \cdots & b_{d_x-1}^\infty
	\end{bmatrix} = \begin{bmatrix}
		b_0^\infty& \cdots & b_{d_x-1}^\infty
	\end{bmatrix}G^\infty dx \] 
	and compare both sides of \eqref{equalityforrepresentationatinfty}, i.e.
	\[\begin{bmatrix}
		b_0^\infty& \cdots & b_{d_x-1}^\infty
	\end{bmatrix}w\frac{dx}{r}  = \begin{bmatrix}
		b_0^\infty& \cdots & b_{d_x-1}^\infty
	\end{bmatrix}dv + d\begin{bmatrix}
		b_0^\infty& \cdots & b_{d_x-1}^\infty
	\end{bmatrix}v+ \begin{bmatrix}
		b_0^\infty& \cdots & b_{d_x-1}^\infty
	\end{bmatrix}u\frac{dx}{r},\]
	we have 
	\[u=w-r\left(G^\infty v+\frac{dv}{dx}\right).\]
	To have $\ord_\infty(u)>\ord_\infty(w)$, it suffices to choose $v\in K[x,x^{-1}]^{\oplus d_x}$ such that
	\[\ord_\infty\left(\frac{w}{r}\right)= -\mu+1 = \ord_\infty\left(G^\infty v+\frac{dv}{dx}\right)\]
	and the coefficient vector of $G^\infty v+\frac{dv}{dx}$ at $t^{-\mu+1}$ must equal $-w'_{-(\mu+1)}=c/\operatorname{lc}(r)$ (notice that $dx=-t^{-2}dt$), where $c$ is the coefficient vector of $w$ at $x^{\deg(w)}$. If $v=\overline v x^\mu$ with $\overline v\in K^{\oplus d_x}$, then the coefficient of $G^\infty v+\frac{dv}{dx}$ at $x^{\mu-1}$ is $(\mu I-G^\infty_{-1})\overline v$.
	Note that $\mu\geq 1$, by \cref{prop:matrixofdifferential}~\ref{exponents of G}, we know that $\mu I-G^\infty_{-1}\in M_{d_x}(K)$ is invertible, so the linear system 
	\[w_{-(\mu+1)}' = (G^\infty_{-1}- \mu I)\overline{v}\]
	has a unique solution $\overline{v}\in K^{\oplus d_x}$. This proves the first part of \ref{reduction at infty}.
	
	If $\ord_\infty(w)\leq \ord_0(W)-\deg(r)+1$, from our construction of $v$, we have \[\ord_0(v)=\mu = -\ord_\infty(w)-\deg(r)+1\geq-\ord_0(W).\]
	Then for any point $P$ lying above $x=0$ with ramification index $e_P$, we have
	\begin{align*}
		\ord_P\left(\sum_{j=0}^{d_x-1}v_jb_j^\infty\right)&=\ord_P\left(\sum_{i,j=0}^{d_x-1}v_jW_{i+1,j+1}b_i^0\right)\\
		&\geq \inf\{\ord_P(v_jW_{i+1,j+1})\}\\
		&\geq e_P(\ord_0(v)+\ord_0(W))\\
		&\geq0,
	\end{align*}
	finishes the proof.
\end{proof}

\subsection{Computing a cohomology basis}

\begin{art}
	Recall that given an $n$-dimensional algebraic variety $T$ over $K$, we have the \emph{de Rham complex} for $T/K$:
	\[\Omega_{T/K}^\bullet\colon \xymatrix{0\ar[r]& \OO_T\ar[r]^d & \Omega_{T/K}^1\ar[r]^d&\cdots\ar[r]^d&\Omega_{T/K}^n\ar[r]&0}.\]
	The de Rham cohomology of $T$ is defined as the hypercohomology $H^i_{\mathrm{dR}}(T)\coloneq \mathbb{H}^i(\Omega^\bullet_{T/K})$. In particular,
	\[H_{\mathrm{dR}}^1(U)=\Coker(d\colon K[x,1/r,y]/(Q)\to K[x,1/r,y]/(Q) dx).\]
	
	On the other hand, we define the \emph{first rigid cohomology} (or \emph{first Monsky--Washnitzer cohomology}) of $\widetilde{U}$ as \[H_{\mathrm{rig}}^1(\widetilde{U})
	\coloneq\Coker(d\colon \mathcal{R}^\dagger_K\rightarrow \Omega^1_{\mathcal R^\dagger_K/K}).\]
	A classical result says that 
	\[H^1_{\mathrm{rig}}(\widetilde{U})\simeq H^1_{\mathrm{dR}}(U).\]
	From the theory of rigid cohomology, a Frobenius lift $F_\tau^*\colon \mathcal{R}^\dagger\rightarrow  \mathcal{R}^\dagger$ will induce a $\tau$-semilinear automorphism of $H^1_{\mathrm{rig}}(\widetilde{U})$, hence a $\tau$-semilinear automorphism of $H^1_{\mathrm{dR}}(U)$.
\end{art}

\begin{theorem} \label{thm: computing de rham via E B}
	Consider the following $K$-vector spaces:
	\begin{align*}
		E_0:&=\Set{\left(\sum_{i=0}^{d_x-1}u_i(x)b^0_i\right)\frac{dx}{r} |  u\in K[x]^{\oplus d_x}}\subset \Omega_{K(X)/K}^1(U^0),\\
		E_\infty:&=\Set{\left(\sum_{i=0}^{d_x-1}v_i(x,x^{-1})b_i^\infty\right)\frac{dx}{r}| v\in K[x,x^{-1}]^{\oplus d_x}, \ord_\infty(v)\geq \ord_0(W) - \deg(r)+2}\\
		\subset& \Omega_{K(X)/K}^1(U^\infty),\\
		B_0:&=\Set{\sum_{i=0}^{d_x-1}u_i(x)b^0_i | u\in K[x]^{\oplus d_x}}\\
		B_\infty:&=\Set{\sum_{i=0}^{d_x-1}v_i(x,x^{-1})b_i^\infty| v\in K[x,x^{-1}]^{\oplus d_x}, \ord_\infty(v)\geq \ord_0(W)+1},\\
	\end{align*}
	Then $E_0\cap E_\infty$ and $d(B_0\cap B_\infty)$ are finite dimensional $K$-vector spaces and 
	\[H_{\mathrm{dR}}^1(U)\simeq (E_0\cap E_\infty)/((E_0\cap E_\infty)\cap d(B_0\cap B_\infty)).\]
\end{theorem}

\begin{proof}
The proof of \cite[Theorem~3.6]{tuitman2017counting} applies over $K$.
We include it for the reader's convenience.

	Notice that the elements in $E_0$ and $B_0$ have bounded poles everywhere but points lying above $x=\infty$, and elements in  $E_\infty$ and $B_\infty$ have bounded poles everywhere but points lying $x=0$. So elements in $E_0\cap E_\infty$ and $d(B_0\cap B_\infty)$ have bounded poles everywhere on $X$. Hence they are contained in the space of global sections of  some line bundles on $X$ and are of finite dimension.
	
	Next, we show that every class in $H^1_{\mathrm{dR}}(U)$
	can be represented by a $1$-form in $E_0\cap E_\infty$.
	By repeatedly applying
	\cref{lemma:reductionprocedureforcohomologyclass}~\ref{reduction at finite points},
	every class has a representative $\omega\in E_0$.
	Write
	$\omega=\left(\sum_{i=0}^{d_x-1}w_ib_i^\infty\right)\frac{dx}{r}$
	with $w\in K[x,x^{-1}]^{\oplus d_x}$.
	We repeatedly apply
	\cref{lemma:reductionprocedureforcohomologyclass}~\ref{reduction at infty}
	while $\ord_\infty(w)\leq\ord_0(W)-\deg(r)+1$.
	By its last assertion, each subtracted differential belongs
	to $dB_0\subseteq E_0$, so the remainder stays in $E_0$.
	Since $\ord_\infty(w)$ strictly increases at each step,
	the process terminates with a representative in $E_0\cap E_\infty$.
	
	Next, we show that every class in $H^1_{\mathrm{dR}}(U)$ can be represented by a $1$-form in $E_0\cap E_\infty$. This is divided into two steps. By \cref{lemma:reductionprocedureforcohomologyclass} \ref{reduction at finite points}, every class can be represented by a $1$-form $\omega$ in $E_0$. Write $\omega=\left(\sum_{i=0}^{d_x-1}w_ib_i^\infty\right)\frac{dx}{r}$ with $w\in K[x,x^{-1}]^{\oplus d_x}$. If $\ord_\infty(w)\leq\ord_0(W)-\deg(r)+1$, by \cref{lemma:reductionprocedureforcohomologyclass} \ref{reduction at infty}, we have 
	\begin{align*}
		\left(\sum_{i=0}^{d_x-1}w_ib^\infty_i\right)\dfrac{dx}{r} = d\left(\sum_{i=0}^{d_x-1}v_ib_i^\infty\right)+\left(\sum_{i=0}^{d_x-1}u_ib_i^\infty\right)\dfrac{dx}{r}
	\end{align*} 
for some $u, v\in K[x,x^{-1}]^{\oplus d_x}$ with  $\ord_\infty(u)>\ord_\infty(w)$. By the last assertion of the same proposition, we may take $\sum_{i=0}^{d_x-1} v_i b_i^\infty\in B_0$, so the remainder lies in
$E_0$ since $dB_0\subseteq E_0$.
Repeating this step finitely many times gives
$\ord_\infty(w)>\ord_0(W)-\deg(r)+1$.
Hence, every cohomology class can be represented by a $1$-form in $E_0\cap E_\infty$.
	
	Finally, we show that if $\omega\in E_0\cap E_\infty$ is exact, then $\omega\in d(B_0\cap B_\infty)$. Let $\omega=\left(\sum_{i=0}^{d_x-1}u_i(x)b^0_i\right)\frac{dx}{r}=\left(\sum_{j=0}^{d_x-1}v_j(x,x^{-1})b_j^\infty\right)\frac{dx}{r}\in E_0\cap E_\infty$ be an exact $1$-form. For any point $P\in X(\overline{K})$ not lying over $x=\infty$, we have
	\begin{align*}
		\ord_P(\omega) = \ord_P\left(\left(\sum_{i=0}^{d_x-1}u_i(x)b^0_i\right)\frac{dx}{r}\right)\geq -1.
	\end{align*}
	For any point $P\in X(\overline{K})$ lying over $x=\infty$, notice that $dx=-x\frac{d(x^{-1})}{x^{-1}}$, we have
		\begin{align*}
		\ord_P(\omega) &= \ord_P\left(\left(\sum_{j=0}^{d_x-1}v_j(x,x^{-1})b_j^\infty\right)\frac{dx}{r}\right)\\
		&\geq e_P(\ord_\infty(v)+\deg(r)+\ord_\infty(x))-1,\\
		&\geq e_P(\ord_0(W)+1)-1.
		\end{align*}
	If $\omega=df$ for some $f\in \OO_X(U)$, then by our discussion above, we have that $\ord_P(f)\geq0$ at all points $P$ not lying over $x=\infty$ and that $\ord_P(f) \geq e_P(\ord_0(W)+1)$ at any point $P$ lying over $x=\infty$. Hence $f\in \OO_X(X^0)$ and $x^{\ord_0(W)+1}f\in\OO_X(X^\infty)$, i.e. $f\in B_0\cap B_\infty$.
\end{proof}
\begin{remark}\label[prmk]{remark:E0Einf}
For an element $\left(\sum_{i=0}^{d_x-1}u_i(x)b^0_i\right)\frac{dx}{r} =\left(\sum_{j=0}^{d_x-1}v_j(x,x^{-1})b_j^\infty\right)\frac{dx}{r} \in E_0\cap E_\infty$, we have that
			\[\begin{bmatrix}
				u_0\\
				u_1\\
				\vdots\\
				u_{d_x-1}
			\end{bmatrix} = W\begin{bmatrix}
				v_0\\
				v_1\\
				\cdots\\
				v_{d_x-1}
			\end{bmatrix}\] 
		and	$\deg(u) = -\ord_\infty(u)$, so 
			\[\ord_\infty(u)  \geq \ord_\infty(W)+\ord_\infty(v)\geq  \ord_\infty(W)+\ord_0(W) - \deg(r)+2,\]
			and 
			\[\ord_\infty(v)\geq \ord_\infty(W^{-1})+\ord_\infty(u)\geq\ord_\infty(W^{-1}) + (\ord_\infty(W)+\ord_0(W) - \deg(r)+2),\]
			\[\ord_0(v)\geq \ord_0(W^{-1})+\ord_0(u)\geq \ord_0(W^{-1}).\]
So we can shrink $E_0$ such that $\deg(u)$ is bounded from above as above. The base changing will induce the following map
		\begin{align*}
			E_0\rightarrow \Set{v\in K[x,x^{-1}]^{\oplus d_x}\mid\parbox[c]{.65\linewidth}{ $\ord_0(v)\geq\ord_0(W^{-1}),\\ \ord_\infty(v)\geq \ord_\infty(W^{-1}) + \ord_\infty(W)+\ord_0(W) - \deg(r)+2$}}
		\end{align*}
	which is injective, and notice that entries of the elements in the codomain are of the form $\sum_{i=\ord_0(W^{-1})}^{-(\ord_\infty(W^{-1}) + \ord_\infty(W)+\ord_0(W) - \deg(r)+2)}a_ix^i$.
	    Taking the coefficient vector in the $\boldsymbol{b}^\infty$ basis and projecting onto the excluded Laurent monomials, $E_0\cap E_\infty$ is the kernel of the following linear map
	    \[E_0\rightarrow \Set{v\in K[x,x^{-1}]^{\oplus d_x}\mid\parbox[c]{.65\linewidth}{ $\ord_0(v)\geq-(\ord_0(W)-\deg(r)+1),\\ \ord_\infty(v)\geq  \ord_\infty(W^{-1}) + \ord_\infty(W)+\ord_0(W) - \deg(r)+2$}}.\]
		
		Similarly, for an element $\left(\sum_{i=0}^{d_x-1}u_i(x)b^0_i\right)\in B_0\cap B_\infty$, we have
		$\deg(u)\leq -\ord_0(W)-\ord_\infty(W)-1$. We can shrink $B_0$ such that $\deg(u)$ is bounded from above by $-\ord_0(W)-\ord_\infty(W)-1$, and $B_0\cap B_\infty$ is the kernel of the following linear map, obtained by the same coefficient projection
		\[B_0\rightarrow \Set{v\in K[x,x^{-1}]^{\oplus d_x}\mid\parbox[c]{.55\linewidth}{ $\ord_0(v)\geq -\ord_0(W),\\
				\ord_\infty(v)\geq \ord_\infty(W^{-1})+\ord_0(W)+\ord_\infty(W)+1$}}.\]
		Using $G^0$, it is not hard to get the image $d(B_0\cap B_\infty)$ in $E_0\cap E_\infty$.
\end{remark}

Given a form $\omega\in \Omega_{X/K}^1(U)$, the proofs of \cref{lemma:reductionprocedureforcohomologyclass} and \cref{thm: computing de rham via E B} give an algorithm to represent a cohomology class by a $1$-form in $E_0\cap E_\infty$.

\begin{definition}
	A meromorphic differential $1$-form on $X$ over $K$ is said to be of 
	\begin{itemize}
		\item \emph{the first kind} if it is holomorphic, i.e. an element in $H^0(X,\Omega_{X/K}^1)$;
		\item \emph{the second kind} if it has residue $0$ at every point;
		\item \emph{the third kind} if all its poles are simple with residues in $\Z$.
	\end{itemize}
\end{definition}

Recall that the differential forms of the second kind on $X$ modulo exact differential forms, i.e. the differentials of rational functions, is isomorphic to $H^1_{\mathrm{dR}}(X)$ as $K$-vector spaces. They are of dimension $2g$ since we have the short exact sequence
\begin{align}
\xymatrix{0\ar[r]&H^0(X,\Omega_{X/K}^1)\ar[r]& H_{\mathrm{dR}}^1(X)\ar[r]& H^1(X,\OO_X)\ar[r]&0},
\label{exactsequence:cohomology1}
\end{align}
here $H^0(X,\Omega_{X/K}^1)$ is the space of differential forms of the first kind, and $\dim_KH^0(X,\Omega_{X/K}^1) = \dim_KH^1(X,\OO_X)=g$.

\begin{definition}
For any $\omega\in\Omega_{X/K}^1(U)$ and $P\in X\setminus U$, let $\mathrm{res}_P(\omega)\in \OO_{X,P}/(t_P)\simeq \kappa(P)$ be $a_{-1}$ in the Laurent series expansion
\[\omega=(a_{-k}t_P^{-k}+\cdots+a_{-1}t_P^{-1}+\cdots)dt_P,\]
where $t_P\in K(X)$ is a uniformizer for $P$.
Moreover, we denote
\[\mathrm{res}_0=\bigoplus_{P\in X\setminus U: x(P)\not=\infty}\mathrm{res}_P, \ \ \mathrm{res}_\infty=\bigoplus_{P\in X\setminus U: x(P)=\infty}\mathrm{res}_P,\]
\end{definition}
\begin{remark}\label{remark: exact sequence for residue map}
	We have an exact sequence
	\begin{align}\label{exactsequence:cohomology2}
		\xymatrix{0\ar[r]& H_{\mathrm{dR}}^1(X)\ar[r]& H_{\mathrm{dR}}^1(U)\ar[rr]^-{\mathrm{res}_0\oplus\mathrm{res}_\infty}&& \bigoplus_{P\in X\setminus U}\OO_{X,P}/(t_P)}.
	\end{align}
Write $\mathrm{res}=\mathrm{res}_0\oplus\mathrm{res}_\infty$. Under the
comparison with rigid cohomology, functoriality of the residue sequence gives
\[
 \mathrm{res}(F_\tau^*\omega)
 =q_\tau F_{Z,\tau}^*(\mathrm{res}(\omega)),
\]
where $F_{Z,\tau}^*$ is the induced $\tau$-semilinear Frobenius on the
boundary residue space. Thus the coefficients are twisted and geometric
boundary points are transported accordingly. In particular, the kernel
of $\mathrm{res}$, namely $H_{\mathrm{dR}}^1(X)$, is preserved by $F_\tau^*$.
\end{remark}

\begin{proposition}
	Let $\omega=\sum_{i=0}^{d_x-1}w_i(x)b_i^0\frac{dx}{r}\in\Omega_{X/K}^1(U)$ with $w\in K[x]^{\oplus d_x}$. For any $x_0\in\overline K$ with $r(x_0)=0$, let $v_{x_0}\coloneq w|_{x=x_0}\in \overline{K}^{\oplus d_x}$, $G_{-1}^{0}(x_0)\coloneq (x-x_0)G^0|_{x=x_0}\in M_{d_x}(\overline{K})$, and $\mathcal{E}^{x_0}_\lambda$ the eigenspace of $G_{-1}^{0}(x_0)$ with eigenvalue $\lambda$. Notice that $\overline{K}^{\oplus{d_x}}=\bigoplus_{\lambda}\mathcal{E}_\lambda^{x_0}$. Then $\mathrm{res}_0(\omega)=0$ $\Longleftrightarrow$ the projection of $v_{x_0}$ onto $\mathcal{E}_0^{x_0}$ vanishes for all $x_0\in\overline K$ with $r(x_0)=0$.
\end{proposition}
\begin{proof}
	Let $P \in X\setminus U$ with $x(P)\not=\infty$. Then $\ord_P\left(\frac{dx}{r}\right)=-1$ and $\ord_P(\omega)\geq-1$. Hence $\mathrm{res}_P(\omega)=0$ if and only if $\ord_P\left(\sum_{i=0}^{d_x-1}w_i(x)b_i^0\right)\geq 1$. We fix a root $x_0$ of $r$ and extend the ground field to $\overline K$ for the following local calculation. As argued in the proof of \cref{prop:matrixofdifferential}~\ref{exponents of G}, $G_{-1}^0(x_0)$ is the matrix of the map
	\[(x-x_0)\frac{d}{dx}\colon \OO_X(X^0)/(x-x_0)\to \OO_X(X^0)/(x-x_0)\]
	with respect to the basis $b_0^0, \dots, b_{d_x-1}^0$. Moreover, via the isomorphisms of vector spaces $$\OO_X(X^0)/(x-x_0)\simeq \bigoplus_{P\in X^0, x(P)=x_0}\OO_{X,P}/(t_P^{e_P})\simeq \bigoplus_{P\in X^0, x(P)=x_0}\bigoplus_{n=1}^{e_P}(t^{n-1}_P)/(t^n_P),$$
	where $t_P$ is a uniformizer of $\OO_{X,P}$, the $0$-eigenspace corresponds to $\bigoplus_{P\in X^0, x(P)=x_0}\OO_{X,P}/(t_P)\simeq \bigoplus_{P\in X^0, x(P)=x_0}K$. Hence $\ord_P\left(\sum_{i=0}^{d_x-1}w_i(x)b_i^0\right)\geq 1$ for each $P\in X\setminus U$ with $x(P)=x_0$ if and only if its projection in $0$-eigenspace is $0$.	
\end{proof}

Similarly, the following proposition holds.

\begin{proposition}\label{prop:eigenspace at infty}
Let $\omega=\sum_{i=0}^{d_x-1}w_i(x,x^{-1})b_i^\infty\frac{dx}{r}\in\Omega^1_{X/K}(U)$ with $w\in K[x,x^{-1}]^{\oplus d_x}$ satisfying $\ord_\infty(w)>-\deg(r)$. Let $v\coloneq x^{1-\deg(r)}w|_{x=\infty}\in K^{\oplus d_x}$, $G_{-1}^{\infty}\coloneq -xG^\infty|_{x=\infty}\in M_{d_x}(K)$, and $\mathcal{E}^{\infty}_\lambda$ the eigenspace of $G_{-1}^{\infty}$ with eigenvalue $\lambda$. Notice that $K^{\oplus{d_x}}=\bigoplus_{\lambda}\mathcal{E}_\lambda^\infty$. Then $\mathrm{res}_\infty(\omega)=0$ $\Longleftrightarrow$ the projection of $v$ onto $\mathcal{E}_0^{\infty}$ vanishes.
\end{proposition}

We can use the following algorithm to compute a basis $\omega_1,\dots, \omega_{m}$ of $E_0\cap E_\infty$ including an integral cohomology basis.

\begin{algorithm}[Computing cohomology basis]
\label[palgm]{algthm:cohomologybasis}
		\
	
	\textbf{Input}: A good model $(Q, W^0, W^\infty)$.
	
	\textbf{Output}: A basis of integral $1$-forms $\omega_1,\dots, \omega_{n}$ of $E_0\cap E_\infty$, and functions $g_{m+1},\dots, g_n\in B_0\cap B_\infty$ whose coefficient vectors in both $\boldsymbol{b}^0$ and $\boldsymbol{b}^\infty$ are integral such that via the decomposition $$E_0\cap E_\infty\simeq H^0(X,\Omega^1_{X/K})\oplus \left(H_{\mathrm{dR}}^1(X)/H^0(X,\Omega^1_{X/K})\right)\oplus (H_{\mathrm{dR}}^1(U)/H_{\mathrm{dR}}^1(X)) \oplus d(B_0\cap B_\infty),$$ 
	we have that
	\begin{itemize}
		\item $\omega_1,\dots,\omega_g$ form a basis of $H^0(X,\Omega^1_{X/K})$;
		\item $\omega_1,\dots, \omega_{2g}$ represent a basis of $H^1_{\mathrm{dR}}(X)$;
		\item $\omega_1,\dots, \omega_{m}$ represent a basis of $H^1_{\mathrm{dR}}(U)$;
		\item $\omega_{m+1}=dg_{m+1}, \dots, \omega_{n}=dg_{n}$ form a basis of $d(B_0\cap B_\infty)$, in particular, $g_{m+1},\dots, g_n$ are linearly independent in $B_0\cap B_\infty$.
	\end{itemize}
	\begin{enumerate1}
		\item\label{algthm:cohomologybasis1} ($E_0$) We shrink $E_0$ as \cref{remark:E0Einf}, and via the basis of $E_0$ \[\left\{x^jb_i^0\cdot\frac{dx}{r}\mid 0\leq j\leq \deg(r)-\ord_0(W)-\ord_\infty(W)-2, \ \  0\leq i\leq d_x-1\right\},\]
		we write $E_0\simeq K^{\oplus d_x(\deg(r)-\ord_0(W)-\ord_\infty(W)-1)}$.
		\item \label{algthm:cohomologybasisE0capEinf} ($E_0\cap E_\infty$) Construct a map by the projection via the base change from $E_0$ to the $K$-vector space generated by 
		\[\Set{x^jb_i^\infty\cdot\frac{dx}{r}\mid \parbox[c]{.65\linewidth}{$-\ord_0(W)+\deg(r)-1\leq  j\leq -\ord_\infty(W)-\ord_\infty(W^{-1})-\ord_0(W)+\deg(r)-2, \ \  0\leq i\leq d_x-1$}},\]
		then we can obtain the kernel $E_0\cap E_\infty$.

\item\label{algthm:cohomologybasis2} (The second kind)
At finite boundary points, compute the residue maps using the $0$-eigenspaces of $G_{-1}^0(x_0)$.
At infinity, first subtract exact differentials to obtain a form satisfying the condition $\ord_\infty(w)>-\deg(r)$ in \cref{prop:eigenspace at infty}, then apply the preceding $0$-eigenspace criterion.
Intersecting the kernels of these maps with $E_0\cap E_\infty$, we obtain the subspace $V$ of $1$-forms of the second kind. 
		\item (The first kind) Construct the natural projection via the base change from $E_0$ to the $K$-vector space generated by  
		\[\Set{x^jb_i^\infty\cdot\frac{dx}{r}\mid \parbox[c]{.727\linewidth}{$\deg(r)\leq  j\leq -\ord_\infty(W)-\ord_\infty(W^{-1})-\ord_0(W)+\deg(r)-2,$  $0\leq i\leq d_x-1$}},\] 
		whose kernel $V'$ is the subspace of $E_0$ consisting of $1$-forms with order $\geq -1$ at infinity, then we obtain the subspace $H^0(X,\Omega^1_{X/K})$ of $E_0\cap E_\infty$ consisting of holomorphic $1$-forms (or $1$-forms of the first kind) which is $V'\cap V$.
		\item ($B_0\cap B_\infty$) Compute $B_0\cap B_\infty$ analogously to steps~\ref{algthm:cohomologybasis1}--\ref{algthm:cohomologybasisE0capEinf} (see the end of \cref{remark:E0Einf}), and compute $d(B_0\cap B_\infty)$ as a subspace of $E_0\cap E_\infty$.
		\item (Taking a basis) We can take a basis $g_m=1, g_{m+1},\dots, g_n$ of $B_0\cap B_\infty$, then $dg_{m+1},\dots, dg_n$ is a basis of $d(B_0\cap B_\infty)$. On the other hand, we have 
		\begin{align*}
H^1_{\mathrm{dR}}(X)/H^0(X,\Omega_{X/K}^1)& \simeq V/(H^0(X,\Omega_{X/K}^1)+d(B_0\cap B_\infty)),\\
H^1_{\mathrm{dR}}(U)/H_{\mathrm{dR}}^1(X)&\simeq(E_0\cap E_\infty)/V.
		\end{align*}
It is not hard to take an integral basis of these spaces. 
	\end{enumerate1}

\end{algorithm}

Given a differential form of the second kind $\omega$, we can apply the following algorithm to represent its image in $H_{\mathrm{dR}}^1(X)$ with respect to the basis $\omega_1,\dots, \omega_{2g}$.

\begin{algorithm}[Reduction procedure] \label[palgm]{algthm:reductionprocedure} 
	\
	
	\textbf{Input}: The integral $1$-forms $\omega_1,\dots, \omega_m\in E_0\cap E_\infty$, and rational functions $g_{m+1},\dots, g_n$ given by \cref{algthm:cohomologybasis}, and an integral differential form $\omega=\sum_{l\leq l_0}\frac{\sum_{i=0}^{d_x-1}w_{l,i}b_i^0}{r^l}\frac{dx}{r}\in \Omega_{X/K}^1(U)$ with $w_l=(w_{l,0},\dots, w_{l,d_x-1})\in K^\circ[x]^{\oplus d_x}$ and $\deg(w_{l,i})<\deg(r)$. The matrix $M^0=rG^0\in M_{d_x}(K[x])$. 
	
	\textbf{Output}: Functions $f_0, f_\infty, f_{end}=\sum_{j=m+1}^na_jg_j\in \OO_X(U)=\sum_{i=0}^{d_x-1}K[x,1/r]b_i^0$ and $\omega'=\sum_{i=1}^ma_i\omega_i$ such that
	\[\omega=df_0+df_\infty+df_{end}+\omega'.\]
	\begin{enumerate1}
		\item \label{reduction:finite} (Finite reduction) We run the finite reduction from $l=l_0$ to $l=1$. 
		\begin{enumerate}[leftmargin=*,label=\it(1.\arabic*),ref=\it{(1.\arabic*)}]
			\item \label{finite reduction 1.1} For each $l$ fixed, calculate $u,v\in K[x]^{\oplus d_x}$ in \cref{lemma:reductionprocedureforcohomologyclass}~(1). 
            Compute a matrix $\widetilde{M}\in M_{d_x}(K[x])$ such that
			\[\widetilde{M}\equiv \left(\frac{M^0}{r'}-lI\right)^{-1}/r' \ \ (\mod r),\]
			where $r'=dr/dx$.
			Then set \[v_l\coloneq (\widetilde{M}w_l)\bmod r, \ \  u_l\coloneq\frac{w_l-(M^0-lr'I)v_l}{r}-\frac{dv_l}{dx}.\]
			We replace $w_{l-1}$ by $w_{l-1}+u_l$.
			\item \label{finite reduction 1.2} Repeat Step~\ref{finite reduction 1.1} $l_0$-times, we have $\omega= df_0 + \widetilde{\omega}$ with $\widetilde{\omega} = \left(\sum_{i=0}^{d_x-1}\widetilde{u}_i(x)b^0_i\right)\frac{dx}{r}\in E_0$.
		\end{enumerate}
		\item \label{reduction:infinite} (Infinite reduction) 
		\begin{enumerate}[leftmargin=*,label=\it(2.\arabic*),ref=\it{(2.\arabic*)}] 
			 \item \label{infinite reduction 2.1} Write $\widetilde{\omega}=\sum_{i=0}^{d_x-1}\widetilde{w}_ib_i^\infty\frac{dx}{r}$ by 
			 \[\begin{bmatrix}
			 	\widetilde{w}_0\\
			 	\vdots\\
			 	\widetilde{w}_{d_x-1}
			 \end{bmatrix}=W^{-1}\begin{bmatrix}
			 \widetilde{u}_0\\
			 \vdots\\
			 \widetilde{u}_{d_x-1}
		 \end{bmatrix}.\]
\item \label{infinite reduction 2.2} If $\ord_\infty(\widetilde{w})\leq \ord_0(W)-\deg(r)+1$, we then calculate $u,v\in K[x,x^{-1}]^{\oplus d_x}$ in \cref{lemma:reductionprocedureforcohomologyclass}~\ref{reduction at infty}.
Let $\mu=-\ord_\infty(\widetilde{w})-\deg(r)+1$. Assume that $\mu\geq 1$, otherwise we are done, and move to \ref{reduction:final}. Let $c$ be the coefficient vector of $\widetilde{w}$ at $x^{\deg(\widetilde{w})}$, and set
\[
c_{-(\mu+1)}\coloneq-\frac{c}{\operatorname{lc}(r)}.
\]
Since $t=1/x$ and $dx=-t^{-2}dt$, this is the coefficient of $t^{-(\mu+1)}dt$ in $\widetilde{w}\,dx/r$. Now set
\[
\overline{v}\coloneq (G^\infty_{-1}-\mu I)^{-1}c_{-(\mu+1)}, \ \ v\coloneq\overline{v}x^\mu, \ \ u\coloneq \widetilde{w}-r\left(G^\infty v+\frac{dv}{dx}\right).
\]
Then we replace $\widetilde{w}$ by $u$.
			\item \label{infinite reduction 2.3} Repeat Step~\ref{infinite reduction 2.2} until $\ord_\infty(\widetilde{w})> \ord_0(W)-\deg(r)+1$, we have that $\widetilde{\omega}= df_\infty + \widetilde{\widetilde{\omega}}$ for some $\widetilde{\widetilde{\omega}} = \left(\sum_{i=0}^{d_x-1}v_i(x)b^0_i\right)\frac{dx}{r}\in E_0\cap E_\infty$.
		\end{enumerate}
	    \item \label{reduction:final} (Final reduction) Since $\omega-df_0-df_\infty\in E_0\cap E_\infty$, we can find $a_1,\dots, a_n\in K$ such that \[\omega-df_0-df_\infty=\sum_{j=1}^na_j\omega_j.\]
	    Set $f_{end}=\sum_{j=m+1}^na_jg_j$ and $\omega' = \omega-df_0-df_{\infty}-df_{end}$.
	\end{enumerate1}
\end{algorithm}
\begin{remark}
	In infinite reduction \ref{infinite reduction 2.2}, $$\mu\coloneq -\ord_\infty(\widetilde w)-\deg(r)+1\geq -\ord_0(W).$$
	Notice that $\widetilde u=(\widetilde{u}_0, \dots, \widetilde{u}_{d_x-1})$ is given by $u_l$ and $\sum_{l\leq 0}\frac{\sum_{i=0}^{d_x-1}w_{l,i}b_i^0}{r^l}\frac{dx}{r}$, so we cannot bound $\mu$ from above by $-\ord_\infty(W^{-1})-1$ by \cref{prop:valuationoffinitereduction} below.
\end{remark}
\begin{remark}
	In the implementation, the finite and infinite reduction matrices are precomputed once at each required level and reused when reducing the \(2g\) Frobenius pullbacks.
\end{remark}

\subsection{Precision bounds for reduction}

We conclude this section by estimating the loss of precision of the reduction procedure. The proofs follow the ideas of \cite[Lemma~4.3.4]{bas2007point} and \cite[Proposition~3.7, Proposition~3.8]{tuitman2016counting}.

For arbitrary integral bases, compare also \cite[Propositions~3.7--3.8]{tuitman2017counting}. With our normalization, integer denominators contribute $v_K(a)=e_Kv_p(a)$ to the precision loss.

\begin{proposition}
	\label[pprop]{prop:valuationoffinitereduction}

	Let $\omega = \frac{\sum_{i=0}^{d_x-1}w_ib_i^0}{r^l}\frac{dx}{r}\in \Omega^1_{\mathcal{U}/K^\circ}(\mathcal{U})$, where $l\in \N_{\geq 1}$ and $w\in K^\circ[x]^{\oplus d_x}$ satisfies $\deg(w)<\deg(r)$. We set
	\[e_0\coloneq \max\{e_P\mid P\in {X}\setminus {U}, x(P)\not=\infty\}.\]
	By \cref{lemma:reductionprocedureforcohomologyclass}~\ref{reduction at finite points}, there are  $u\in K[x]^{\oplus d_x}$ and $f=\sum_{j=1}^l\frac{\sum_{i=0}^{d_x-1}(v_j)_ib_i^0}{r^j}$, $v_j\in K[x]^{\oplus d_x}$ satisfying $\deg((v_j)_{i})<\deg(r)$ for all $1\leq j\leq l$, $0\leq i\leq d_x-1$, such that
		\[\omega=df+\left(\sum_{i=0}^{d_x-1}u_ib_i^0\right)\frac{dx}{r}\]
Then the following statements hold.
	\begin{enumerate1}
		\item \label{prop:valuationoffinitereduction1} We have \[p^{\lfloor\log_p(le_0)\rfloor}u\in K^\circ[x]^{\oplus d_x},  \ \ p^{\lfloor\log_p(le_0)\rfloor}v_j\in K^\circ[x]^{\oplus d_x}.\]
		\item \label{prop:valuationoffinitereduction2} If we write $\left(\sum_{i=0}^{d_x-1}u_ib_i^0\right)\frac{dx}{r}=\left(\sum_{i=0}^{d_x-1}u'_ib_i^\infty\right)\frac{dx}{r}$ by basis change $u' = W^{-1}u$, then
		\[\ord_\infty(u')\geq \ord_\infty(W^{-1}) -\deg(r)+2.\]
	\end{enumerate1}
	
\end{proposition}
\begin{proof}

\ref{prop:valuationoffinitereduction1} 
 We set $n_0\coloneq p^{\lfloor\log_p(le_0)\rfloor}$, and claim that $n_0v_j\in K^\circ[x]^{\oplus d_x}$ for any $j$. Otherwise, there is $c\in \N_{\geq 1}$ such that $\min_{1\leq j\leq l}\{v_K(\pi^cn_0v_j)\}=0$. Let $1\leq j_0\leq l$ be the maximal one such that $v_K(\pi^cn_0v_{j_0})=0$. 
 Let $h$ be the reduction of $\pi^cn_0f$ modulo $\pi$, and $h'$ the reduction of $\sum_{i=0}^{d_x-1}\pi^cn_0(v_{j_0})_ib_i^0$ modulo $\pi$. 
 
 For any $P\in {X}\setminus {U}$ with $x(P)\not=\infty$, the extension $K(P)/K$ is unramified by the smoothness of $\mathcal Z_{\mathcal X}$ (see \cref{prop:smoothmodelassociatedtotriple} \ref{sm model of X}), hence $\pi$ is a uniformizer of $K(P)$, and $v_K$ extends to $K(P)$ with $v_K(\pi)=1$. We choose a uniformizer $t_P$ of $P$ on $\mathcal{X}$ which is also a uniformizer of $P$ on $X$ (resp. $\Sp(P)$ on $\mathcal{X}_s$) after restriction, see \cref{prop:uniformizeratapoint} in the next section. Then around such $P$, we have $\ord_P(dx/r)=-1$, so $$\text{$\ord_P(\omega), \ord_P(df)\geq -le_P-1$ \ \  and \ \ $\ord_P\left(\left(\sum_{i=0}^{d_x-1}u_ib_i^0\right)\frac{dx}{r}\right) \geq -1$.}$$
 Hence we can write
	\[df=\sum_{k\geq -le_P-1}a_{k}t_P^kdt_P, \ \ \ \ f=\lambda_P+\sum_{\substack{k\geq -le_P-1\\k\ne-1}}\frac{a_k}{k+1}t_P^{k+1},\]
where $a_k,\lambda_P\in K(P)$.
Since $\ord_P(dx/r)=-1$, the remainder $\omega-df$ has at most a simple pole at $P$. As $\omega$ is integral, we have $a_k\in K(P)^\circ$ for $-le_P-1\leq k\leq-2$.
For these indices, $v_p(k+1)\leq\lfloor\log_p(le_0)\rfloor$, so
\[v_K(\pi^cn_0a_k/(k+1))\geq c\geq1.\]
This implies that $\ord_{\Sp(P)}(h)\geq0$ and $\ord_{\Sp(P)}(r^{j_0}h)\geq j_0e_P\geq e_P$. Also note that \[\OO_{\mathcal{X}}(\mathcal{X}-x^{-1}(\infty))/(r)\simeq \bigoplus_{P\in X\setminus U, x(P)\not=\infty}\OO_{\mathcal{X},P}/(t_P^{e_P}),\]
	 so $r^{j_0}h=0$ in $\OO_{\mathcal{X}_s}(\mathcal{X}_s-x^{-1}(\infty))/(r) = \bigoplus_{i=0}^{d_x-1}(\widetilde{K}[x]/r)b_i^0$ which implies that $h'=0$. Hence the reduction of $\pi^cn_0v_{j_0}$ is $0$ in $(\widetilde{K}[x]/r)^{\oplus d_x}$. This is impossible since $\deg(v_{j_0})<\deg(r)$. So our claim holds. This implies that $p^{\lfloor\log_p(le_0)\rfloor}u\in K^\circ[x]^{\oplus d_x}$.

\ref{prop:valuationoffinitereduction2} Put $D=\deg(r)$ and $a=\ord_\infty(W^{-1})$. Write
\[
f=\boldsymbol{b}^\infty z,\qquad z=\sum_{j=1}^l\frac{W^{-1}v_j}{r^j}.
\]
Since $\deg(v_j)<D$, we have
\[
\ord_\infty(z)\geq a+1.
\]
By \cref{prop:matrixofdifferential} \ref{regularity of G}, $\ord_\infty(G^\infty dx)\geq -1$, so $\ord_\infty(G^\infty)\geq1$. Differentiating with respect to $x$ raises the order at infinity by at least one. Consequently,
\[
\ord_\infty\left(\frac{dz}{dx}+G^\infty z\right)\geq a+2.
\]
Comparing coefficients in $\omega=df+\boldsymbol{b}^\infty u'\,dx/r$ gives
\[
u'=\frac{W^{-1}w}{r^l}-r\left(\frac{dz}{dx}+G^\infty z\right).
\]
The first term has order at least $a+(l-1)D+1\geq a+1$, and the second at least $a+2-D$. Since $D\geq1$, it follows that
\[
\ord_\infty(u')\geq a-D+2.
\]
\end{proof}
\begin{remark}
This proposition implies the loss of precision of $\omega$ is $\lfloor\log_p(le_0)\rfloor v_K(p)$, i.e. if $w$ is correct to precision $N$, then 
		$u,v$ are correct to precision $N-\lfloor\log_p(le_0)\rfloor v_K(p)$.
\end{remark}

\begin{proposition}\label[pprop]{prop:valuationofinfinitereduction}
	
	Let $\omega = \sum_{i=0}^{d_x-1}w_i(x,x^{-1})b_i^\infty\frac{dx}{r}\in \Omega^
	1_{\mathcal{U}/K^\circ}(\mathcal{U})$, where $w\in K^\circ[x, x^{-1}]^{\oplus d_x}$ satisfies $\ord_\infty(w)\leq \ord_0(W)-\deg(r)+1$. We set $\mu=-\ord_\infty(w)-\deg(r)+1$ and
	\[e_\infty\coloneq\max\{e_P\mid P\in {X}\setminus {U}, x(P)=\infty\}.\]
	By \cref{lemma:reductionprocedureforcohomologyclass}~\ref{reduction at infty}, there are $u\in K[x,x^{-1}]^{\oplus d_x}$ and $f=\sum_{j=-\ord_0(W)}^\mu\left(\sum_{i=0}^{d_x-1}(v_j)_ib_i^\infty\right)x^j$ with $\ord_\infty(u)>\ord_0(W)-\deg(r)+1$ and $v_j\in K^{\oplus d_x}$ for all $-\ord_0(W)\leq j\leq\mu$, such that
	\[\omega=df+\left(\sum_{i=0}^{d_x-1}u_i(x,x^{-1})b_i^\infty\right)\frac{dx}{r}.\]
	Then
	\[p^{\lfloor\log_p(\mu e_\infty)\rfloor}u\in K^\circ[x,x^{-1}]^{\oplus d_x}, \ \ p^{\lfloor\log_p(\mu e_\infty)\rfloor}v_j\in (K^\circ)^{\oplus d_x}.\]
\end{proposition}
\begin{proof}
	This proof is similar to the one of \cref{prop:valuationoffinitereduction}~\ref{prop:valuationoffinitereduction1}. We set $n_0\coloneq p^{\lfloor\log_p(\mu e_\infty)\rfloor}$, and claim that $n_0v_j\in (K^\circ)^{\oplus d_x}$ for any $-\ord_0(W)\leq j\leq\mu$. Otherwise, there is $c\in \N^+$ such that $\min_{-\ord_0(W)\leq j\leq\mu}\{v_K(\pi^cn_0v_j)\}=0$. Let $-\ord_0(W)\leq j_0\leq\mu$ be the maximal one such that $v_K(\pi^cn_0v_{j_0})=0$. Let $h$ be the reduction of $\pi^cn_0f =\sum_{j=-\ord_0(W)}^\mu\left(\sum_{i=0}^{d_x-1}\pi^cn_0(v_j)_ib_i^\infty\right)x^j$, and $h'$ the reduction of $\sum_{i=0}^{d_x-1}\pi^cn_0(v_{j_0})_ib_i^\infty$. Let $t\coloneq1/x$.
	
	For any $P\in {X}\setminus {U}$ with $x(P)=\infty$, the extension $K(P)/K$ is unramified as in the proof of \cref{prop:valuationoffinitereduction} \ref{prop:valuationoffinitereduction1}. Let $t_P$ be a uniformizer of $\overline{\{P\}}\subset \mathcal{X}$. Then around such $P$, we have $\ord_P(dx/r)=e_P(\deg(r)-1)-1$, so $\ord_P(\omega), \ord_P(df)\geq-\mu e_P-1$ and
	\begin{align*}
		\ord_P\left(\left(\sum_{i=0}^{d_x-1}u_ib_i^\infty\right)\frac{dx}{r}\right)&\geq e_P(\ord_0(W)-\deg(r)+2)+e_P(\deg(r)-1)-1\\
		&=e_P(\ord_0(W)+1)-1.
	\end{align*}
	Hence we can write
	\[df=\sum_{k\geq-\mu e_P-1}a_kt_P^kdt_P, \ \ \ \ f=\lambda_P+\sum_{\substack{k\geq-\mu e_P-1\\k\ne-1}}\frac{a_k}{k+1}t_P^{k+1},\]
	where $a_k,\lambda_P\in K(P)$.
	Since $\omega\in \Omega^1_{\mathcal{U}/K^\circ}(\mathcal{U})$, for any $k\leq e_P(\ord_0(W)+1)-2$, we have $a_k\in K(P)^\circ$, so $n_0a_k/(k+1)\in K(P)^\circ$ and $v_K(\pi^cn_0a_k/(k+1))\geq1$. Hence $\ord_{\Sp(P)}(h)\geq e_P(\ord_0(W)+1)$ and
	$$\ord_{\Sp(P)}(t^{j_0}h)\geq e_P(j_0+\ord_0(W)+1)\geq e_P.$$
	Notice that
	\[\OO_{\mathcal{X}}(\mathcal{X}-x^{-1}(0))/(t)\simeq \bigoplus_{P\in X\setminus U, x(P)=\infty}\OO_{\mathcal{X},P}/(t_P^{e_P}),\]
	so $t^{j_0}h=0$ in $\OO_{\mathcal{X}_s}(\mathcal{X}_s-x^{-1}(0))/(t)=\bigoplus_{i=0}^{d_x-1}\widetilde{K}b_i^\infty$ which implies that $h'=0$. Hence the reduction of $\pi^cn_0v_{j_0}$ is $0$ in $\widetilde{K}^{\oplus d_x}$. This is impossible by our choice of $j_0$. So our claim holds. This immediately implies that $n_0u\in K^\circ[x,x^{-1}]^{\oplus d_x}$.
\end{proof}

\begin{art}[Change of basis for the final reduction]
	\label{notation:finalreductionmatrix}
	Use the finite-dimensional space $E_0$ and its monomial basis
	\[
	\mathcal E=
	\left(x^jb_k^0\frac{dx}{r}\right)_{
		0\leq k\leq d_x-1,\ 0\leq j\leq D},
	\qquad
	D=\deg(r)-\ord_0(W)-\ord_\infty(W)-2,
	\]
	from \cref{algthm:cohomologybasis}~\ref{algthm:cohomologybasis1},
	ordered first by $k$ and then by $j$, and put $s=\dim_K E_0$.
	Fix the integral basis $\omega_1,\dots, \omega_n$ of $E_0\cap E_\infty$ by \cref{algthm:cohomologybasis}, with $\omega_i=dg_i$ for $m<i\leq n$.
	Extend them to a basis
	\[
	\mathcal B=
	(\omega_1,\dots,\omega_m,
	\xi_1,\dots,\xi_{s-n},dg_{m+1},\dots,dg_n)
	\]
	of $E_0$, choosing the additional forms $\xi_i$ from
	$\mathcal E$. Define the change-of-basis matrix
	$C_{end}\in\GL_s(K)$ by
	\[
	\mathcal B=\mathcal E C_{end},
	\qquad
	c_{\mathrm{end}}
	\coloneq -v_K(C_{end}^{-1}).
	\]
	Thus $C_{end}^{-1}$ converts monomial coordinates
	into coordinates with respect to $\mathcal B$, and
	$\pi^{c_{end}}C_{end}^{-1}$ is integral. Notice that $c_{end}\in\N$ since $\mathcal B$ is an integral basis.
\end{art}

\begin{proposition}
	\label[pprop]{prop:valuationoffinalreduction}
	Let 
	\(\omega=
	\left(\sum_{k=0}^{d_x-1}w_kb_k^0\right)\frac{dx}{r}
	=\left(\sum_{k=0}^{d_x-1}w'_kb_k^\infty\right)\frac{dx}{r}
	\in E_0\cap E_\infty\cap
	\Omega^1_{\mathcal U/K^\circ}(\mathcal U)
	\), where $w\in K^\circ[x]^{\oplus d_x}$.
	With the basis fixed in \cref{notation:finalreductionmatrix}, the final step of
	\cref{algthm:reductionprocedure} gives a unique decomposition
	\begin{align*}
		\omega=df_{end}+\sum_{i=1}^{m}a_i\omega_i,
		\qquad &f_{end}=\sum_{i=m+1}^{n}a_i g_i,\\
	\end{align*}
	Write 
	\[		f_{end}=\sum_{k=0}^{d_x-1}u_kb_k^0
	=\sum_{k=0}^{d_x-1}u'_kb_k^\infty
	\in B_0\cap B_\infty\]
	with $u\in K[x]^{\oplus d_x}, u'\in K[x,x^{-1}]^{\oplus d_x}$.
	Then
	\begin{gather*}
		\pi^{c_{end}}u\in K^\circ[x]^{\oplus d_x},
		\qquad
		\pi^{c_{end}}u'\in
		K^\circ[x,x^{-1}]^{\oplus d_x}, \qquad
		\pi^{c_{end}}a_i\in K^\circ
		\quad(1\leq i\leq n).
	\end{gather*}
\end{proposition}

\begin{proof}
	We keep the notation in  \cref{notation:finalreductionmatrix}.
Notice that the monomial coordinate vector
	$\mathbf e(\omega)$ (with respect to $\mathcal E$) lies in $(K^\circ)^s$.
	The change-of-basis identity gives
	\[
	C_{end}^{-1}\mathbf e(\omega)
	=
	\begin{bmatrix}
		a_1\\ \vdots\\ a_m\\
		0_{s-n}\\
		a_{m+1}\\ \vdots\\ a_n
	\end{bmatrix}.
	\]
	The definition of $c_{end}$ therefore implies
	$\pi^{c_{end}}a_i\in K^\circ$ for every $i$.
	Since the $g_i$ have integral coefficient vectors in both integral bases,
	$\pi^{c_{end}}f_{end}$ is integral on
	$\mathcal U$. Hence $\pi^{c_{end}}u\in K^\circ[x]^{\oplus d_x},
	\pi^{c_{end}}u'\in
	K^\circ[x,x^{-1}]^{\oplus d_x}.$
\end{proof}

\section{Computing Coleman integrals}
\label{computing Coleman integrals}

To calculate Coleman integrals, we use the $K$-linear lift
$F_q^*=(F_\tau^*)^{n_\tau}$. The coefficient automorphism $\tau$ of
the single step may be nontrivial. The preparation and tiny-integral
algorithms below do not involve a Frobenius lift.

\subsection{Preparation}

Let us first recall the basic properties of Coleman integrals, see \cite[Theorem~3.1]{balakrishnan2020explicit}.

\begin{theorem}[Coleman, Coleman-de Shalit]
	Let $\eta, \xi$ be $1$-forms on a wide open $V$ of $X^\an$ and $P,R,S\in V(K)$. Let $a,b\in K$. The Coleman integral has the following properties:
	\begin{enumerate1}
		\item Linearity: $\int_{P}^R(a\eta+b\xi)=a\int_P^R\eta+b\int_P^R\xi$.
		\item Additivity in endpoints: $\int_{P}^R(\xi)=\int_P^S\xi+\int_S^R\xi$.
		\item Change of variables: If $V'\subset X'$ is a wide open subspace of a Berkovich space $X'$ and $\phi\colon V\rightarrow V'$ is an analytic map, then $\int_P^R\phi^*\xi=\int_{\phi(P)}^{\phi(R)}\xi$.
		\item Fundamental theorem of calculus: $\int_P^Rdf=f(R)-f(P)$ for an analytic function $f$ on $V$.
		\item Galois equivariance: The integral is compatible with the action of $\Gal(\C_p/\Q_p)$.
	\end{enumerate1}
\end{theorem}
\begin{remark}
		The integral is independent of the choice of fields since they can be viewed as an integral on a wide open of $X_{\C_p}^\an$, where $\C_p$ is the completion of the algebraic closure $\overline{\Q_p}$. 
\end{remark}

As said in \cite{balakrishnan2020explicit}, to specify a point $P$ in $X(K)$, it is not sufficient to give the $(x,y)$-coordinates since we don't require $Q$ to be smooth. However, since $(Q, W^0, W^\infty)$ has good reduction, $\boldsymbol{b}^0$ (resp. $\boldsymbol{b}^\infty$) form a basis of integral closure of $K[x]$ (resp. $K[x^{-1}]$) in $K(X)$, a point $P$ is uniquely determined by the value of $x$ (resp. $x^{-1}$ if $P$ is infinite) and the values of $\boldsymbol{b}^0$ (resp. $\boldsymbol{b}^\infty$ if $P$ is infinite).

\begin{proposition}\label[pprop]{prop:uniformizeratapoint}
	Let $P\in X(K)$ and put $\widetilde P=\Sp(P)$.
	Suppose that $t_P\in\OO_{\mathcal X,\widetilde P}\subset K(X)$
	satisfies $t_P(P)=0$ and that the maximal ideal of
	$\OO_{\mathcal X,\widetilde P}$ is generated by $\pi$ and $t_P$ (we call such a $t_P$  an \emph{integral local parameter at $P$}). 
 Then
	\begin{enumeratea}
		\item \label{local equation of closure} $t_P$ is a local equation of the section $\overline{\{P\}}\subset \mathcal X$ at $\widetilde P$;
		\item \label{reduction of uniformizer} the restriction of $t_P$ to the generic fiber is a uniformizer at $P$ on $X$;
		the reduction $\widetilde t_P$ of $t_P$ is a uniformizer at $\widetilde P$ on $\mathcal X_s$;
		\item \label{isomorphism with disc}
		$t_P$ induces an isomorphism between the residue disc of $P$ and $D(0,1)$, sending $P$ to $0$, where $D(0,1)$ is the open unit disc.
	\end{enumeratea}
	In practice, we can take
	\[t_P=\begin{cases}
		x-x(P) & \text{ if $e_{\widetilde{P}}=1$ (or $t=1/x-1/x(P)$ if $P$ is infinite),}\\
		b_i^0-b_i^0(P) \text{ for some $i$} & \text{ if $e_{\widetilde{P}}\geq 2$ (with $\boldsymbol{b}^0$ replaced by $\boldsymbol{b}^\infty$ if $P$ is infinite).}
	\end{cases}\]
	where $e_{\widetilde P}$ is the ramification index of the reduced map $\widetilde x\colon\mathcal X_s\rightarrow\P^1_{\widetilde{K}}$ at $\widetilde P$.
\end{proposition}

\begin{proof}
	\ref{local equation of closure} Let $I\subset \OO_{\mathcal X,\widetilde P}$ be the ideal of the section $\overline{\{P\}}$ at $\widetilde P$, and set $\mathfrak m=(\pi,t_P)$.
	Since $\mathcal X$ is smooth over $K^\circ$ of relative dimension one, $\OO_{\mathcal X,\widetilde P}$ is a regular local ring of dimension two, and $I=(s)$ for a local equation $s$ of the section $\overline{\{P\}}$.
	
	The condition $t_P(P)=0$ implies that $t_P\in I$, so $t_P=as$ for some $a\in \OO_{\mathcal X,\widetilde P}$.
	If $a$ were not a unit, then $t_P\in\mathfrak m^2$.
	Since $\mathfrak m=(\pi,t_P)$, this would imply
	\[\dim_{\widetilde K}\mathfrak m/\mathfrak m^2\leq1,\]
	contrary to the regularity and dimension of $\OO_{\mathcal X,\widetilde P}$.
	Thus $a$ is a unit, and $I=(t_P)$.
	
	\ref{reduction of uniformizer}
	Restricting the ideal $I$ of the section to the generic fiber at $P$ gives
	\[
	I\OO_{X,P}=(t_P)\OO_{X,P},
	\]
	which is the maximal ideal of $\OO_{X,P}$.
	Hence $t_P$ is a uniformizer at $P$ on $X$.
	
	Moreover, $\OO_{\mathcal X_s,\widetilde P}=\OO_{\mathcal X,\widetilde P}/(\pi)$ whose maximal ideal is generated by $\widetilde t_P$, because $\mathfrak m=(\pi,t_P)$. Since $\mathcal X_s$ is smooth at $\widetilde P$, this local ring is a discrete valuation ring.
	Thus $\widetilde t_P$ is a uniformizer at $\widetilde P$ on $\mathcal X_s$.

	\ref{isomorphism with disc} This is a classical result, see the proof of \cite[Proposition~2.2 (ii)]{bosch1985stable}.
	
	We now justify the practical choices.
	Write $u=x$ and $b_i=b_i^0$ if $P$ is finite,
	and $u=1/x$ and $b_i=b_i^\infty$ otherwise.
	By definition,
	\[
	e_{\widetilde P}
	=
	\ord_{\widetilde P}
	\bigl(\widetilde u-\widetilde{u(P)}\bigr).
	\]
	If $e_{\widetilde P}=1$, the reduction of $u-u(P)$
	is a uniformizer at $\widetilde P$, so we can take
	$t_P=u-u(P)$.
	
	Assume $e_{\widetilde P}\geq2$. Since the reductions of the $b_i$ form an integral basis over the corresponding affine chart of $\P^1_{\widetilde K}$, the maximal ideal $\mathfrak m$ of $\OO_{\mathcal X_s,\widetilde P}$ is generated by the reductions
	\[\widetilde u-\widetilde{u(P)} \quad\text{and}\quad \widetilde b_i-\widetilde{b_i(P)}, \qquad 0\leq i\leq d_x-1.\]
	The first generator belongs to $\mathfrak m^2$,
	because its order at $\widetilde P$ is
	$e_{\widetilde P}\geq2$.
	Since $\dim_\kappa\mathfrak m/\mathfrak m^2=1$,
	at least one of the remaining generators has order
	one at $\widetilde P$.
	Choose such an $i$ and take $t_P=b_i-b_i(P)$.
	
	In both cases, $t_P$ is integral at $\widetilde P$,
	vanishes at $P$, and its reduction is a uniformizer.
	Therefore $\pi$ and $t_P$ generate the maximal ideal
	of $\OO_{\mathcal X,\widetilde P}$, as required.
\end{proof}

\begin{remark}\label{rem:ramification-at-good-very-bad}
	In our code, we take $t_P$ only when $P$ is good or very bad. In this case
	\[
	e_{\widetilde P}=e_P.
	\]
	Indeed, if $P$ is good, both the generic and reduced
	maps are unramified at the corresponding points,
	so $e_{\widetilde P}=e_P=1$.
	
	If $P$ is very bad, as noted in
	\cref{rem:unique-very-bad-point}, $P$ is the only
	very bad point in its residue disc.
	Write $u=x$ if $P$ is finite and $u=1/x$ otherwise.
	Every zero of $u-u(P)$ is very bad, since $x(P)$
	is a root of $r$ in the finite case and $u(P)=0$
	in the infinite case.
	Thus $\overline{\{P\}}$ is the only horizontal
	zero component of $u-u(P)$ passing through
	$\widetilde P$.
	There is also no vertical zero component, because
	$\widetilde u-\widetilde{u(P)}$ is a nonzero function
	on $\mathcal X_s$.
	
	By \cref{prop:uniformizeratapoint}, $t_P$ is a local
	equation of $\overline{\{P\}}$ at $\widetilde P$.
	Since $\ord_P(u-u(P))=e_P$, and
	$\OO_{\mathcal X,\widetilde P}$ is a regular local ring,
	the preceding description of the zero components gives
	\[
	u-u(P)=a\,t_P^{e_P},
	\qquad
	a\in\OO_{\mathcal X,\widetilde P}^{\times}.
	\]
	Reducing this equality, and using that
	$\widetilde t_P$ is a uniformizer and
	$\widetilde a$ is a unit at $\widetilde P$, yields
	\[
	e_{\widetilde P}
	=
	\ord_{\widetilde P}
	\bigl(\widetilde u-\widetilde{u(P)}\bigr)
	=e_P,\]
	which proves the equality.
\end{remark}

To compute a local expansion in a uniformizer $t$, we first compute a minimal polynomial of the function over $K(t)$.

\begin{algorithm}[Minimal polynomial of a function in a function field]
	\label[palgm]{algthm:minimalpolynomialofafunction}
	\
	
	\textbf{Input}: A function field
	$K'(C)=K'(x)[y]/(Q)$ of a curve $C$ over a field
	extension $K'/K$, and two elements $f_1,f_2\in K'(C)$,
	with $f_1$ nonconstant. Let $z$ be an indeterminate.
	
	\textbf{Output}: The minimal polynomial
	$P_{f_2}(z)\in K'(f_1)[z]$ of $f_2$ over $K'(f_1)$.
	
	\begin{enumerate1}
		\item \label{algthm:minimalpolynomialofafunction1}
		Take $d\in\N_{\geq1}$ large enough and calculate
		\[
		f_1^if_2^j=\sum_{k=0}^{d_{ij}}\alpha_{ijk}y^k
		\]
		for $0\leq i,j\leq d$, with $\alpha_{ijk}\in K'(x)$.
		
		\item \label{algthm:minimalpolynomialofafunction2}
		Find a nonzero element
		$c=(c_{ij})_{0\leq i,j\leq d}\in(K')^{\oplus(d+1)^2}$
		such that $\sum_{i,j=0}^{d}c_{ij}f_1^if_2^j=0$.
		Equivalently, find
		\[
		B(z)=\sum_{i,j=0}^{d}c_{ij}f_1^iz^j\in K'[f_1][z],
		\qquad B(f_2)=0.
		\]
		\begin{enumerate}[leftmargin=*,label=\it(2.\arabic*),ref=\it{(2.\arabic*)}]
			\item \label{algthm:minimalpolynomialofafunction21}
			Choose a nonzero common denominator $D(x)\in K'[x]$
			for the $\alpha_{ijk}$. Choose $\widetilde d\in\N$
			bounding the degrees of all $D(x)\alpha_{ijk}$, and write
			\[
			D(x)\alpha_{ijk}
			=\sum_{l=0}^{\widetilde d}a_{ijkl}x^l\in K'[x],
			\]
			allowing $a_{ijkl}=0$. Then
			$D(x)f_1^if_2^j\in K'[x,y]/(Q)$ for all $i,j$.
			
			\item \label{algthm:minimalpolynomialofafunction22}
Find $c=(c_{ij})_{0\leq i,j\leq d}\in {K'}^{\oplus(d+1)^2}$ such that $\sum_{l=0}^{\widetilde{d}}\sum_{k=0}^{\max_{i,j}\{d_{ij}\}}\left(\sum_{i,j=0}^{d}c_{ij}a_{ijkl}\right)x^ly^k=0$, i.e. solve the linear equation
\[
\sum_{i,j=0}^{d}c_{ij}a_{ijkl}=0
\]
for $0\leq l\leq\widetilde d$ and
$0\leq k\leq\max_{i,j}\{d_{ij}\}$.
			
			\item If $c=0$ is the only solution, increase $d$ and repeat Steps~\ref{algthm:minimalpolynomialofafunction1}~\ref{algthm:minimalpolynomialofafunction2}.
		\end{enumerate}
		
		\item Factor $B(z)$ in $K'(f_1)[z]$ and return the monic
		irreducible factor $P_{f_2}(z)$ satisfying $P_{f_2}(f_2)=0$.
	\end{enumerate1}
\end{algorithm}

\begin{algorithm}[Expansion of a function in an integral local parameter]
	\label[palgm]{algthm:expansionofafunctionatagoodpoint}
	
	\
	
	\textbf{Input}: An integral local parameter $t$ in the sense of \cref{prop:uniformizeratapoint} at a point $P\in X(K)$, and a function $h\in K(X)$ integral at $\Sp(P)\in\mathcal X$.
	
	\textbf{Output}: The expansion $h_P(t)=\sum_{i=0}^\infty a_it^i\in K^\circ[[t]]$ of $h$ around $P$, under the Hensel conditions below.
	
	\begin{enumerate1}
		\item Use \cref{algthm:minimalpolynomialofafunction},
		with $f_1=t$ and $f_2=h$, to calculate the minimal
		polynomial of $h$ over $K(t)$.
		After clearing denominators and scaling, obtain a
		polynomial relation $P_h(t,z)\in K^\circ[t,z]$
		with nonzero reduction
		$\widetilde P_h(t,z)\in\widetilde K[t,z]$.
		
		\item Compute the expansion
		$\widetilde h(t)\in\widetilde K[[t]]$ of the reduction
		of $h$ at $\Sp(P)$.
		Using this expansion and $h(P)$, construct an
		approximate root $h_0(t)\in K^\circ[[t]]$
		on the intended branch.
		Put $P_h'=\partial P_h/\partial z$, and let $v_t$
		denote the $t$-adic valuation.
		Once
		\[
		s=v_t(P_h'(t,h_0))<\infty,
		\qquad
		v_t(P_h(t,h_0))>2s,
		\]
		and the leading $t$-coefficient of $P_h'(t,h_0)$
		is a unit of $K^\circ$, apply Newton iteration:
		\[
		h_{n+1}
		=h_n-\frac{P_h(t,h_n)}{P_h'(t,h_n)}.
		\]
		This yields $h_P(t)$.
	\end{enumerate1}
\end{algorithm}

\begin{remark}
	We compute only a finite truncation of $\widetilde h(t)$ and increase
	its $t$-precision as needed to select the intended branch of
	$P_h(z)=0$. If the available coefficients do not distinguish that
	branch, the algorithm does not certify a root at that precision.
\end{remark}

\subsection{Tiny integrals}

If two endpoints lie in different residue discs and one of them is bad but not suitable bad, we replace it by a suitable bad point $S\in X(\overline K)$ in the same residue disc, see \cref{prop:convergentdomainoffinitereduction,rem:iterate-convergence}. If $S$ is finitely bad, the condition $v_K(r(x(S)))\in(0,1/q)$ implies that $S\notin X(K)$. By \cref{lem:suitable-bad-point-criterion}, such a point exists over $K(\pi^{1/e})$ for sufficiently large $e$. Thus, in the rest of this section, we allow endpoints in $X(K(\pi^{1/e}))$ for sufficiently large $e$.

\begin{definition}
	Let $P,R\in X(K(\pi^{1/e}))$, and $\omega\in \Omega_{X/K}^1(U)$. We say that a Coleman integral $\int_P^R\omega$ is a \emph{tiny integral} if $P,R$ lie in the same residue disc with respect to the model $\mathcal{X}$. 
\end{definition}
\begin{remark}
	By saying that a point $P\in X(K')$ over a finite extension $K'/K$ is in a residue disc $D$, we mean the following statement: the rational point $P: \Spec(K')\rightarrow X$ determines a rigid point $P\in X^\an$, then $P\in D\subset X^\an$. For a general finite extension $K'/K$, this disc need not be
	$K$-rational. In the auxiliary totally ramified extension
	$K(\pi^{1/e})/K$, however, the residue field is unchanged, so every
	point reduces to a $\widetilde K$-rational point.
\end{remark}	

\begin{algorithm}[Tiny integrals]
	\label[palgm]{algthm:tinyintegral}
	
	\
	
	\textbf{Input}: Points $P,R\in X(K(\pi^{1/e}))$ in a $K$-rational residue disc and a differential $\omega\in\Omega_{X/K}^1(U)$ of the second kind, regular at $P$ and $R$.
	
	\textbf{Output}: The value $\int_P^R\omega$.
	
	\begin{enumerate1}
		\item If the residue disc is bad, take $P'\in X(K)$ to be its unique very bad point, which is $K$-rational by \cref{rem:unique-very-bad-point}. Otherwise choose $P'\in X(K)$ in the same residue disc.
		
		\item Apply \cref{prop:uniformizeratapoint} and \cref{algthm:expansionofafunctionatagoodpoint} to determine an integral local parameter at $P'$ and the power series expansion $x(t), b_i^0(t)$ with respect to $t$. 
		\item Integrate using $t$ as coordinate:
		\[\int_P^R\omega=\int_{t(P)}^{t(R)}\omega(t).\]
	\end{enumerate1}
\end{algorithm}

\begin{remark}
	We explain why our algorithm always works.
	 
The integral local parameter $t$ gives
	\[\widehat{\OO}_{\mathcal X,\widetilde{P'}}\simeq K^\circ[[t]].\]
	Hence $x(t),b_i^0(t)$ (resp.\ $1/x(t),b_i^\infty(t)$ in an infinite disc) have integral coefficients and converge for $|t|<1$. 
	
	The only possible pole of $\omega$ in the disc is at $P'$. After multiplication by a scalar and a
	power of $t$, it is integral at $\widetilde{P'}$.
	Thus its Laurent expansion has finitely many negative powers and coefficients with valuations
	bounded below.
	Since $\omega$ is of the second kind, the $t^{-1}dt$ coefficient vanishes.
	Division by the integers arising in termwise integration causes only logarithmic valuation loss, so the power series part of the primitive still converges for $|t|<1$.
	Regularity at $P,R$ therefore permits evaluation at both endpoints.
\end{remark}
\subsection{Integrals in general}

We now explain the computation of Coleman integrals between different residue discs. Tiny integrals reduce the computation at arbitrary endpoints to good or suitable bad endpoints, as in \cref{algthm:integralonbasis}.

\begin{art}\label{idea of computing integral in general}
	Let $P,R\in X(K(\pi^{1/e}))$ be good or suitable bad points in the sense of \cref{classification of points def}, $\omega_1,\dots, \omega_{2g}$ a basis of $H_{\mathrm{dR}}^1(X)$ given in \cref{algthm:cohomologybasis}.
	We explain how the Frobenius iterate gives the integrals $\int_{P}^R\omega_i$, $ 1\leq i\leq 2g$. Put
	\[
	\boldsymbol\omega=
	\begin{bmatrix}
		\omega_1\\
		\vdots\\
		\omega_{2g}
	\end{bmatrix}.
	\]
	Reducing the pullbacks (see \cref{algthm:matrixoffrobeniusonfunctions} and \cref{algthm:reductionprocedure}) gives
	$F_\tau^*\boldsymbol\omega=A_\tau\boldsymbol\omega+ d\boldsymbol f_\tau$. Notice that \cref{remark: exact sequence for residue map} ensures that no additional basis
	classes of $H^1_{\mathrm{dR}}(U)$ occur.
	The action on constants is $\tau$, so applying $F_\tau^*$ once more gives
	\begin{align*}
	(F_\tau^*)^2\boldsymbol\omega
=& F_\tau^*\bigl(A_\tau\boldsymbol\omega+ d \boldsymbol f_\tau\bigr)\\
=& \tau(A_\tau)F^*_\tau\boldsymbol\omega+ d\bigl(F_\tau^*\boldsymbol f_\tau\bigr)\\
=& \tau(A_\tau)A_\tau\boldsymbol\omega+d\bigl(F_\tau^*\boldsymbol f_\tau+\tau(A_\tau)\boldsymbol f_\tau\bigr).
	\end{align*}
	Iterating $n_\tau$ times, we obtain the $K$-linear $q$-power Frobenius lift and its matrix: $q=q_\tau^{n_\tau}$, $F_q^*=(F_\tau^*)^{n_\tau}$,
	\[
	F_q^*\boldsymbol\omega=A_q\boldsymbol\omega+d\boldsymbol g_q\]
with\[
	A_q=\tau^{n_\tau-1}(A_\tau)\cdots\tau(A_\tau)A_\tau, \qquad
	\boldsymbol g_q
	= \sum_{j=0}^{n_\tau-1}\bigl(\tau^{n_\tau-1}(A_\tau)\cdots\tau^{j+1}(A_\tau)\bigr)(F_\tau^*)^j\boldsymbol f_\tau,
	\]
where the parenthesized product is $I_{2g}$ when $j=n_\tau-1$.

For $0\leq j\leq n_\tau$, put $X_j\coloneq X^{\tau^j}$, given by the equation
$Q_j\coloneq Q^{\tau^j}=0$.
If $\beta=F_\tau^*(y)$ (see \cref{proposition:expressionoffrobeniuslift}), its coefficient twist $\beta_j=\beta^{\tau^j}$ defines a $K$-morphism, on the relevant strict neighborhoods,
\[G_j\colon X_j\longrightarrow X_{j+1},\qquad (x,y)\longmapsto(x^{q_\tau},\beta_j(x,y)).\]
Since $X_{n_\tau}=X$, the composite is $F_q$ on relevant strict neighborhoods, $X\to X$. By \cref{rem:evaluation-of-beta-at-suitable-bad-points}, $F_q(P), F_q(R)$ are well-defined since they are either good or suitable bad.

Write
\[\mathbf I=\int_P^R\boldsymbol\omega,\qquad \mathbf t_P=\int_P^{F_q(P)}\boldsymbol\omega,\qquad \mathbf t_R=\int_{F_q(R)}^R\boldsymbol\omega.\]
The map $F_q$ acts as the $q$-power Frobenius on the special fiber,
so $P$ and $F_q(P)$ lie in the same residue disc, as do $R$
and $F_q(R)$. Thus $\mathbf t_P$ and $\mathbf t_R$ are vectors of
tiny integrals. Functoriality and the displayed formula for
$F_q^*\boldsymbol\omega$ give
\[
\mathbf I-\mathbf t_P-\mathbf t_R
=\int_{F_q(P)}^{F_q(R)}\boldsymbol\omega
=\boldsymbol g_q(R)-\boldsymbol g_q(P)+A_q\mathbf I,
\]
and therefore
\[(A_q-I_{2g})\mathbf I=\boldsymbol g_q(P)-\boldsymbol g_q(R)-\mathbf t_P-\mathbf t_R.\]
The eigenvalues of $A_q$ have complex absolute value $q^{1/2}>1$,
so $A_q-I_{2g}$ is invertible and the system determines $\mathbf I$.

Notice that, for $S=P$ or $R$, the values $\boldsymbol g_q(S)$ can be obtained without forming the iterated functions. Let $S_0=S$ and $S_{j+1}=G_j(S_j)\in X_{j+1}(\overline K)$. For any function $h\in \mathcal R^\dagger$, the construction gives $((F_\tau^*)^jh)(S)=h^{\tau^j}(S_j)$ whenever these evaluations converge. Set $V_0(S)=0$ and compute
\[V_{j+1}(S)=\boldsymbol f_\tau^{\tau^j}(S_j)+\tau^j(A_\tau)V_j(S), \qquad 0\leq j<n_\tau.\]
Then $S_{n_\tau}=F_q(S)$ and $V_{n_\tau}(S)=\boldsymbol g_q(S)$.
For example, $V_2(S)=\boldsymbol f_\tau^\tau(S_1) +\tau(A_\tau)\boldsymbol f_\tau(S)$.
At a suitable bad point $S$, each one-step correction function converges along this orbit, by \cref{prop:convergentdomainoffinitereduction}.
\end{art}

The construction above requires good or suitable bad endpoints. We therefore give an algorithm to find a suitable bad point in the residue disc of a given bad point.

\begin{lemma}\label{lem:suitable-bad-point-criterion}
	Let $P\in X(K)$ be very bad, put $\widetilde P=\Sp(P)$, and
	let $e_{\widetilde P}$ be the ramification index of the reduction map
	$\widetilde x\colon\mathcal X_s\to\P^1_{\widetilde K}$ at $\widetilde P$. Choose the local parameter $t$
	of \cref{prop:uniformizeratapoint}. For an integer $e\geq1$,
	let $S_e$ be the point obtained by evaluating its local
	expansions at $t=\pi^{1/e}$.
	\begin{enumeratea}
		\item\label{lem:suitable-bad-point-criterion finite} If $P$ is finite, then
		\[
		S_e\text{ is suitable bad}
		\quad\Longleftrightarrow\quad
		e>qe_{\widetilde P},
		\]
		and in this case $v_K(r(x(S_e)))=e_{\widetilde P}/e$.
		\item\label{lem:suitable-bad-point-criterion infinite} If $P$ is infinite, then $S_e$ is suitable bad for every $e\geq1$, and $v_K(1/x(S_e))=e_{\widetilde P}/e$. In particular, $S_1\in X(K)$ is suitable bad.
	\end{enumeratea}
\end{lemma}
\begin{proof}
	Put $u(t)=x(t)-x(P)$ in the finite case and $u(t)=1/x(t)$ in the infinite case. Write $u(t)=\sum_{i\geq1}a_it^i$.
	The reduction of $u(t)$ has $t$-order $e_{\widetilde P}$.
	Hence the $a_i$ are integral, $a_i$ is divisible by $\pi$ for $i<e_{\widetilde P}$, and $a_{e_{\widetilde P}}$ is a unit.

	\ref{lem:suitable-bad-point-criterion finite} If $e>qe_{\widetilde P}$, the term $a_{e_{\widetilde P}}\pi^{e_{\widetilde P}/e}$ has valuation $e_{\widetilde P}/e<1$. All preceding terms have larger valuation, as do all subsequent terms. Thus $v_K(u(\pi^{1/e}))=e_{\widetilde P}/e$.
	The distinct reductions of the roots of $r$ give $r(x)=(x-x(P))h(x)$ with $h(x(S_e))$ a unit.
	This proves the stated valuation formula and the strict bound $<1/q$.
	If $e\leq qe_{\widetilde P}$, every term of $u(\pi^{1/e})$ has valuation at least $1/q$, and the same unit factor shows that the strict suitability bound fails.

	\ref{lem:suitable-bad-point-criterion infinite} By \cref{rem:ramification-at-good-very-bad}, $u(t)=t^{e_{\widetilde P}}A(t)$ with $A(t)\in K^\circ[[t]]^\times$. Thus $A(\pi^{1/e})$ is a unit, and
	\[
	v_K(1/x(S_e))=\frac{e_{\widetilde P}}{e}>0
	\]
	for every $e\geq1$. Hence $S_e$ is infinite but not very infinite, so it is suitable bad.
\end{proof}

\begin{algorithm}[Find a suitable bad point close to the boundary]
	\label[palgm]{algthm:suitablebadpoint}	
	\
	
	\textbf{Input}: A very bad point $P\in X(K)$ given by $x(P),\boldsymbol{b}^0(P)$ (resp. $1/x(P), \boldsymbol{b}^\infty(P)$ if $P$ is infinite), and optionally an integer $e$.
	
	\textbf{Output}: A suitable bad point $S\in X(K(\pi^{1/e}))$, given by $x(S),\boldsymbol{b}^0(S)$ (resp. $1/x(S),\boldsymbol{b}^\infty(S)$ if $P$ is infinite). 
	
	\begin{enumerate1}
		\item \label{algthm:suitablebadpoint 1} Compute $e_{\widetilde P}$ as the ramification index of $\widetilde x\colon\widetilde X\to\P^1_{\widetilde K}$ at $\widetilde P\coloneq\Sp(P)$, and determine a uniformizer $t$ for $P$ by \cref{prop:uniformizeratapoint}.
		\item \label{algthm:suitablebadpoint 2} Check that $e>qe_{\widetilde P}$; if $e$ is not given or fails this check, put $e\coloneq 2qe_{\widetilde P}+1$.
		\item \label{algthm:suitablebadpoint 3} Apply \cref{algthm:expansionofafunctionatagoodpoint} to compute $x(t), b_i^0(t)$ (resp. $1/x(t), b_i^\infty(t)$) to the required $t$-precision. Set $t_0=\pi^{1/e}$, and let $S$ be given by $x(t_0), b_i^0(t_0)$ (resp. $1/x(t_0), b_i^\infty(t_0)$).
	\end{enumerate1}
	
\end{algorithm}

\begin{remark}\label{rmk:e}
	In our code, for integrals with two endpoints $P_1, P_2$ whose reductions are $\widetilde{P_1}, \widetilde{P_2}$, respectively, if $e$ is not given or does not pass  \ref{algthm:suitablebadpoint 2}, we will take $e=2q\max\{e_{\widetilde{P_1}},e_{\widetilde{P_2}}\}+1$ by default. Notice that if $e$ is large, we might have less precision loss by \cref{prop:accurateforfi}~\ref{prop:accurateforfi1}.
\end{remark}

\begin{remark}
	If $P$ is bad but not very bad, it is not hard to find the corresponding very bad point $P'$, then we can apply the algorithm to $P'$. 
\end{remark}

With the idea demonstrated in \cref{idea of computing integral in general}, we first compute the Frobenius action on the cohomology basis and the corresponding correction functions.

\begin{algorithm}[Matrices of $F_\tau^*$ and $F_q^*$ on cohomology]
	\label[palgm]{algthm:semilinear-and-linear-frobenius-matrices}
	\
	
	\textbf{Input}: 
	The basis vector $\boldsymbol\omega$ defined above from \cref{algthm:cohomologybasis}.
	
	\textbf{Output}: $A_\tau$, $A_q\in M_{2g}(K)$ and, when needed,
	correction functions $\boldsymbol f_\tau$, $\boldsymbol g_q$ satisfying
	\begin{align*}
		F_\tau^*\boldsymbol\omega&=d\boldsymbol f_\tau+A_\tau\boldsymbol\omega,\\
		F_q^*\boldsymbol\omega&=d\boldsymbol g_q+A_q\boldsymbol\omega,
		\qquad F_q^*=(F_\tau^*)^{n_\tau}.
	\end{align*}

	\begin{enumerate1}
		\item Compute the function matrix $M$ by
		\cref{algthm:matrixoffrobeniusonfunctions}, and compute each
		$F_\tau^*\omega_i$ using \eqref{frobeniusofadifferentialform}. In
		particular, every coefficient $w_i(x)$ is replaced by
		$w_i^\tau(x^{q_\tau})$, and $dx$ by
		$q_\tau x^{q_\tau-1}dx$.
		\item Apply \cref{algthm:reductionprocedure} to these forms. Set
		$f_{\tau,i}=f_{i,0}+f_{i,\infty}+f_{i,end}$ and record
		\[
		F_\tau^*\omega_i=df_{\tau,i}
		+\sum_{j=1}^{2g}(A_\tau)_{ij}\omega_j.
		\]

		\item Set $B_0=I_{2g}$ and $\boldsymbol g_0=0$. For
		$r=0,\dots,n_\tau-1$, compute
		\begin{equation*}\label{eq:frobenius-iteration}
			B_{r+1}=\tau^r(A_\tau)B_r,\qquad
			\boldsymbol g_{r+1}=(F_\tau^*)^r\boldsymbol f_\tau
			+\tau^r(A_\tau)\boldsymbol g_r.
		\end{equation*}
		Omit the function recurrence if only the matrix is required. Return
		\begin{equation}\label{eq:linear-frobenius-matrix}
			A_q=B_{n_\tau}
			=\tau^{n_\tau-1}(A_\tau)\cdots\tau(A_\tau)A_\tau,
			\qquad \boldsymbol g_q=\boldsymbol g_{n_\tau}.
		\end{equation}
	\end{enumerate1}
\end{algorithm}

\begin{algorithm}[$q$-power Frobenius of a point and correction values]
	\label[palgm]{algthm:frobeniusofapoint}
	\
	
	\textbf{Input}: A good or suitable bad point $P\in X(K(\pi^{1/e}))$, and the small-lift data, $A_\tau$,
	$\boldsymbol f_\tau$ from
	\cref{algthm:semilinear-and-linear-frobenius-matrices}.
	
	\textbf{Output}: $F_q(P)$ and $\boldsymbol g_q(P)$.
	
	\begin{enumerate1}
		\item Set $P_0=P$ and $V_0=0$.
		\item For $j=0,\dots,n_\tau-1$, evaluate
		\begin{equation}\label{eq:frobenius-point-evaluation}
			P_{j+1}=G_j(P_j),\qquad
			V_{j+1}=\boldsymbol f_\tau^{\tau^j}(P_j)+\tau^j(A_\tau)V_j.
		\end{equation}
		If only $F_q(P)$ is needed, omit the vector recurrence.
		At a good point this means taking $x(P_{j+1})=x(P_j)^{q_\tau}$ and the Hensel root of $Q_{j+1}(x(P_{j+1}),Y)$ with residue $y(P_j)^{q_\tau}$. Reconstruct the integral coordinates using $(W^0)^{\tau^{j+1}}$.
		At finite bad or infinite points, evaluate the twisted lift on the integral coordinates $\boldsymbol{b}^0$ or $\boldsymbol{b}^\infty$, respectively; at infinity use $1/x(P_{j+1})=(1/x(P_j))^{q_\tau}$. Use the branch determined by the overconvergent lift. A residue condition alone may not distinguish roots in a bad fiber.
		\item We have $F_q(P)=P_{n_\tau}$ and $\boldsymbol g_q(P)=V_{n_\tau}$.
	\end{enumerate1}
\end{algorithm}

\begin{remark}\label{rem:iterate-convergence}
	At a finitely bad endpoint write $r_j=r^{\tau^j}$ and $\varepsilon=v_K(r(x(P)))$. Since $r_{j+1}(x^{q_\tau})\equiv r_j(x)^{q_\tau}\pmod\pi$, the condition $0<\varepsilon<1/q$ implies
	\[v_K(r_j(x(P_j)))=q_\tau^j\varepsilon<1/q_\tau\qquad(0\leq j<n_\tau).\]
	If $P$ is infinite and not very infinite, then $v_K(1/x(P_j))=q_\tau^jv_K(1/x(P))$ is positive and finite, so every $P_j$ is again infinite and not very infinite, and no bound is needed.
	Consequently the conditions of \cref{classification of points def} ensure convergence at every intermediate step, by \cref{prop:convergentdomainoffinitereduction,rem:evaluation-of-beta-at-suitable-bad-points}.
\end{remark}

\begin{algorithm}[Non-tiny integrals on basis]
	\label[palgm]{algthm:integralonbasis}
	\
	
	\textbf{Input}: Points $P, R\in X(K)$ in different residue discs, and the basis $\omega_1,\dots, \omega_{2g}$ of $H^1_{\mathrm{dR}}(X)$ given by \cref{algthm:cohomologybasis}, with every $\omega_i$ regular at both $P$ and $R$.
	
	\textbf{Output}: The values $\int_P^R\omega_i$ for $i=1,\dots,2g$.
	
	\begin{enumerate1}
		\item Find very bad points: If $P$ (resp. $R$) is good, set $P'=P$ (resp. $R'= R$); If $P$ (resp. $R$) is bad, set $P'\in X(K)$ (resp. $R'\in X(K)$) the corresponding very bad point.

 \item Find suitable points: If $P'$ (resp. $R'$) is good, set $P''=P'$ (resp. $R''=R'$); If $P'$ (resp. $R'$) is bad (hence very bad), apply \cref{algthm:suitablebadpoint} to find a suitable bad point $P''$, (resp. $R''$).	

		\item Apply \cref{algthm:tinyintegral} to compute the tiny integrals $\int_P^{P''}\omega$ and $\int_{R''}^R\omega$ for $\omega=\omega_i$, $i=1,\dots,2g$.
		\item Compute $\int_{P''}^{R''}\omega$ for all $\omega=\omega_1,\dots,\omega_{2g}$ in the following steps. 
		\begin{enumerate}[leftmargin=*,label=\it(4.\arabic*),ref=\it{(4.\arabic*)}] 
			\item Apply \cref{algthm:semilinear-and-linear-frobenius-matrices} to obtain the small-step data and $A_q$ satisfying
			\[F_q^*\omega_i=dg_{q,i}+\sum_{j=1}^{2g}(A_q)_{ij}\omega_j\]
			for $i=1,\dots, 2g$.
			\item Apply \cref{algthm:frobeniusofapoint} to calculate $F_{q}(P'')$ and $F_{q}(R'')$, and apply \cref{algthm:tinyintegral} to compute tiny integrals $\int_{P''}^{F_q(P'')}\omega_i$ and $\int_{F_q(R'')}^{R''}\omega_i$ for $i =1,\cdots, 2g$. 
			\item Use \cref{algthm:frobeniusofapoint} to compute $g_{q,i}(P''), g_{q,i}(R'')$ for $i=1,\dots, 2g$ and use the system of equations
			\begin{align}
				\sum_{j=1}^{2g}(A_q-I)_{ij}\left(\int_{P''}^{R''}\omega_j\right)=g_{q,i}(P'')-g_{q,i}(R'') - \int_{P''}^{F_{q}(P'')}\omega_i-\int_{F_{q}(R'')}^{R''}\omega_i
				\label{equation:linearequationfromfrobenius}
			\end{align}
			to solve all $\int_{P''}^{R''}\omega_i$.
		\end{enumerate}
		\item We have $\int_P^R\omega=\int_P^{P''}\omega+\int_{P''}^{R''}\omega+\int_{R''}^R\omega$ for $\omega=\omega_i$, $i=1,\dots,2g$.
	\end{enumerate1}
\end{algorithm}

\begin{remark}\label{rem:basis-regular-at-endpoints}
	Throughout the integration algorithms, we assume that all basis differentials $\omega_1,\dots,\omega_{2g}$ are regular at both endpoints.
	For the basis of \cref{algthm:cohomologybasis}, this holds whenever neither endpoint is very infinite: the forms are regular on $U$, and at each finite point of $X\setminus U$ they have at most a simple pole, since they belong to $E_0$; their residues vanish because they are of the second kind.
	The condition defining $E_\infty$ bounds the pole orders above infinity but does not exclude higher-order poles there. Thus regularity of every basis form at a very infinite endpoint remains an additional assumption.
\end{remark}

\begin{remark}
	Our choices of $P'',R''$ ensure that $g_{q,i}(P'')$ and $g_{q,i}(R'')$ converge, by \cref{rem:iterate-convergence}. 
\end{remark}

\begin{algorithm}[Integrals in general]
	\
	
	\textbf{Input}: Points $P, R\in X(K)$, and a differential form $\omega = \left(\sum_{i=0}^{d_x-1}w_ib_i^0\right)\frac{dx}{r}\in \Omega_{X/K}^1(U)$ of the second kind, regular at $P$ and $R$.
	
	\textbf{Output}: The value $\int_P^R\omega$.
	
	\begin{enumerate1}

		\item (Integral of basis) Apply \cref{algthm:tinyintegral} or \cref{algthm:integralonbasis} to calculate $\int_{P}^{R}\omega_i$, $i=1,\dots, 2g$.
		
		\item (Reduction) Apply \cref{algthm:reductionprocedure} to calculate $f\coloneq f_0+f_\infty+f_{end}\in \sum_{i=0}^{d_x-1}K[x,1/r]b_i^0$ and $\omega'=\sum_{i=1}^{2g}a_i\omega_i$ such that 
		$\omega=df+\omega'.$
		\item Since $\omega$ and all $\omega_i$ are regular at the endpoints, so is $df=\omega-\sum_i a_i\omega_i$, and hence $f$ has no pole there. Compute $f(R)-f(P)$ and
		\[\int_P^R\omega=f(R)-f(P)+\sum_{i=1}^{2g}a_i\int_{P}^{R}\omega_i.\]
	\end{enumerate1}
\end{algorithm}

\begin{remark}
	In our code for \cref{algthm:reductionprocedure}, we have to require $\omega$ to be integral to control the precision. 
\end{remark}

\section{Precision bounds}

\label{precision bounds}

In 	{\sffamily Magma}, an element $x\in K$ is stored as $x=\pi^vu+O(\pi^{v+N_0})$, where $v$ is the valuation of $x$, and $u$ is a unit known to precision $N_0$ in $K^\circ$, i.e. the image of $u$ in $K^\circ/\pi^{N_0}$. We call $N_0$ the \emph{relative precision} and $v+N_0$ the \emph{absolute precision}. By precision, we will always mean absolute precision. 

Given a positive integer $N\in \N_{\geq 1}$, we will store points in $X(K(\pi^{1/e}))$ accurate to (absolute) precision $N$. We want to compute the precision of the integrals $\int_P^R\omega_i$ from \cref{algthm:integralonbasis}, i.e. the precision of tiny integral $\int_P^{P''}\omega_i, \int_R^{R''}\omega_i$ and the precision of the solution of \eqref{equation:linearequationfromfrobenius} in \cref{algthm:integralonbasis}. The former is given by \cref{prop:precisionfortinyintegral} and the latter is given by \cref{prop:precisionofsolutionoflinearsystem}. 

We show some basic properties of precision. Recall that, for a matrix $A=(A_{ij})$ over $K$, we put $v_K(A)=\min_{i,j}v_K(A_{ij})$ and measure absolute precision entrywise. The estimates below also hold over any finite extension $K'/K$, with
$v_K$ extended to $K'$.

\begin{lemma}\label[plemma]{lemma:precisionofbasicoperation}
	\begin{enumerate1}
		\item \label{lemma:precision basic 1} Let $A$, $B$, $A_1, \dots, A_s$ be matrices over $K$, accurate to absolute precisions $N$, $N'$, $N_1, \dots, N_s$, respectively, with compatible sizes for the indicated operations.
		\begin{enumerate}[leftmargin=*,label=\it(1.\arabic*),ref=\it{(1.\arabic*)}]
			\item \label{lemma:precision basic 1.i}
			The matrices $A+B$ and $A-B$ are correct to precision $\min\{N,N'\}$.
			
			\item \label{lemma:precision basic 1.ii}
			If $A\in\GL_n(K)$ and $N+(n-1)v_K(A)>v_K(\det A)$, then the inverse of the approximation is defined and $A^{-1}$ is correct to precision
			\[N-2v_K(\det A)+2(n-1)v_K(A).\]
			In particular, for $n=1$, i.e. $A=a\in K^\times$, the inverse $a^{-1}$ is correct to precision $N-2v_K(a)$, provided $N>v_K(a)$.
			
			\item \label{lemma:precision basic 1.iii}
		The product $A_1\cdots A_s$  is correct to precision
			\[\min_{\emptyset\neq J\subseteq\{1,\ldots,s\}}\left\{\sum_{i\in J}N_i+\sum_{h\notin J}v_K(A_h)\right\}.\]
			In particular, if $N_i\geq v_K(A_i)$ for every $i$, their product is correct to precision
			\[\min_{1\leq i\leq s}\left\{N_i+\sum_{h\neq i}v_K(A_h)\right\}.\]
		\end{enumerate}
		
		\item \label{lemma:precision basic 2} Let $A(x)\in M_{m\times n}(K[x])$ and $z\in K^\circ$ both be accurate to precision $N\geq0$. Then $A(z)$ is correct to precision $\min\{N,N+v_K(A(x))\}$, where $v_K(A(x))$ is the minimum valuation of the coefficients of all entries of $A(x)$.
	\end{enumerate1}
\end{lemma}

\begin{proof}
	\ref{lemma:precision basic 1.i}
	In each entry, the error of $A\pm B$ is a sum or difference of errors with valuations at least $N$ and $N'$. Its valuation is therefore at least $\min\{N,N'\}$.
	
	\ref{lemma:precision basic 1.iii}
	Write the approximations of $A_i$ as $A_i+E_i$, where $v_K(E_i)\geq N_i$. For each $J\subseteq\{1,\ldots,s\}$, put
	\[C_{i,J}=
	\begin{cases}
		E_i,& i\in J,\\
		A_i,& i\notin J.
	\end{cases}\]
	Then
	\[(A_1+E_1)\cdots(A_s+E_s)-A_1\cdots A_s=\sum_{\emptyset\neq J\subseteq\{1,\ldots,s\}}C_{1,J}\cdots C_{s,J}.\] 
	Since
	$v_K(CD)\geq v_K(C)+v_K(D)$, this term has valuation
	at least
	\[\sum_{i\in J}N_i+\sum_{h\notin J}v_K(A_h).\]
	Taking the minimum over $J$ proves the general bound.
	
	If $N_i\geq v_K(A_i)$ for every $i$, then for any
	$j\in J\not=\emptyset$,
	\[\sum_{i\in J}N_i+\sum_{h\notin J}v_K(A_h)\geq N_j+\sum_{h\neq j}v_K(A_h).\]
	The stated particular case follows.
	
	\ref{lemma:precision basic 1.ii} First suppose $n=1$. If	$v_K(\varepsilon)\geq N>v_K(a)$, then $v_K(a+\varepsilon)=v_K(a)$, and
	\[	(a+\varepsilon)^{-1}-a^{-1}	=-\frac{\varepsilon}{a(a+\varepsilon)}	\]
	has valuation at least $N-2v_K(a)$. This proves the scalar inverse rule, which we will use for the determinant.
	
	Suppose $n\geq 2$, and let $A^*$ denote the adjugate of $A$. The Leibniz formula for determinant gives
	\begin{align}\label{eq:leibniz formula}
\det A=
\sum_{\sigma\in S_n}\operatorname{sgn}(\sigma)
\prod_{i=1}^n A_{i,\sigma(i)}.
	\end{align}
	Each product has valuation at least $nv_K(A)$.
	Consequently $v_K(\det A)\geq nv_K(A)$, and the hypothesis
	$N+(n-1)v_K(A)>v_K(\det A)$ implies $N>v_K(A)$.
	
Every entry of $A$ is accurate to precision $N$ and
has valuation at least $v_K(A)$. Applying the general
product estimate in \ref{lemma:precision basic 1.iii}
to each product in \eqref{eq:leibniz formula} gives precision
at least
\begin{align*}
	\min_{\emptyset\neq J\subseteq\{1,\ldots,n\}}
	\bigl\{|J|N+(n-|J|)v_K(A)\bigr\} =
	N+(n-1)v_K(A).
\end{align*}
Thus $\det A$ is correct to precision $N+(n-1)v_K(A)$.
	
 Each entry of $A^*$ is a signed sum of	products of $n-1$ entries of $A$. The same argument shows	that $A^*$ is correct to precision $N+(n-2)v_K(A)$ and that	$v_K(A^*)\geq(n-1)v_K(A)$.
	
Because the precision of $\det A$ is at least $N+(n-1)v_K(A)>v_K(\det A)$, the scalar inverse rule (i.e. $n=1$ case) therefore shows that $(\det A)^{-1}$ is correct to precision $N+(n-1)v_K(A)-2v_K(\det A)$ and has valuation $-v_K(\det A)$.
	
Finally, $A^{-1}=(\det A)^{-1}A^*$. Applying \ref{lemma:precision basic 1.iii} with two factors
gives precision at least
\begin{align*}
	\min\{&
	N+(n-2)v_K(A)-v_K(\det A),\\
	&
	N+2(n-1)v_K(A)-2v_K(\det A),\\
	&
	2N+(2n-3)v_K(A)-2v_K(\det A)\}\\
	={}&N+2(n-1)v_K(A)-2v_K(\det A).
\end{align*}
Indeed, the first term minus the claimed bound is $v_K(\det A)-nv_K(A)\geq0$, and the third minus it is $N-v_K(A)>0$. In particular, the approximation to $A$ is invertible, and its inverse has the asserted precision.
	
	\ref{lemma:precision basic 2}
	Write $A(x)=\sum_i C_ix^i$. Let $C_i+E_i$ and $z+\varepsilon$ be approximations satisfying $v_K(E_i),v_K(\varepsilon)\geq N$. Since $z\in K^\circ$ and $N\geq0$, the element $z+\varepsilon$ is integral. The evaluation error is
	\[\sum_i E_i(z+\varepsilon)^i+\sum_{i\geq1}C_i\bigl((z+\varepsilon)^i-z^i\bigr).\]
	The first sum has valuation at least $N$. For $i\geq1$,
	\[(z+\varepsilon)^i-z^i = \varepsilon\sum_{j=0}^{i-1} (z+\varepsilon)^{i-1-j}z^j\]
	has valuation at least $N$. Thus the second sum has valuation at least $N+v_K(A(x))$. The evaluation error therefore has valuation at least $\min\{N,N+v_K(A(x))\}$.
\end{proof}
\begin{remark}
	In the finite-precision local fields of {\sffamily Magma}, the valuation returned by \texttt{Valuation} is at most
	the absolute precision.
	Thus, when using the reported valuations and precisions, the condition $N_i\geq v_K(A_i)$ in \ref{lemma:precision basic 1.iii} holds automatically and need not be checked in the code.
	The precision condition for inversion in \ref{lemma:precision basic 1.ii} must still be checked.
\end{remark}

For the tiny integral, we have the result as \cite[Proposition~4.1]{balakrishnan2020explicit}.

\begin{proposition}\label[pprop]{prop:precisionfortinyintegral}
	Let $P,R\in X(K(\pi^{1/e}))$ be two points in the same
	residue disc accurate to precision $N\geq1$. Let $t$ be a
	uniformizer at the point $P'$ from \cref{algthm:tinyintegral} ($N$ is the absolute precision of $t(P),t(R)$). Put
	\[
	v_{\min}\coloneq
	\min\{v_K(t(P)),v_K(t(R)),N\},
	\]
	with $v_K(0)=+\infty$.
	
	Let $\omega=g(t)dt$ be a nonzero differential of the second
	kind with
	\[
	g(t)=a_{-k}t^{-k}+a_{-k+1}t^{-k+1}+\cdots
	\in K^\circ[[t]][t^{-1}]
	\]
	for some integer $k$, with $a_{-k}\neq0$ (notice that $k\not=1$ since $\omega$
	is of the second kind).
	If $k\geq2$, assume that $t(P)t(R)\neq0$ and
	\[
	N>\max\{v_K(t(P)),v_K(t(R))\}.
	\]
	If $g$ is accurate to $\pi$-adic precision $N$ and truncated
	modulo $t^l$ for $l\geq1$, then the tiny integral
	$\int_P^R\omega$ computed in \cref{algthm:tinyintegral}
	is correct to precision $\min\{\nu_1,\nu_2,\nu_3\}$, where
	\begin{align*}
		\nu_1
		&=v_{\min}+\min_{i\geq l}
		\{iv_{\min}-e_K\lfloor\log_p(i+1)\rfloor\},\\
		\nu_2
		&=N+\min_{0\leq i\leq l-1}
		\{iv_{\min}-e_K\lfloor\log_p(i+1)\rfloor\},\\
		\nu_3
		&=
		\begin{cases}
			N-k\max\{v_K(t(P)),v_K(t(R))\}
			-e_K\lfloor\log_p(k-1)\rfloor,
			&k\geq2,\\
			+\infty,
			&k\leq0.
		\end{cases}
	\end{align*}
\end{proposition}

\begin{proof}
	Notice that $a_{-1}=0$ since $\omega$ is of the second kind.
	Thus
	\begin{equation}\label{eq:integralfortinyintegral}
		\int_P^R\omega
		=\int_{t(P)}^{t(R)}g(t)\,dt
		=\sum_{\substack{i\geq-k\\i\neq-1}}
		\frac{a_i}{i+1}
		\bigl(t(R)^{i+1}-t(P)^{i+1}\bigr).
	\end{equation}
	Since $P,R$ both lie in the residue disc $|t|<1$,
	we have $v_{\min}\geq1/e$.
	
	Firstly, we bound the error coming from omitting the terms
	with $i\geq l$. Since the coefficients $a_i$ are integral,
	\[
	v_K(t(P)^{i+1}),v_K(t(R)^{i+1})
	\geq(i+1)v_{\min},
	\]
	and
	\[
	v_K(i+1)=e_K\ord_p(i+1)
	\leq e_K\lfloor\log_p(i+1)\rfloor.
	\]
	Therefore, the valuation of this error is at least $\nu_1$.
	
	Next, we bound the error coming from terms with $0\leq i\leq l-1$.
	Since $N\geq v_{\min}$, by \cref{lemma:precisionofbasicoperation}~\ref{lemma:precision basic 1.iii}, the powers $t(P)^{i+1},t(R)^{i+1}$ are correct to precision
	$N+iv_{\min}$, so is $t(R)^{i+1}-t(P)^{i+1}$.
	Moreover, $t(R)^{i+1}-t(P)^{i+1}$ has valuation at least
	$(i+1)v_{\min}$. By \cref{lemma:precisionofbasicoperation}~\ref{lemma:precision basic 1.iii} again, $a_i(t(R)^{i+1}-t(P)^{i+1})$ is correct to precision $N+iv_{\min}$. Dividing by $i+1$ and taking the minimum over these indices gives $\nu_2$.
	
	Finally, suppose that $k\geq2$. We bound the error coming
	from terms with $-k\leq i\leq-2$.
By \cref{lemma:precisionofbasicoperation}~\ref{lemma:precision basic 1.ii}~\ref{lemma:precision basic 1.iii}, the loss of precision  of an inversion is $2$ times the valuation, so $t(P)^{i+1}$ (resp. $t(R)^{i+1}$) is accurate to precision $$N+(-i-2)v_K(t(P))+2(i+1)v_K(t(P)) = N+iv_K(t(P)) \text{ (resp. $N+iv_K(t(R))$)},$$ 
then $\frac{a_i}{i+1}(t(P)^{i+1}-t(R)^{i+1})$ is accurate to precision 
\begin{align*}
	&\min\{v_K(a_i)+N+\min\{iv_K(t(P)),iv_K(t(R))\},N+v_K(t(P)^{i+1}-t(R)^{i+1})\} -v_K(i+1)\\
	\geq& N-(-i)\max\{v_K(t(P)),v_K(t(R))\}-e_K\lfloor\log_p(|i+1|)\rfloor.
\end{align*}
Hence $\sum_{i=-k}^{-2}\frac{a_i}{i+1}(t(R)^{i+1}-t(P)^{i+1})$ is accurate to precision 
$\nu_3$. 

	If $k\leq0$, there are no negative-power terms, so
	$\nu_3=+\infty$.
	
	Combining the three bounds completes the proof.
\end{proof}

\begin{remark}
	Once $P,R,\omega$ and $N$ are fixed, $\nu_3$ can be calculated when $k\geq2$; for $k\leq1$ there are no negative-power terms. Notice that, for fixed $v_{\min}$, $\nu_1$ depends only on $l$.
	In our implementation, $l$ is chosen by \texttt{tadicprec}, and the bounds above are applied to the truncation actually used.
	
	In our implementation, when $v_{\min}$ is not given, we take $v_{\min} = 1/e$. We also take $P$ to be good or very bad, so $t$ is a uniformizer of $P$ and $v_K(t(P))=+\infty$.
\end{remark}

In the rest of this section, we consider integrals that are not
tiny. Recall the spaces $E_0,E_\infty,B_0,B_\infty$ defined in
\cref{thm: computing de rham via E B}. Let
\[
\omega_i=\sum_{j=0}^{d_x-1}u_j(x)b_j^0\frac{dx}{r},
\qquad i=1,\dots,2g,
\]
be the integral representatives in $E_0\cap E_\infty$ of the
basis of $H^1_{\mathrm{dR}}(X)$ obtained in
\cref{algthm:cohomologybasis}. We first estimate the precision
of evaluating the one-step correction function
\[
f_{\tau,i}\coloneq f_{i,0}+f_{i,\infty}+f_{i,end}
\]
at a point where it converges. Here
$f_{i,0},f_{i,\infty},f_{i,end}\in\mathcal R_K^\dagger$
are obtained by applying \cref{algthm:reductionprocedure}
to $F_\tau^*\omega_i$.

The following bounds apply to the correction functions for
$F_\tau^*$ and their coefficient twists. The precision of
$\boldsymbol g_q$ is obtained by tracking the evaluations and
matrix products in \eqref{eq:frobenius-point-evaluation}.

\begin{lemma}\label[plemma]{lemma:frobenius of a form}
	Let $\omega=\sum_{i=0}^{d_x-1}w_ib_i^0\frac{dx}{r}\in E_0\cap E_\infty$ be an integral $1$-form. Then there are unique $w_{jk}\in K^\circ[x]$, $j\in\Z, 0\leq k\leq d_x-1$ with $\deg(w_{jk})< \deg(r)$ such that
	\begin{align}\label{expresionofFw1}
		F_\tau^*\omega = q_\tau\sum_{j\in\Z}\sum_{k=0}^{d_x-1}\frac{w_{jk}(x)}{r^j}b_k^0\frac{dx}{r}.
	\end{align}
	Moreover, for any $N\in \N$, if $j\geq q_\tau N$, then $v_K(w_{jk})\geq N$. 
\end{lemma}
\begin{proof}
	For the uniqueness, by \cref{rmk:differential form of Rdagger}, $\Omega_{\mathcal{R}^\dagger/K^\circ}^1 = \mathcal{R}^\dagger dx$  is free over $\mathcal{R}^\dagger$ of rank $1$, and $\mathcal{R}^\dagger=\oplus_{k=0}^{d_x-1}\mathcal{S}^\dagger b_k^0$ is free over $\mathcal{S}^\dagger$. So $F_\tau^*\omega$ can be expressed uniquely as a linear combination of $b_0^0,\dots, b_{d_x-1}^0$ over $\mathcal{S}^\dagger$. Then by the uniqueness of $r$-adic expansion in $\mathcal{S}^\dagger$, $w_{jk}$ are unique.
	
	By \cref{prop:approximationofFrobenius}~\ref{frobenius of b}, there are $\psi_{i,k,j}\in K^\circ[x]$ such that $$\deg(\psi_{i,k,j})\leq\max\{-\ord_\infty(W)-q_\tau\ord_\infty(W^{-1})-(q_\tau-1)\deg(r), \deg(r)-1\},$$
	\[F_\tau^*\left(\frac{b_i^0}{r}\right)= \sum_{k=0}^{d_x-1}\sum_{j=0}^{\infty}\frac{\psi_{i,k,j}}{r^j}\cdot\frac{b_k^0}{r},\]
	and for any $N\in\N_{\geq 1}$ and $j\geq q_\tau N$, $\psi_{i,k,j}\equiv 0 \ \  (\mod \pi^N\mathcal{R}^\dagger)$. Notice that $r\in K^\circ[x]$ with leading coefficient in $(K^\circ)^\times$ (see \cref{def:triplehasgoodreduction}), 
	then for any $j\geq0$ and $0\leq k\leq d_x-1$, there are $a_{jkl}\in K^\circ[x]$ (by Euclidean algorithm in $K^\circ[x]$), $l\in \N$ such that $\deg(a_{jkl})<\deg(r)$ and
	\[\left(x^{q_\tau-1}\sum_{i=0}^{d_x-1}w_i^\tau(x^{q_\tau})\psi_{i,k,j}(x)\right) = \sum\limits_{l=0}^{\infty}a_{jkl}r^l.\]
	Notice that $a_{jkl} = 0$ when
	$$l\deg(r)> q_\tau-1+q_\tau\max\{\deg(w_i)\}+\max\{-\ord_\infty(W)-q_\tau\ord_\infty(W^{-1})-(q_\tau-1)\deg(r), \deg(r)-1\}.$$
	So $a_{jkl}=0$ for $l$ large enough (independent of $j,k$).
	For $j\in\Z, 0\leq k\leq d_x-1$, we set $w_{jk} = \sum_{s\geq\max\{0,j\}}a_{sk(s-j)}$ which is a finite sum and $\deg(w_{jk})<\deg(r)$. Then
	\begin{align*}
		F_\tau^*\omega &= \sum_{i=0}^{d_x-1}q_\tau x^{q_\tau-1}w_i^\tau(x^{q_\tau})F_\tau^*\left(\frac{b_i^0}{r}\right)dx \nonumber\\
		&=q_\tau x^{q_\tau-1}\sum_{i=0}^{d_x-1}w_i^\tau(x^{q_\tau})\left(\sum_{k=0}^{d_x-1}\sum_{j=0}^{\infty}\frac{\psi_{i,k,j}(x)}{r^j}\frac{b_k^0}{r}\right)dx\nonumber\\
		&=q_\tau\sum_{k=0}^{d_x-1}\sum_{j=0}^\infty\frac{\sum_{l=0}^\infty a_{jkl}r^l}{r^j}\frac{b_k^0}{r} dx\\
		&=q_\tau\sum_{j\in\Z}\sum_{k=0}^{d_x-1}\frac{w_{jk}}{r^j}b_k^0\frac{dx}{r}.
	\end{align*}

	For any $N\in\N$, if $j\geq q_\tau N$, then $v_K\left(x^{q_\tau-1}\sum_{i=0}^{d_x-1}w_i^\tau(x^{q_\tau})\psi_{i,k,j}(x)\right)\geq N$, $0\leq k\leq d_x-1$ since $\omega\in \Omega_{\mathcal{U}/K^\circ}^1(\mathcal{U})$. Since the leading coefficient of $r$ is a unit in $K^\circ$, the division algorithm in $K^\circ[x]$ applies to $\pi^{-N}x^{q_\tau-1}\sum_{i=0}^{d_x-1}w_i^\tau(x^{q_\tau})\psi_{i,k,j}(x)$, giving an $r$-adic expansion with coefficients in $K^\circ[x]$, i.e. $\pi^{-N}a_{jkl}\in K^\circ[x]$. Then $v_K(a_{jkl})\geq N$ for any $l\in\N$. So $v_K(w_{jk})\geq N$. 
\end{proof}

\begin{lemma}
	\label[plemma]{lemma:estimateonfunctions}
	Let $\omega=\sum_{i=0}^{d_x-1}w_ib_i^0\frac{dx}{r}\in E_0\cap E_\infty$ be an integral $1$-form. We write
	\[F_\tau^*\omega=df_0+df_{\infty}+df_{end}+\omega'\]
	with $\omega'\in E_0\cap E_\infty$ after the reduction process, \cref{algthm:reductionprocedure}. Then the following statements hold.
	\begin{enumerate1}
		\item \label{lemma:estimateonfunctions f0} The function $f_{0}$ is of the form
		\begin{align}\label{expresionoff0}
			\sum_{k=0}^{d_x-1}\sum_{j=1}^\infty\frac{c_{jk}}{r^j}b_{k}^0,
		\end{align}
		where $c_{jk}\in K[x]$ are of degree $<\deg(r)$ and
		\begin{equation}\label{eq:finite-reduction-coefficient-bound}
			v_K(c_{jk})\geq\rho(j)\coloneq
			\min_{a\in\Z_{\geq j}}
			\left\{\left\lfloor\frac{a}{q_\tau}\right\rfloor
			+v_K(q_\tau)-e_K\lfloor\log_p(ae_0)\rfloor\right\}.
		\end{equation}
		In particular, $\rho(j)\geq j/q_\tau-O(\log j)$ as $j\to\infty$.
		\item \label{lemma:estimateonfunctions finfty} The function $f_{\infty}$ is of the form
		\begin{align}			\label{expresionoffinfty}
			\sum_{k=0}^{d_x-1}\sum_{j=0}^{\kappa_1}c_{jk}x^jb_k^0=\sum_{k=0}^{d_x-1}\sum_{j=\kappa_2}^{\kappa_3}d_{jk}x^jb_k^\infty,
		\end{align}
		where $c_{jk}, d_{jk}\in K$ and 
		\begin{align*}
			\kappa_1&\leq -\ord_\infty(W)-\min\{q_\tau(\ord_0(W)+1), (\ord_\infty(W^{-1})+1)\}, \\
			\kappa_2&\geq -\ord_0(W),\\
			\kappa_3&\leq -\min\{q_\tau(\ord_0(W)+1), \ord_\infty(W^{-1})+1\}.
		\end{align*}
		\item \label{lemma:estimateonfunctions fend} The function $f_{end}$ is of the form
		\begin{align}				\label{expresionoffend}
			\sum_{k=0}^{d_x-1}\sum_{j=0}^{\lambda_1}c_{jk}x^jb_k^0=\sum_{k=0}^{d_x-1}\sum_{j=\lambda_2}^{\lambda_3}d_{jk}x^jb_k^\infty,
		\end{align}
		where $c_{jk},d_{jk}\in K$ and 
		\begin{align*}
			\lambda_1&\leq -\ord_\infty(W)-\ord_0(W)-1,\\
			\lambda_2&\geq \ord_0(W^{-1}),\\
			\lambda_3&\leq -(\ord_0(W)+1).
		\end{align*}
	\end{enumerate1}
	
\end{lemma}
\begin{proof}
\ref{lemma:estimateonfunctions f0} By
\cref{lemma:frobenius of a form}, the coefficients of the
pole-order $a$ part in \eqref{expresionofFw1} (we take $N=\lfloor a/q_\tau\rfloor$) have valuations
at least $\lfloor a/q_\tau\rfloor+v_K(q_\tau)$.
Applying finite reduction to this part alone loses at most
$e_K\lfloor\log_p(ae_0)\rfloor$ digits and produces a
primitive with pole indices $1\leq j\leq a$, by
\cref{prop:valuationoffinitereduction}~
\ref{prop:valuationoffinitereduction1}.
Thus every coefficient of this primitive has valuation at least
\[
\left\lfloor\frac{a}{q_\tau}\right\rfloor
+v_K(q_\tau)-e_K\lfloor\log_p(ae_0)\rfloor.
\]
Finite reduction is linear, and only input parts with
pole orders $a\geq j$ can contribute to $c_{jk}$. Therefore \eqref{eq:finite-reduction-coefficient-bound} holds.
Moreover, the displayed lower bound tends to infinity with $a$,
so these contributions form a convergent sum.

	\ref{lemma:estimateonfunctions finfty} By separating $j>0$ and $j\leq0$ in \eqref{expresionofFw1}, finite reduction, base change and infinite reduction, we can write
	\begin{align}
		F_\tau^*\omega
		&=\left(df_{0}+\sum_{j>0}\sum_{i=0}^{d_x-1}\left(u_{jk}b_i^0\right)\frac{dx}{r}\right)+\sum_{k=0}^{d_x-1}u_kb_k^0\dfrac{dx}{r} \nonumber\\
		&=\left(df_{0}+\sum_{j>0}\sum_{i=0}^{d_x-1}\left(u'_{jk}b_i^\infty\right)\frac{dx}{r}\right)+\sum_{k=0}^{d_x-1}u'_kb_k^\infty\dfrac{dx}{r} \nonumber\\
		&=\left(df_0+df_{+\infty}+\sum_{j>0}\sum_{i=0}^{d_x-1}\left(v_{jk}b_k^\infty\right)\frac{dx}{r}\right)+\left(df_{-\infty}+\sum_{k=0}^{d_x-1}v_kb_k^\infty\dfrac{dx}{r}\right),
		\label{expresionofFw2}
	\end{align}
	where $u_k = q_\tau\sum_{j\leq 0} \frac{w_{jk}}{r^j}$. 
	
	Firstly, we consider the term with $j>0$. By \cref{prop:valuationoffinitereduction}~\ref{prop:valuationoffinitereduction2}, we know that $\ord_\infty(u'_{jk})\geq \ord_\infty(W^{-1})-\deg(r)+2$. Hence, by \cref{lemma:reductionprocedureforcohomologyclass}~\ref{reduction at infty}, the terms $d_{jk}x^j$ in \eqref{expresionoffinfty} from $f_{+\infty}$ have degree bounded from above by $-(\ord_\infty(W^{-1})+1)$. 
	
	Next, we consider the term with $j\leq 0$. We write
\begin{align*}
	F_\tau^*\omega=&q_\tau\sum_{k=0}^{d_x-1}\left(\sum_{j\leq 0}\frac{w_{jk}}{r^{j}} +\sum\limits_{i=1}^\infty\frac{w_{ik}}{r^i}\right)b_k^0\frac{dx}{r}\\
	=&\sum_{k=0}^{d_x-1}\left( u'_{k}+v'_k\right)b_k^\infty\frac{dx}{r}
\end{align*}
We view $\omega = \sum_{i=0}^{d_x-1}w'_ib_i^\infty\frac{dx}{r}$ as a rational section of $x_*\Omega_{X/K}$ around $\infty$.
By the definition of $E_0\cap E_\infty$, we have $\ord_\infty(w'_i)\geq \ord_0(W)-\deg(r)+2$. By a proof similar to that of \cref{prop:approximationofFrobenius}, we can show that the Frobenius structure on $R^0x_*\OO_{U}$ does not have a pole at $\infty$ with respect to the basis $b_0^\infty, \dots, b_{d_x-1}^\infty$. Moreover, $F_\tau^*(\frac{dx}{x})=q_\tau\frac{dx}{x}$. So the coefficients of $F_\tau^*\omega$ with respect to the basis $b_0^\infty\frac{dx}{x}, \dots, b_{d_x-1}^\infty\frac{dx}{x}$ have order at $\infty$ bounded from below by $q_\tau(\ord_0(W)+1)$. Hence the coefficients (i.e. $u'_0+v'_0, \dots, u'_{d_x-1}+v'_{d_x-1}$) of $F_\tau^*\omega$ with respect to the basis $b_0^\infty\frac{dx}{r}, \dots, b_{d_x-1}^\infty\frac{dx}{r}$ have order at $\infty$ bounded from below by $q_\tau(\ord_0(W)+1)-\deg(r)+1$. On the other hand, since $\deg(w_{ik})<\deg(r)$ for any $i\geq 1$, we have that $\ord_\infty(v'_k)\geq 1+\ord_\infty(W^{-1})$. From our discussion above, we have 
\begin{align}\label{eq:ord at infinity of u'}
	\ord_\infty(u'_k)\geq \min\{q_\tau(\ord_0(W)+1)-\deg(r)+1, \ord_\infty(W^{-1})+1\}
\end{align}
for any $0\leq k\leq d_x-1$. By \cref{lemma:reductionprocedureforcohomologyclass}~\ref{reduction at infty} the terms $d_{jk}x^j$ in \eqref{expresionoffinfty} from $f_{-\infty}$ have degrees bounded from above by $-\min\{q_\tau(\ord_0(W)+1),\ord_\infty(W^{-1})+\deg(r)\}$. Notice that $\deg(r)\geq 1$ (since $g(X)\geq 1$), with the discussion for $f_{+\infty}$, we have
\begin{align*}
	\kappa_3\leq -\min\{q_\tau(\ord_0(W)+1), \ord_\infty(W^{-1})+1\}.
\end{align*} Since the infinite reduction stops when $\ord_\infty(\sum_{j>0}v_{jk}+v_k)>\ord_0(W)-\deg(r)+1$ for each $0\leq k\leq d_x-1$, by \cref{lemma:reductionprocedureforcohomologyclass}~\ref{reduction at infty}, we have $\kappa_2\geq -\ord_0(W)$. From the change in basis from $b_0^0, \dots, b_{d_x-1}^0$ to $b_0^\infty, \dots, b_{d_x-1}^\infty$ by $W$, we have $\kappa_1\leq -\ord_\infty(W)+\kappa_3$, which gives an upper bound for $\kappa_1$.
	
	\ref{lemma:estimateonfunctions fend} We know that $F_\tau^*\omega-df_0-df_\infty-df_{end}\in E_0\cap E_\infty$, $df_{end}$ is the part contained in $d(B_0\cap B_\infty)$, $f_{end}\in B_0\cap B_\infty$. By the definition of $B_0$, we have $\lambda_2\geq \ord_0(W^{-1})$. By the definition of $B_\infty$, we have $\lambda_3\leq-(\ord_0(W)+1)$. The bound of $\lambda_1$ is given by base change, see \cref{remark:E0Einf}.
\end{proof}

\begin{proposition}
	\label[pprop]{prop:convergentdomainoffinitereduction}
	Keep the notation in
	\cref{prop:precisionofreductionprocess},
	and let $P\in X(\overline K)$.
	\begin{enumerate1}
		\item \label{convergence-f0-finite}
		If $P$ is a finite bad point, then the series
		defining $f_0(P)$ converges whenever
		$v_K(r(P))<1/q_\tau$.
		
		\item \label{convergence-f0-infinite}
		If $P$ is an infinite point with
		$x(P)\neq\infty$, then the series defining
		$f_0(P)$ converges.
	\end{enumerate1}
\end{proposition}

\begin{proof}
	\ref{convergence-f0-finite}
	Put $\varepsilon=v_K(r(P))<1/q_\tau$.
	By \eqref{eq:finite-reduction-coefficient-bound}
	and the integrality of $x(P)$ and $b_k^0(P)$,
	\[
	v_K\left(
	\frac{c_{jk}(P)}{r(P)^j}b_k^0(P)
	\right)
	\geq\rho(j)-j\varepsilon
	\geq
	j(1/q_\tau-\varepsilon)-O(\log j)
	\longrightarrow+\infty.
	\]
	This proves convergence.
	
	\ref{convergence-f0-infinite}
	Put $a=v_K(x(P))<0$. Since $r$ has integral coefficients and unit leading coefficient, we have
	\[v_K(r(P))=\deg(r)\cdot a.\]
	Moreover, $\deg(c_{jk})<\deg(r)$ implies
	\[v_K(c_{jk}(P))\geq\rho(j)+(\deg(r)-1)a.\]
	Consequently, for each $k$,
	\[
	v_K\left(
	\frac{c_{jk}(P)}{r(P)^j}b_k^0(P)
	\right)
	\geq
	\rho(j)+(\deg(r)-1)a-j\deg(r)\cdot a+v_K(b_k^0(P))
	\longrightarrow+\infty,
	\]
	because $\rho(j)\geq j/q_\tau-O(\log j)$
	and $a<0$. This proves convergence.
\end{proof}

We cannot calculate $F_\tau^*\omega$ with $\omega\in E_0\cap E_\infty$ exactly in the implementation; we can only calculate it modulo $\pi^{N_{\mathrm{max}}}$ for some $N_{\mathrm{max}}\in \N_{> v_K(q_\tau)}$ by \cref{prop:approximationofFrobenius}~\ref{frobenius of b}, here we use the notation $N_{\mathrm{max}}$ consistent with the one in the code. During the proof above, we notice that if we fix $N_{\mathrm{max}}$, then $F_\tau^*\omega$ is of the form $q_\tau\sum_{j=-\infty}^{q_\tau(N_{\mathrm{max}}-v_K(q_\tau))-1}\sum_{k=0}^{d_x-1}\frac{w_{jk}}{r^j}b_k^0\frac{dx}{r}+O(\pi^{N_{\mathrm{max}}})$. Next, we will choose a suitable $N_\mathrm{max}$ such that $f_0, f_\infty, f_{end} \ \ (\mod \pi^N)$ can be determined.

\begin{proposition}\label[pprop]{prop:precisionofreductionprocess}
	For integers $m>v_K(q_\tau)$, define $h_0(m)$ and let
	$h_1,h_2\in\N$:
	\begin{align*}
		h_0(m)&\coloneq e_K\left\lfloor\log_p
		\bigl((q_\tau(m-v_K(q_\tau)+1)-1)e_0\bigr)\right\rfloor,\\
		h_1&\coloneq e_K\left\lfloor\log_p
		\max\{1,-(\ord_\infty(W^{-1})+1)e_\infty\}\right\rfloor,\\
		h_2&\coloneq e_K\left\lfloor\log_p\max\left\{1,
		-\min\{q_\tau(\ord_0(W)+1),\ord_\infty(W^{-1})+1\}e_\infty
		\right\}\right\rfloor.
	\end{align*}
	Let $\omega\in E_0\cap E_\infty$ be an integral $1$-form,
	and $N\in\N_{\geq1}$. Fix the basis and primitives of
	\cref{notation:finalreductionmatrix}, and let $c_{end}$ be
	the constant defined there. We choose
	$N_{\max}\in\Z_{>v_K(q_\tau)+1}$ such that
	\[
	m-\max\{h_0(m)+h_1,h_2\}\geq N+c_{end}
	\qquad\text{for every integer }m\geq N_{\max}.
	\]
	Let
	\[
	\widehat{\omega}
	=q_\tau\sum_{j=-\infty}^{q_\tau(N_{\max}-v_K(q_\tau))-1}
	\sum_{k=0}^{d_x-1}
	\frac{\widehat{w}_{jk}}{r^j}b_k^0\frac{dx}{r}
	\]
	be a differential form satisfying the following conditions:
	\begin{itemize}
		\item $\deg(\widehat{w}_{jk})<\deg(r)$ for every
		$j>0$ and $0\leq k\leq d_x-1$;
		\item
		$\ord_\infty(\widehat{u'_k})\geq
		\min\{q_\tau(\ord_0(W)+1),\ord_\infty(W^{-1})+1\}
		-\deg(r)+1$
		for every $0\leq k\leq d_x-1$, where
		$\sum_{k=0}^{d_x-1}\widehat{u'_k}b_k^\infty\frac{dx}{r}
		=q_\tau\sum_{j\leq0}\sum_{k=0}^{d_x-1}
		\frac{\widehat{w}_{jk}}{r^j}b_k^0\frac{dx}{r};$
		\item
		\[
		\widehat{\omega}
		=q_\tau\sum_{j=-\infty}^{q_\tau(N_{\max}-v_K(q_\tau))-1}
		\sum_{k=0}^{d_x-1}
		\frac{w_{jk}}{r^j}b_k^0\frac{dx}{r}
		+O(\pi^{N_{\max}}),
		\]
		where $w_{jk}$ is given in \eqref{expresionofFw1}.
	\end{itemize}
	Then
	\[
	f_0=\widehat f_0+O(\pi^N),\qquad
	f_\infty=\widehat f_\infty+O(\pi^N),\qquad
	f_{end}=\widehat f_{end}+O(\pi^N),
	\]
	where $f_0,f_\infty,f_{end}$ (resp.\
	$\widehat f_0,\widehat f_\infty,\widehat f_{end}$)
	are given by \cref{algthm:reductionprocedure} for
	$F_\tau^*\omega$ (resp.\ $\widehat\omega$).
\end{proposition}

\begin{proof}
	We keep the notation in \eqref{expresionofFw1} and
	\eqref{expresionofFw2}, and put $N_*=N+c_{end}$.
	By our choice of $N_{\max}$,
	\[
	m-h_0(m)\geq N_*+h_1
	\quad(m\geq N_{\max}),\qquad
	N_{\max}-h_2\geq N_*.
	\]
	
	Firstly, we consider the reduction process for $j>0$.
	If $j\geq q_\tau(N_{\max}-v_K(q_\tau))$ and
	$\min_{0\leq k\leq d_x-1}v_K(q_\tau w_{jk})=m<\infty$,
	then \cref{lemma:frobenius of a form} gives
	\[
	m\geq N_{\max},\qquad
	j\leq q_\tau(m-v_K(q_\tau)+1)-1.
	\]
	By \cref{prop:valuationoffinitereduction}~\ref{prop:valuationoffinitereduction1},
	the corresponding function and remainder form after finite
	reduction have coefficient valuations at least
	\[
	m-e_K\lfloor\log_p(je_0)\rfloor
	\geq m-h_0(m)\geq N_*+h_1.
	\]
	The subsequent infinite reduction loses at most $h_1$ digits,
	by \cref{prop:valuationoffinitereduction}~\ref{prop:valuationoffinitereduction2}
	and \cref{prop:valuationofinfinitereduction}.
	Then the resulting function and form vanish modulo $\pi^{N_*}$.
	Similarly, for
	$1\leq j\leq q_\tau(N_{\max}-v_K(q_\tau))-1$, we have
	\[
	q_\tau\sum_{k=0}^{d_x-1}\frac{w_{jk}}{r^j}b_k^0\frac{dx}{r}
	\equiv
	q_\tau\sum_{k=0}^{d_x-1}
	\frac{\widehat w_{jk}}{r^j}b_k^0\frac{dx}{r}
	\pmod{\pi^{N_{\max}}}.
	\]
	Finite reduction loses at most $h_0(N_{\max})$ digits,
	and the subsequent infinite reduction loses at most $h_1$
	digits. Thus the corresponding functions and forms agree
	modulo $\pi^{N_*}$, since
	\[
	N_{\max}-h_0(N_{\max})-h_1\geq N_*.
	\]
	Here the change of basis does not decrease coefficient
	valuations, since the matrix
	$W^{-1}\in\GL_{d_x}(K^\circ[x,x^{-1},1/r])$.
	This proves the estimate for $f_0$, since $N_*\geq N$.
	
	Next, we consider the reduction process for $j\leq0$. We have
	\[
	q_\tau\sum_{j\leq 0}\sum_{k=0}^{d_x-1}
	\frac{w_{jk}}{r^j}b_k^0\frac{dx}{r}
	\equiv
	q_\tau\sum_{j\leq 0}\sum_{k=0}^{d_x-1}
	\frac{\widehat w_{jk}}{r^j}b_k^0\frac{dx}{r}
	\pmod{\pi^{N_{\max}}}.
	\]
	After changing to the basis
	$b_0^\infty,\dots,b_{d_x-1}^\infty$, we have
	\[
	\sum_{k=0}^{d_x-1}u'_kb_k^\infty\frac{dx}{r}
	\equiv
	\sum_{k=0}^{d_x-1}\widehat{u'_k}b_k^\infty\frac{dx}{r}
	\pmod{\pi^{N_{\max}}}.
	\]
	By \eqref{eq:ord at infinity of u'} and our assumption,
	\[
	\ord_\infty(u'_k),\ \ord_\infty(\widehat{u'_k})
	\geq
	\min\{q_\tau(\ord_0(W)+1),\ord_\infty(W^{-1})+1\}
	-\deg(r)+1.
	\]
	Thus the corresponding functions and forms agree modulo
	$\pi^{N_{\max}-h_2}$ by
	\cref{prop:valuationofinfinitereduction}, and hence modulo
	$\pi^{N_*}$. Together with the preceding estimates, this
	proves the estimate for $f_\infty$, since $N_*\geq N$.
	
	Finally, put
	\[
	\eta\coloneq
	F_\tau^*\omega-df_0-df_\infty
	-\bigl(\widehat\omega-d\widehat f_0-d\widehat f_\infty\bigr).
	\]
	By the preceding estimates, $\eta\in E_0\cap E_\infty$
	and its coefficient valuations are at least $N_*$. By \cref{prop:valuationoffinalreduction}, with respect to the given basis, the loss of precision of $\eta$ after the final reduction procedure is $c_{end}$. This proves the estimate for $f_{end}$, since $N_*-c_{end}=N$.
\end{proof}

\begin{remark}
	Put $N_*=N+c_{end}$. To choose $N_{\max}$, start with
	\[
	N_{\max}=\max\{v_K(q_\tau)+2,N_*+h_2\}
	\]
	and increase it until
	\[
	N_{\max}-h_0(N_{\max})\geq N_*+h_1.
	\]
	Such a candidate exists because $m-h_0(m)\to\infty$.
	
	We then check whether this candidate satisfies the required
	condition at every higher valuation level. Choose an integer
	\[
	M\geq\max\{N_{\max},N_*+h_1\}
	\]
	large enough that
	\begin{gather*}
		M-v_K(q_\tau)+1-\frac1{q_\tau}\geq\frac{e_K}{\ln p},\\
		\bigl((q_\tau(M-v_K(q_\tau)+1)-1)e_0\bigr)^{e_K}
		\leq p^{M-N_*-h_1}.
	\end{gather*}
	Such an $M$ exists because the first inequality holds for
	all sufficiently large $M$, while the left-hand side of
	the second grows polynomially and its right-hand side
	exponentially. Then check
	\[
	m-h_0(m)\geq N_*+h_1
	\qquad(N_{\max}\leq m<M).
	\]
	If any test fails, replace $N_{\max}$ by one plus the
	largest integer $m$ for which it fails. The resulting
	$N_{\max}$ satisfies
	\[
	m-\max\{h_0(m)+h_1,h_2\}\geq N+c_{end}
	\qquad\text{for every integer }m\geq N_{\max}.
	\]
	
	Indeed, the first condition on $M$ ensures that the function
	\[
	x-e_K\log_p\bigl((q_\tau(x-v_K(q_\tau)+1)-1)e_0\bigr)
	\]
	is increasing for $x\geq M$, and the second ensures that
	its value at $M$ is at least $N_*+h_1$.
	Replacing the logarithm by its floor only increases this
	function, so $m-h_0(m)\geq N_*+h_1$ for every integer
	$m\geq M$. The remaining valuation levels are covered by
	the finite tests, and $N_{\max}-h_2\geq N_*$ holds by
	construction.
\end{remark}

\begin{remark}
	Rigorously speaking, to compare $f_0$ (resp. $f_\infty$, resp. $f_{end}$) and $\widehat{f}_0$ (resp. $\widehat{f}_\infty$, resp. $\widehat{f}_{end}$), we should fix a "basis" of $\mathcal{R}^\dagger\widehat{\otimes}_{K^\circ}K$ (or a large space) over $K$. We have two choices of "basis", one is $\left\{\frac{x^i}{r^j}b_{k}^0\mid  0\leq i< \deg(r), 0\leq k\leq d_x-1, j\in\Z\right\}$, the other one is  $\left\{\frac{x^i}{r^j}b_{k}^\infty\mid  0\leq i< \deg(r), 0\leq k\leq d_x-1, j\in\Z\right\}$. In fact, in our proof, for "$f_0, f_{end}$", we choose the first "basis", for $f_\infty$, we choose the second one. 
	Since $W\in \GL_{d_x}(K^\circ[x,1/x, 1/r])$, the proposition holds no matter which basis we choose.
\end{remark}

\begin{remark}
	We record two implementation details.
	\begin{itemize}
		\item  In our code, $\widehat\omega$ is obtained from $F_\tau^*\omega$ by reducing the coefficients of $q_\tau w_{jk}$ modulo $\pi^{N_{\max}}$. In particular, any coefficient divisible by $\pi^{N_{\max}}$ is replaced by $0$. 

		\item If we write 
		\[q_\tau\sum_{j=1}^{q_\tau(N_{\mathrm{max}}-v_K(q_\tau))-1}\sum_{k=0}^{d_x-1}\frac{\widehat{w}_{jk}}{r^j}b_k^0\frac{dx}{r} = d\widehat f_0+ \sum_{k=0}^{d_x-1}\widehat{u_k}b_k^\infty\frac{dx}{r},\]
		then by \cref{prop:valuationoffinitereduction}, $\ord_\infty(\widehat{u_k}) \geq \ord_\infty(W^{-1})-\deg(r)+2$.
		Hence the lower bound for $\ord_\infty(\widehat{u_k'})$ in \cref{prop:precisionofreductionprocess} is equivalent to $$\ord_{\infty}(\widehat{u_{k}}+\widehat{u_{k}'})\geq \min\{q_\tau(\ord_0(W)+1), \ord_\infty(W^{-1})+1\}-\deg(r)+1.$$
	\end{itemize}
	
\end{remark}

\begin{remark}
	If $\omega$ runs through $\omega_1,\dots,\omega_{2g}$, reduction of the
	approximations to $F_\tau^*\omega$ produces $\widehat A_\tau$ with
	$A_\tau=\widehat A_\tau+O(\pi^N)$ by the proof of \cref{prop:precisionofreductionprocess} and \cref{prop:valuationoffinalreduction}  (see \cref{algthm:semilinear-and-linear-frobenius-matrices} for the definition of $A_\tau$). 
\end{remark}

\begin{proposition}
	\label[pprop]{prop:accurateforfi}
	Keep the notation in \cref{prop:precisionofreductionprocess}.
	Suppose that $F_\tau^*\omega$ is accurate to precision $N_{\max}$,
	and that $f_0,f_\infty,f_{end}$ are accurate to precision $N$.
	Let $P\in X(K(\pi^{1/e}))$ be given to precision $N$ in the coordinates $(x(P),\boldsymbol{b}^0(P))$, or
	$(1/x(P),\boldsymbol{b}^\infty(P))$ if $P$ is infinite.
	Assume that $P$ is good, finitely bad with $v_K(r(P))<1/q_\tau$, or infinite with $x(P)\ne\infty$ and $N>v_K(1/x(P))$.
	\begin{enumerate1}
\item\label{prop:accurateforfi1}
Set $j_0=q_\tau(N_{\max}-v_K(q_\tau))$ and
\[
v_0\coloneq
\min_{\substack{0\leq k\leq d_x-1\\1\leq j<j_0}}
\{v_K(c_{jk})\},
\]
where $c_{jk}$ is given in \eqref{expresionoff0}.
For a good or finitely bad point, put
\[
T_0(P)\coloneq
\min_{a\in\Z_{\geq j_0}}
\left\{
\left\lfloor\frac{a}{q_\tau}\right\rfloor+v_K(q_\tau)
-e_K\lfloor\log_p(ae_0)\rfloor-av_K(r(P))
\right\}.
\]
Then the value $f_0(P)$ is correct to precision
\[
\begin{cases}
	\begin{aligned}
		\min\bigl\{&T_0(P),\,N-(j_0-1)v_K(r(P))\\
		&\qquad+\min\{v_0-v_K(r(P)),0\}\bigr\},
	\end{aligned}
	&\begin{gathered}\text{if $P$ is good or finitely bad}\end{gathered},\\[2ex]
	\begin{aligned}
		&N+(\ord_\infty(W^{-1})+1)v_K(1/x(P))\\
		&\qquad+\min\{v_0-v_K(1/x(P)),0\},
	\end{aligned}
	&\text{if $P$ is infinite}.
\end{cases}
\]
(At a good point $T_0(P)\geq N$, so the term $T_0(P)$ introduces no additional loss).
		
		\item\label{prop:accurateforfi2}
		Set
			\[v_\infty\coloneq\begin{cases}
			\min_{\substack{0\leq k\leq d_x-1\\ 0\leq j\leq \kappa_1}}\{v_K(c_{jk})\}& \splitcond{ if $P$ is good or is finitely bad;}\\
			\min_{\substack{0\leq k\leq d_x-1\\ \kappa_2\leq j\leq \kappa_3}}\{v_K(d_{jk})\}& \splitcond{ if $P$ is infinite,}
		\end{cases}\]
		where $c_{jk},d_{jk}$ are given in \eqref{expresionoffinfty}.
		Then the value $f_\infty(P)$ is correct to precision
		\[
		\begin{cases}
			N+\min\{v_\infty,0\},&\begin{gathered}\text{if $P$ is good or finitely bad}\end{gathered},\\[1ex]
			\begin{aligned}
				&N+\min\{q_\tau(\ord_0(W)+1),\ord_\infty(W^{-1})+1\}v_K(1/x(P))\\
				&\qquad+\min\{v_\infty-v_K(1/x(P)),0\},
			\end{aligned}&\text{if $P$ is infinite}.
		\end{cases}
		\]
	
		\item\label{prop:accurateforfi3}
		Set 
		\[v_{end}\coloneq\begin{cases}
			\min_{\substack{0\leq k\leq d_x-1\\ 0\leq j\leq \lambda_1}}\{v_K(c_{jk})\}& \text{ if $P$ is good or is finitely bad;}\\
			\min_{\substack{0\leq k\leq d_x-1\\ \lambda_2\leq j\leq \lambda_3}}\{v_K(d_{jk})\}& \text{ if $P$ is infinite,}
		\end{cases}\]
		where $c_{jk}, d_{jk}$ are given in \eqref{expresionoffend}.
		Then the value $f_{end}(P)$ is correct to precision
		\[
		\begin{cases}
			N+\min\{v_{end},0\},&\begin{gathered}\text{if $P$ is good or finitely bad}\end{gathered},\\[1ex]
			\begin{aligned}
				&N+(\ord_0(W)+1)v_K(1/x(P))\\
				&\qquad+\min\{v_{end}-v_K(1/x(P)),0\},
			\end{aligned}
			&\text{if $P$ is infinite}.
		\end{cases}
		\]
	\end{enumerate1}
\end{proposition}
\begin{proof}
	Since the coordinates of $P$ are integral in the relevant chart, we have $v_K(b_k^0(P))\geq0$ in the finite case and $v_K(b_k^\infty(P))\geq0$ in the infinite case, by \cref{def:triplehasgoodreduction}~\ref{good reduction: b are basis of model}.

\ref{prop:accurateforfi1} We first show that $v_K(c_{jk})\geq N$ for $j\geq j_0$, $0\leq k\leq d_x-1$. For any $a\geq j\geq j_0$, put
\[
m=\left\lfloor\frac{a}{q_\tau}\right\rfloor+v_K(q_\tau).
\]
By the definition of $j_0$, we have $m\geq N_{\max}$ and
\[
a\leq q_\tau(m-v_K(q_\tau)+1)-1.
\]
Hence, by the definition of $h_0$ and the condition on $N_{\max}$ in
\cref{prop:precisionofreductionprocess}, 
\[h_0(m)= e_K\left\lfloor\log_p\bigl((q_\tau(m-v_K(q_\tau)+1)-1)e_0\bigr)\right\rfloor\geq e_K\left\lfloor\log_p\bigl(ae_0\bigr)\right\rfloor,\]
\begin{align*}
	\left\lfloor\frac{a}{q_\tau}\right\rfloor+v_K(q_\tau)
	-e_K\lfloor\log_p(ae_0)\rfloor\geq m-h_0(m)\geq N+h_1\geq N.
\end{align*}
Taking the minimum over $a\geq j$ gives $\rho(j)\geq N+h_1\geq N$ for every $j\geq j_0$.
Thus $v_K(c_{jk})\geq N$ for any $j\geq j_0$, $0\leq k\leq d_x-1$.

To give the precision of $f_0(P)$, we consider \eqref{expresionoff0}. If $P$ is good or finitely bad, then $x(P)$ is integral. Since $r$ has integral coefficients, $r(P)$ is accurate to absolute precision $N$. As $v_K(r(P))<1/q_\tau$ and $N\geq1$, we have $N>v_K(r(P))$.
By \cref{lemma:precisionofbasicoperation}~\ref{lemma:precision basic 1.iii}, $r(P)^{-j}$ is accurate to absolute precision $N-(j+1)v_K(r(P))$, and $c_{jk}(P)$ is accurate to precision $N+\min\{v_0,0\}$ with $v_K(c_{jk}(P))\geq v_0$. By  \cref{lemma:precisionofbasicoperation} again, $\frac{c_{jk}(P)}{r(P)^j}b_k^0(P)$ is correct to precision
\begin{align*}
	&\min\{N-(j+1)v_K(r(P))+v_K(b_k^0(P))+v_0, N+v_K(b_k^0(P))-jv_K(r(P))\}\\
	\geq& N-jv_K(r(P))+\min\{v_0-v_K(r(P)),0\}\\
	\geq& N-(j_0-1)v_K(r(P))+\min\{v_0-v_K(r(P)),0\}
\end{align*}
for $1\leq j<j_0$. 
This proves the stated bound for the retained terms.

	For the omitted terms $j\geq j_0$, \eqref{eq:finite-reduction-coefficient-bound} gives 
	\[v_K\left(\frac{c_{jk}(P)}{r(P)^j}b_k^0(P)\right)\geq \rho(j)-jv_K(r(P)).\]
	Since $v_K(r(P))\geq0$, the definition of $\rho$ gives
	\begin{align*}
		\min_{j\geq j_0}\{\rho(j)-jv_K(r(P))\}&=\min_{j\geq j_0}\min_{a\geq j}
		\left\{
		\left\lfloor\frac{a}{q_\tau}\right\rfloor+v_K(q_\tau)
		-e_K\lfloor\log_p(ae_0)\rfloor-jv_K(r(P))
		\right\}\\
		&=\min_{a\geq j_0}\min_{j_0\leq j\leq a}
		\left\{
		\left\lfloor\frac{a}{q_\tau}\right\rfloor+v_K(q_\tau)
		-e_K\lfloor\log_p(ae_0)\rfloor-jv_K(r(P))
		\right\}\\
		&=T_0(P).
	\end{align*}
	Indeed, for each fixed $a$, the inner minimum is attained at $j=a$.
	At a good point, $v_K(r(P))=0$, and the uniform condition
	on $N_{\max}$ in \cref{prop:precisionofreductionprocess}
	gives $\rho(j)\geq N$ for $j\geq j_0$.
	Thus the omitted terms introduce no additional loss.
	
	If $P$ is infinite, the point is given by $(1/x(P), \boldsymbol{b}^\infty(P))$, we put $t=1/x$. Then by \cref{lemma:precisionofbasicoperation}~\ref{lemma:precision basic 1.ii} \ref{lemma:precision basic 1.iii}, for any $j\in\Z, x(P)^j=t(P)^{-j}$ is correct to precision $N-(j+1)v_K(t(P))$ and $v_K(x(P)^j) = -jv_K(t(P))$. 
	Writing $c_{jk}(x)=x^{\deg(c_{jk})}f_{jk}(t)$ with the reversed polynomial $f_{jk}(t)=t^{\deg(c_{jk})}c_{jk}(1/t)\in K[t]$. By \cref{lemma:precisionofbasicoperation}~\ref{lemma:precision basic 2}, $f_{jk}(P)$ is correct to precision $\min\{N,N+v_0\}$ and $v_K(f_{jk}(P))\geq v_0$. By \cref{lemma:precisionofbasicoperation}~\ref{lemma:precision basic 1.iii}, 
	$c_{jk}(P)=x(P)^{\deg(c_{jk})}f_{jk}(P)$ is correct to precision
	\begin{align*}
&\min\{N-(\deg(c_{jk})+1)v_K(t(P))+v_0, \min\{N,N+v_0\}-\deg(c_{jk})v_K(t(P))\}\\
=&
N-\deg(c_{jk})v_K(1/x(P))
+\min\{v_0-v_K(1/x(P)),0\},
	\end{align*}
	with
	\[
	v_K(c_{jk}(P))\geq v_0-\deg(c_{jk})v_K(1/x(P)).
	\]
	Since $r$ has integral coefficients and unit leading coefficient,
	the reversed polynomial $t^{\deg(r)}r(1/t)$ is a unit at $t(P)$, and $t(P)^{\deg(r)}r(P)$ is correct to precision $N$ by \cref{lemma:precisionofbasicoperation}~\ref{lemma:precision basic 2}.
	Consequently, $r(P)^{-j} = x(P)^{-j\deg(r)}(t(P)^{\deg(r)}r(P))^{-j}$ has valuation $j\deg(r)v_K(1/x(P))$, and is accurate to precision $$\min\{N+(j\deg(r)-1)v_K(1/x(P)),N+j\deg(r)v_K(1/x(P))\}=N+(j\deg(r)-1)v_K(1/x(P))$$ by \cref{lemma:precisionofbasicoperation}~\ref{lemma:precision basic 1.iii}.
	We next consider $\boldsymbol{b}^0(P)=\boldsymbol{b}^\infty(P)W(P)^{-1}$.
	Since $W^{-1}=(W^\infty)^{-1}W^0\in\GL_{d_x}(K^\circ[x,x^{-1},1/r])$,
	we have $v_K(W(P)^{-1})\geq\ord_\infty(W^{-1})v_K(1/x(P))$,
	and $W(P)^{-1}$ is correct to precision $N+(\ord_\infty(W^{-1})-1)v_K(1/x(P))$ by \cref{lemma:precisionofbasicoperation}~\ref{lemma:precision basic 1.i} \ref{lemma:precision basic 1.ii} \ref{lemma:precision basic 1.iii}.
	The integrality of $b_k^\infty(P)$ gives
	\[
	v_K(b_k^0(P))\geq\ord_\infty(W^{-1})v_K(1/x(P)),
	\]
	and $b_k^0(P)$ is correct to precision
	\[
	N+(\ord_\infty(W^{-1})-1)v_K(1/x(P)).
	\]
Applying
	\cref{lemma:precisionofbasicoperation}, we find that the minimum of
	\begin{itemize}
		\item $N-\deg(c_{jk})v_K(1/x(P))+\min\{v_0-v_K(1/x(P)),0\}+j\deg(r)v_K(1/x(P))$ \\ $+\ord_\infty(W^{-1})v_K(1/x(P))$,
		\item $N+(j\deg(r)-1)v_K(1/x(P))+v_0-\deg(c_{jk})v_K(1/x(P))+\ord_\infty(W^{-1})v_K(1/x(P))$,
		\item $N+(\ord_\infty(W^{-1})-1)v_K(1/x(P))+v_0-\deg(c_{jk})v_K(1/x(P))+j\deg(r)v_K(1/x(P))$,
	\end{itemize}
is the precision of $\frac{c_{jk}(P)}{r(P)^j}b_k^0(P)$, that is
\begin{align*}
	&N+(j\deg(r)-\deg(c_{jk})+\ord_\infty(W^{-1}))v_K(1/x(P))+\min\{v_0-v_K(1/x(P)),0\}\\
		\geq &N+(\ord_\infty(W^{-1})+1)v_K(1/x(P))+\min\{v_0-v_K(1/x(P)),0\}.
\end{align*}
The last inequality follows from $j\geq1$ and $\deg(c_{jk})<\deg(r)$.
For $j\geq j_0$, the preceding argument gives $v_K(c_{jk})\geq N$. Since $\deg(c_{jk})<\deg(r)$, the omitted terms have valuation at least $N+(\ord_\infty(W^{-1})+1)v_K(1/x(P))$. Thus the tail satisfies the stated bound as well.
	This proves the infinite-point bound.
	
	\ref{prop:accurateforfi2}
	If $P$ is good or finitely bad, the expression \eqref{expresionoffinfty} is polynomial in the integral coordinates $x(P),\boldsymbol{b}^0(P)$. Hence \cref{lemma:precisionofbasicoperation}~\ref{lemma:precision basic 2}
	gives precision $N+\min\{v_\infty,0\}$.
	If $P$ is infinite, each term $d_{jk}x(P)^jb_k^\infty(P)$
	is accurate to precision at least
	\[
	N-jv_K(1/x(P))+\min\{v_\infty-v_K(1/x(P)),0\}.
	\]
	Since
	\[
	j\leq-\min\{q_\tau(\ord_0(W)+1),\ord_\infty(W^{-1})+1\},
	\]
	the asserted bound follows.
	
	\ref{prop:accurateforfi3} The proof is the same as for \ref{prop:accurateforfi2}, using the degree bound $\lambda_3\leq-(\ord_0(W)+1)$ in \eqref{expresionoffend}.
\end{proof}

\begin{remark}
	The quantity $T_0(P)$ can be computed by finitely many tests.
	Set
	\[
	C=\left\lfloor\frac{j_0}{q_\tau}\right\rfloor+v_K(q_\tau)
	-e_K\lfloor\log_p(j_0e_0)\rfloor-j_0v_K(r(P)).
	\]
	Choose an integer $M>j_0$ large enough that
	\[
	M\left(\frac{1}{q_\tau}-v_K(r(P))\right)>2e_K
	\]
	and
	\[
	M\left(\frac{1}{q_\tau}-v_K(r(P))\right)+v_K(q_\tau)-1
	-e_K\bigl(\lfloor\log_p(Me_0)\rfloor+1\bigr)\geq C.
	\]
	Such an $M$ exists because $v_K(r(P))<1/q_\tau$ and
	the positive linear term dominates the logarithmic term.
	Then
	\[
	T_0(P)=
	\min_{j_0\leq a<M}
	\left\{
	\left\lfloor\frac{a}{q_\tau}\right\rfloor+v_K(q_\tau)
	-e_K\lfloor\log_p(ae_0)\rfloor-av_K(r(P))
	\right\}.
	\]
	
	Indeed, the function
	\[
	x\left(\frac{1}{q_\tau}-v_K(r(P))\right)+v_K(q_\tau)-1
	-e_K\log_p(xe_0)
	\]
	is increasing for $x\geq M$ by the first condition,
	and its value at $M$ is at least $C$ by the second. Thus every $a\geq M$ contributes
	a value at least $C$. Since the finite range contains
	$a=j_0$, whose value is $C$, its minimum is the full minimum.

\end{remark}
\begin{remark}
	From this proposition, when $P$ is finitely bad, the precision bound for $f_0(P)$ may be negative.
	To ensure positive precision, we require
	\[T_0(P)>0, \qquad N-(j_0-1)v_K(r(P))+\min\{v_0-v_K(r(P)),0\}>0, \]
	where $j_0-1=q_\tau(N_{\max}-v_K(q_\tau))-1$.
\end{remark}

Recall the matrix $A_\tau$ (resp. $A_q$) for the $q_\tau$-power Frobenius (resp. $q$-power Frobenius) lift $F_\tau^*$ (resp. $F^*_q$) given by \cref{algthm:semilinear-and-linear-frobenius-matrices} and the correction values $\boldsymbol{f}_\tau^{\tau^j}(P_j)$, $\boldsymbol{g}_q(P)$ given by \cref{algthm:frobeniusofapoint}.

\begin{lemma}\label[plemma]{lemma:precision-semilinear-composition}
	Let $\widehat A_\tau\in M_{2g}(K)$ approximate $A_\tau$ with
	\[
	v_K(\widehat A_\tau-A_\tau)\geq N_A,
	\qquad N_A\geq v_K(A_\tau).
	\]
	Assume that all arithmetic preserves the stated precision bounds.
	\begin{enumerate1}
		\item
		Let $\widehat A_q$ be the ordered product formed from
		$\widehat A_\tau$ as in \cref{algthm:semilinear-and-linear-frobenius-matrices} \eqref{eq:linear-frobenius-matrix}. Then
		\[
		v_K(\widehat A_q-A_q)
		\geq N_A+(n_\tau-1)v_K(A_\tau).
		\]
		Thus
		\[
		N_A\geq N-(n_\tau-1)v_K(A_\tau)
		\]
		suffices for precision $N$ in $A_q$.
		
		\item
		For $0\leq j<n_\tau$, let $\widehat u_j\in K(\pi^{1/e})^{\oplus 2g}$ approximate
		\[u_j\coloneq\boldsymbol f_\tau^{\tau^j}(P_j), \qquad v_K(\widehat u_j-u_j)\geq N_u.\]
		Let $\widehat{\boldsymbol g}_q(P)$ be computed from $\widehat A_\tau$ and the $\widehat u_j$ using \cref{algthm:frobeniusofapoint} \eqref{eq:frobenius-point-evaluation}.
		With the convention $\min\emptyset=+\infty$, we have
		\[
		\begin{split}
			v_K\bigl(
			\widehat{\boldsymbol g}_q(P)-\boldsymbol g_q(P)
			\bigr)
			\geq\min\Bigl\{&
			N_u+\min\{0,(n_\tau-1)v_K(A_\tau)\},\\
			&
			N_A+\min_{0\leq j\leq n_\tau-2}
			\bigl\{
			v_K(\boldsymbol f_\tau^{\tau^j}(P_j))
			+(n_\tau-2-j)v_K(A_\tau)
			\bigr\}
			\Bigr\}.
		\end{split}
		\]
		Consequently, precision $N$ in $\boldsymbol g_q(P)$
		is ensured by
		\[
		N_u\geq N-\min\{0,(n_\tau-1)v_K(A_\tau)\},
		\]
		and
		\[
		N_A\geq
		N-\min_{0\leq j\leq n_\tau-2}
		\bigl\{
		v_K(\boldsymbol f_\tau^{\tau^j}(P_j))
		+(n_\tau-2-j)v_K(A_\tau)
		\bigr\}.
		\]
	\end{enumerate1}
\end{lemma}

\begin{proof}
	The automorphism $\tau$ preserves $v_K$ and hence every
	fractional ideal $\pi^mK^\circ$, even when $\tau(\pi)\neq\pi$.
	Moreover, $N_A\geq v_K(A_\tau)$ implies
	$v_K(\widehat A_\tau)\geq v_K(A_\tau)$.
	
	\begin{enumerate1}
		\item
		Each conjugate $\tau^j(A_\tau)$ has valuation
		$v_K(A_\tau)$ and is approximated to precision $N_A$.
		Applying the product estimate in
		\cref{lemma:precisionofbasicoperation},
		\ref{lemma:precision basic 1.iii}, gives
		\[
		v_K(\widehat A_q-A_q)
		\geq N_A+(n_\tau-1)v_K(A_\tau).
		\]
		
		\item
		Write
		\[
		C_j\coloneq
		\tau^{n_\tau-1}(A_\tau)\cdots
		\tau^{j+1}(A_\tau),
		\qquad
		\boldsymbol g_q(P)=
		\sum_{j=0}^{n_\tau-1}C_ju_j,
		\]
		where an empty product is $I$.
		Put $\widehat C_j\coloneq\tau^{n_\tau-1}(\widehat A_\tau)\cdots\tau^{j+1}(\widehat A_\tau)$, so that unwinding the recursion \eqref{eq:frobenius-point-evaluation} gives $\widehat{\boldsymbol g}_q(P)=\sum_{j=0}^{n_\tau-1}\widehat C_j\widehat u_j$, and put
		$\ell_j=n_\tau-1-j$.
		Then
		\[
		v_K(\widehat C_j)\geq\ell_jv_K(A_\tau).
		\]
		For $\ell_j\geq1$, \cref{lemma:precisionofbasicoperation},
		\ref{lemma:precision basic 1.iii} gives
		\[
		v_K(\widehat C_j-C_j)
		\geq N_A+(\ell_j-1)v_K(A_\tau),
		\]
		\[
		v_K(\widehat C_j\widehat u_j-C_ju_j)\geq \min\{N_u+\ell_jv_K(A_\tau), N_A+(\ell_j-1)v_K(A_\tau)+v_K(u_j)\}.
		\]
		When $\ell_j=0$, we have $j=n_\tau-1$ and $C_{n_\tau-1}=I$, so	
		\[v_K(\widehat C_j\widehat u_j-C_ju_j) = v_K(\widehat u_{n_\tau-1}-u_{n_\tau-1})\geq N_u.\]
		
		Taking the minimum over $j$ and using
		\[\min_{0\leq j<n_\tau}\ell_jv_K(A_\tau)=\min\{0,(n_\tau-1)v_K(A_\tau)\}\]
		proves the stated bound.
		When $n_\tau=1$, there are no second terms, consistent with the convention $\min\emptyset=+\infty$.
	\end{enumerate1}
\end{proof}

\begin{remark}\label{rem:precision-composed-values}
	Evaluate $u_j=\boldsymbol f_\tau^{\tau^j}(P_j)$ using
	\cref{prop:accurateforfi}, taking into account the actual precision of $P_j$.
	If $u_j$ is computed to absolute precision $N_{u,j}$,
	the proof of \cref{lemma:precision-semilinear-composition} gives the bound
	\[
	\min\left\{
	\min_{0\leq j<n_\tau}
	\bigl\{N_{u,j}+(n_\tau-1-j)v_K(A_\tau)\bigr\},
	\;
	N_A+\min_{0\leq j<n_\tau-1}
	\bigl\{v_K(u_j)+(n_\tau-2-j)v_K(A_\tau)\bigr\}
	\right\}
	\]
	for the absolute precision of $\boldsymbol g_q(P)$, with the convention $\min\emptyset=+\infty$.
	Certified lower bounds for the valuations may be used in this estimate.
	
	Combine the bounds for $\boldsymbol g_q(P'')$ and $\boldsymbol g_q(R'')$ with the tiny-integral precisions to determine the right-hand-side precision $N'$ in \cref{prop:precisionofsolutionoflinearsystem}.
	Twisting coefficients by $\tau$ causes no precision loss, since $\tau$ preserves $v_K$.
\end{remark}
	
	It remains to consider the precision of the solution of \cref{algthm:integralonbasis} \eqref{equation:linearequationfromfrobenius}.

	\begin{proposition}
		\label[pprop]{prop:precisionofsolutionoflinearsystem}
		Keep the notation as in \cref{algthm:integralonbasis}~\eqref{equation:linearequationfromfrobenius}. Suppose that the matrix $A_q\in M_{2g}(K)$ is accurate to precision $N$, and the right hand side of \eqref{equation:linearequationfromfrobenius} is of valuation $n_0$ and accurate to precision $N'\leq N$ according to \cref{prop:precisionfortinyintegral,rem:precision-composed-values}.
		If the right-hand side of \eqref{equation:linearequationfromfrobenius} is nonzero, assume $N'>n_0$ so that its valuation is determined.
		Assume also
		\[
		N+(2g-1)v_K(A_q-I)>v_K(\det(A_q-I));
		\]
		otherwise increase the working precision. This sufficient condition
		ensures that the matrix approximation is invertible and that the
		inverse-precision bound applies. Then the integrals $\int_{P''}^{R''}\omega_i$ computed in \cref{algthm:integralonbasis} are accurate to precision
		\[\min\{N'+v_K((A_q-I)^{-1}),N-2v_K(\det(A_q-I))+2(2g-1)v_K(A_q-I)+n_0\}.\]
	\end{proposition}
\begin{proof}
	It suffices to consider the linear system $Mu=w$, where $M\in\GL_n(K)$, $w\in K(\pi^{1/e})^{\oplus n}$. Note that the estimates of \cref{lemma:precisionofbasicoperation} hold over $K(\pi^{1/e})$ as well.
	Let $\widehat M$ and $\widehat w$ be approximations
	satisfying
	\[v_K(\widehat M-M)\geq N, \qquad v_K(\widehat w-w)\geq N'.\]
	By \cref{lemma:precisionofbasicoperation} \ref{lemma:precision basic 1.ii}, the hypothesis $N+(n-1)v_K(M)>v_K(\det M)$ ensures that $\widehat M$ is invertible and
	\[
	v_K(\widehat M^{-1}-M^{-1})
	\geq N-2v_K(\det M)+2(n-1)v_K(M).
	\]
	The adjugate formula also gives
	\[
	N+v_K(M^{-1})
	\geq N+(n-1)v_K(M)-v_K(\det M)>0.
	\]
Therefore, the identity
\[
M^{-1}-\widehat M^{-1}
=M^{-1}(\widehat M-M)\widehat M^{-1}
\]
shows that the difference has valuation strictly greater
than $v_K(\widehat M^{-1})$. Hence
$v_K(\widehat M^{-1})=v_K(M^{-1})$.
	
	Using
	\[
	\widehat M^{-1}\widehat w-M^{-1}w
	=
	\widehat M^{-1}(\widehat w-w)
	+(\widehat M^{-1}-M^{-1})w,
	\]
	we obtain the precision bound
	\[
	\min\left\{
	N'+v_K(M^{-1}),
	N-2v_K(\det M)+2(n-1)v_K(M)+v_K(w)
	\right\}.
	\]
	Here $v_K(0)=+\infty$. Taking $M=A_q-I$, $n=2g$,
	and $w$ to be the right-hand side of
	\eqref{equation:linearequationfromfrobenius}
	proves the proposition.
\end{proof}

\section{Implementation}
\label{section: implementation}
In this section, we give examples of integration on curves over number fields. For a smooth projective curve $X/F$ with a plane model $Q(x,y)=0$ over $\OO_F$, we use the notation of \cref{subsection:plane models}, such as $W^0,W^\infty,\boldsymbol{b}^0,\boldsymbol{b}^\infty,r$, with $F$ in place of $K$. We then pass to a completion $F_{\mathfrak P}$ for the integration computations.

All computations of this section and of \cref{appendix:a sc and mw} can be reproduced with the {\sffamily Magma} scripts that accompany the code.

\subsection{Computation}

\begin{example}[\cite{siksek2013explicit}~Section~6]\label{example:Siksek-8-1}
	Let $F\coloneq\Q(\theta)$ with $\theta=\sqrt[3]{2}$, and $X$ the hyperelliptic curve defined by
	\[Q(x,y)= y^2-3(4\epsilon^{-1}x^5-\epsilon^{2})\]
	where $\epsilon=1-\theta$. 

    We have that $W^0 = I$, which means that  $F[x,y]/(Q)$ is normal, and $	\boldsymbol{b}^0=[1,y]$ is an integral basis for the function field $F(X)$ over $F[x]$. We also have that
\begin{align*}
	W^\infty=\begin{bmatrix}
		1&0\\
		0&1/x^3\\
	\end{bmatrix},
\end{align*}
so $\boldsymbol{b}^\infty=[1,y/x^3]$ is an integral basis for $F(X)$ over $F[1/x]$. The auxiliary polynomial obtained from the discriminant is
\[r(x)=(-48\theta^2 - 48\theta - 48)x^5 - 12\theta^2 + 24\theta - 12\]
which is irreducible over $F$.

We consider the $F$-rational points on $X$ listed below: 
\begin{center}
	\begin{tabular}{|l|c|c|}
		\hline	\text{Point} & $r(P)$ & Type of points\\ 
		\hline	$P_1=(-\theta^2-\theta-1, 40\theta^2+53\theta+67)$&$32676\theta^2 + 41208\theta + 51876$& good\\
		$P_1'=(-\theta^2-\theta-1, -(40\theta^2+53\theta+67))$&$32676\theta^2 + 41208\theta + 51876$& good\\
		$P_2=(-1, 3\theta+3)$&$36\theta^2 + 72\theta + 36$
		& good\\
		$P_2'=(-1, -(3\theta+3))$&$36\theta^2 + 72\theta + 36$
		& good\\
		$1/x(P_3)=0, \boldsymbol{b}^\infty=(1,0)$&$\infty$& very infinite\\
		\hline
	\end{tabular}
\end{center}

The curve $X$ has genus~$2$. Let $\omega_1,\omega_2$ be the basis of $H^0(X,\Omega^1_{X/F})$ given by \texttt{cohobasis.m}.
We take the prime $\mathfrak{P}\mid 109$ of residue degree~$1$ and $N=20$, then
\begin{align*}
\left(\int_{P_1}^{P_1'}\omega_1,\int_{P_1}^{P_1'}\omega_2\right)&=(25451282+O(109^4),\,-56789000+O(109^4)),\\
\left(\int_{P_1}^{P_2}\omega_1,\int_{P_1}^{P_2}\omega_2\right)&=(-45162733+O(109^4),\,1262438+O(109^4)),\\
\left(\int_{P_1}^{P_2'}\omega_1,\int_{P_1}^{P_2'}\omega_2\right)&=(-70544146+O(109^4),\,-58051438+O(109^4)),\\
\left(\int_{P_1}^{P_3}\omega_1,\int_{P_1}^{P_3}\omega_2\right)&=(12725641+O(109^4),\,-28394500+O(109^4)).
\end{align*}
The rational prime $109$ splits completely in $F$, so $F_{\mathfrak P}=\Q_{109}$ and that computation never leaves $\Q_p$. As a check, $P_1'$ is the image of $P_1$ under the hyperelliptic involution and $P_3$ is a Weierstrass point, so $\int_{P_1}^{P_1'}\omega_i=2\int_{P_1}^{P_3}\omega_i$ for $i=1,2$

Taking instead $\mathfrak P\mid7$ --- and $7$, like $13,19,37$, is \emph{inert} in $\Z[\theta]$ --- makes it a genuine $f_K=3$ computation over $K=\Q_{7^3}$, with residue field $\mathbb F_{343}$. All four affine points remain good. With $N=12$, the script \texttt{example-8.1-inert-7.m} computes the integrals from $P_1$ to $P_1'$, $P_2$ and $P_2'$. The code lifts the $7$-power Frobenius ($q_\tau=7$, $n_\tau=3$) and takes about $0.5$\,s, against about $45$\,s with the $343$-power lift, and both give the same integrals.
 For example, writing $\alpha$ for the chosen generator of $K$ over $\Q_7$, we obtain
\[\int_{P_1}^{P_2}\omega_1= 7\bigl(-685629715\,\alpha^2-642258943\,\alpha-910533089\bigr)+O(7^{12}).\]
The other five values are recorded in \texttt{example-8.1-inert-7.log}.
Each of the six values has non-zero $\alpha$- and $\alpha^2$-components, so they lie in $K\setminus\Q_7$.

\end{example}

\begin{example}\label{example:Xb3s5}
		The modular curve $X(b3,s5)$ has a plane model, see \cite[Lemma~5.2]{freitas2015elliptic}:
	\[Q(x,y)=y^4+(-3x+7)y^3+2(-3x^2-19x-7)y^2+4(-x-3)y+8(3x^3+5x^2+3x).\]
	We consider the curve $X\coloneq X(b3,s5)$ over the number field $F\coloneq\Q(\sqrt{5})$.
    We have that $W^0 = I$, which means that  $F[x,y]/(Q)$ is normal, and $\boldsymbol{b}^0=[1,y,y^2,y^3]$ is an integral basis for the function field $F(X)$ over $F[x]$. We also have that
    \begin{align*}
    	W^\infty=\begin{bmatrix}
    		1&0&0&0\\
    		0&1/x&0&0\\
      		0&0&1/x^2&0\\
            0&0&0&1/x^3\\
    	\end{bmatrix},
    \end{align*}
so $\boldsymbol{b}^\infty=[1,y/x,y^2/x^2,y^3/x^3]$ is an integral basis for $F(X)$ over $F[1/x]$. The auxiliary polynomial obtained from the discriminant is
\begin{align*}
	r(x)=5184(132x^{11} + 1549x^{10} + 20048x^9 + 116742x^8 + 352238x^7 + 629657x^6\\ +
	718900x^5 + 550141x^4 + 288894x^3 + 102346x^2 + 20436x + 1509)
\end{align*}
which is irreducible over $F$.

	We consider the following $\Q(\sqrt{5})$-rational points on $X(b3,s5)$ listed below:
			\begin{center}
				\begin{tabular}{|l|c|c|}
				\hline	\text{Point} & $r(P)$ & Type of points\\ 
				\hline	$P_1=(0,0)$&$7822656$ & good\\
				        $P_2=(1/2,-1)$&$11890992375/16$
				        & good\\
			            $P_3=(-2,-1+\sqrt{5})$&$26568000$& good\\
			            $P_4=(\frac{1-\sqrt{5}}{2},-6+2\sqrt{5})$&$-48708864000\sqrt{5} + 108918432000$
			            & good\\
			            $P_5=(\frac{1+\sqrt{5}}{2},4+2\sqrt{5})$&$48708864000\sqrt{5} + 108918432000$
			            & good\\
			            $1/x(P_6)=0, \boldsymbol{b}^\infty=[1,0,0,0]$&$\infty$& very infinite\\
			 	\hline
			        \end{tabular}
			\end{center}
		The curve $X$ has genus $3$. 
		Notice that $Q$ is reducible modulo $3$, so we take
		$\mathfrak{P}=(7)$ (notice that $7$ is unramified in $\Z[\sqrt{5}]$), $N=20$, then
		\[\left(\int_{P_1}^{P_2}\omega_1,\int_{P_1}^{P_2}\omega_2,\int_{P_1}^{P_2}\omega_3\right)= (O(7^{18}),O(7^{19}),O(7^{19})),\]
		\[\left(\int_{P_1}^{P_3}\omega_1,\int_{P_1}^{P_3}\omega_2,\int_{P_1}^{P_3}\omega_3\right)= (O(7^{17}),O(7^{18}),O(7^{18})),\]
		\[\left(\int_{P_1}^{P_4}\omega_1,\int_{P_1}^{P_4}\omega_2,\int_{P_1}^{P_4}\omega_3\right)= (O(7^{17}),O(7^{18}),O(7^{18})),\]
		\[\left(\int_{P_1}^{P_5}\omega_1,\int_{P_1}^{P_5}\omega_2,\int_{P_1}^{P_5}\omega_3\right)= (O(7^{17}),O(7^{18}),O(7^{18})),\]
		\[\left(\int_{P_1}^{P_6}\omega_1,\int_{P_1}^{P_6}\omega_2,\int_{P_1}^{P_6}\omega_3\right)= (O(7^{17}),O(7^{18}),O(7^{18})).\]
		The displayed integrals all vanish to working precision because the holomorphic integrals are the $p$-adic Abel--Jacobi logarithm of the classes $[P_i-P_1]$, which are torsion in the Jacobian. Indeed, the Jacobian $X(b3,s5)$ is $\Q$-isogenous to a product of elliptic curves in the Cremona isogeny classes $15a$, $75b$ and $75c$. Each of these curves and its quadratic twist by $5$ has rank $0$ over $\Q$, so $J(X)(F)$ is finite, see
		\cite[proof of Lemma~13.1]{freitas2015elliptic}. These outputs are reproduced by \texttt{ColemanIntegration/example-8.2.m}).
\end{example}

We next consider another non-hyperelliptic curve over a number field. 

\begin{example}\label{example:superelliptic-cbrt2}
	Let $X$ be a superelliptic curve of genus $3$ over $F\coloneq\Q(\sqrt[3]{2})$ defined by a plane model:
	\[Q(x,y)=y^4-x(x-1)(x-\sqrt[3]{2}).\]
	We have that $W^0 = I$, which means that  $F[x,y]/(Q)$ is normal, and $\boldsymbol{b}^0=[1,y,y^2,y^3]$ is an integral basis for the function field $F(X)$ over $F[x]$. We also have that
	\begin{align*}
		W^\infty=\begin{bmatrix}
			1&0&0&0\\
			0&1/x&0&0\\
			0&0&1/x^2&0\\
			0&0&0&1/x^3
		\end{bmatrix},
	\end{align*}
	so $\boldsymbol{b}^\infty=[1,y/x,y^2/x^2,y^3/x^3]$ is an integral basis for $F(X)$ over $F[1/x]$. The auxiliary polynomial obtained from the discriminant is
	\[r(x)=-256x(x-1)(x-\sqrt[3]{2}).\]
	
	We consider the following $\Q(\sqrt[3]{2})$-rational points on $X$ listed below:
\begin{center}
	\begin{tabular}{|l|c|c|}
		\hline	\text{Point} & $r(P)$ & Type of points\\ 
		\hline	$x(P_1)=0, \boldsymbol{b}^0=[1,0,0,0]$&$0$ & very bad\\
		$x(P_2)=1, \boldsymbol{b}^0=[1,0,0,0]$&$0$
		& very bad\\
		$x(P_3)=\sqrt[3]{2}, \boldsymbol{b}^0=[1,0,0,0]$&$0$& very bad\\
		$1/x(P_4)=0, \boldsymbol{b}^\infty=[1,0,0,0]$&$\infty$& very infinite\\
		\hline
	\end{tabular}
\end{center}

The curve $X$ has genus~$3$. Let $\omega_1,\omega_2, \omega_3$ be the basis of $H^0(X,\Omega^1_{X/F})$ given by \texttt{cohobasis.m}.
The maximal order in $F$ is $\Z[\sqrt[3]{2}]$. Let $\mathfrak{P}=(3,1+\sqrt[3]{2})\subset \Z[\sqrt[3]{2}]$. We have that
\[(3)=\mathfrak{P}^3,\]
so the completion $F_{\mathfrak{P}}$ of $F$ with respect to $\mathfrak{P}$ is totally ramified over $\Q_3$.
Writing $\pi$ for a uniformizer of $F_{\mathfrak{P}}$ with $N=20$ the holomorphic integrals are
\begin{align*}
\left(\int_{P_1}^{P_2}\omega_1, \int_{P_1}^{P_2}\omega_2, \int_{P_1}^{P_2}\omega_3\right)=(O(\pi^2),O(\pi^2),O(\pi^2)),\\
\left(\int_{P_1}^{P_3}\omega_1, \int_{P_1}^{P_3}\omega_2, \int_{P_1}^{P_3}\omega_3 \right)=(O(\pi^2),O(\pi^2),O(\pi^2)),\\
\left(\int_{P_1}^{P_4}\omega_1, \int_{P_1}^{P_4}\omega_2, \int_{P_1}^{P_4}\omega_3\right)=(O(\pi^2),O(\pi^2),O(\pi^2)),
\end{align*}
all vanishing to the certified precision returned by the algorithm, as they must: for $i=1,2,3$ we have $\mathrm{div}(x-x(P_i))=4P_i-4P_4$, so all classes $[P_i-P_j]$ are $4$-torsion. The very bad endpoints are replaced by suitable bad points over $F_{\mathfrak P}(\pi^{1/e})$, where the code takes $e=2qe_{\widetilde P}+1=25$ (here $q=3$ and $e_{\widetilde P}=4$, see \cref{algthm:suitablebadpoint} \ref{algthm:suitablebadpoint 2} and \ref{rmk:e}). Evaluating at these points near the boundary of their residue discs loses most of the working precision; the code suggests $e=452$ to recover about $19$ digits.
\end{example}

\subsection{Diophantine applications}\label{sec:applications}

In this subsection, we will give an example of searching rational divisors on a curve using our code. 

We fix a number field $F$ (we take $F=\Q(\sqrt5)$ in \cref{ex:app-singleprime}). By a \emph{quadratic point} of a curve $C/F$, we mean a closed point $P\in C$ with $[F(P):F]=2$. Quadratic points correspond to the irreducible effective $F$-rational divisors of degree $2$. The set $C^{(2)}(F)$ also contains sums of two $F$-rational points, with repetitions allowed.

\begin{example}[Recovering sporadic divisors at a single prime]
	\label{ex:app-singleprime}
	Put $F\coloneq\Q(s),\quad s^2=5.$
	A smooth plane quartic model of \(C=X(s3,ns5)\) over \(\Q\) is
	\begin{align}\label{eq:quartic model of Xs3ns5}
		f(x,y)\coloneq y^4-2xy^3-x^2y^2-8y^3-2xy^2+2x^2y+x^3
		+3y^2+3xy-3x^2-2y+3x-1=0,
	\end{align}
	see \cite[Proposition~7.4.3]{caraiani2025on},  where the roles of $x$ and $y$ are interchanged.
	The curve \(C\) has genus \(3\).
	We denote by $\mathcal C$ the projective model over $\Z$ defined by the homogenization of \eqref{eq:quartic model of Xs3ns5}.
	For any prime $p\neq3,5$, its reduction $\mathcal C_p$ is smooth over $\mathbb F_p$.
	
	In this example, we show how our implementation of Coleman integration over arbitrary local fields allows us to recover all $35$ affine sporadic divisors in $C^{(2)}(F)$ using only the prime $19$.
	
	The basic fact supporting the computation is that, for a fixed
	$D_0\in C^{(2)}(F)$ and every differential $\omega$ annihilating
	$J(C)(F)$, we have
	\[
	\int_{D_0}^{D}\omega=0
	\qquad\text{for every }D\in C^{(2)}(F),
	\]
	see \cref{vanishing differentials}.
	In our example, the annihilating space has dimension two.
The argument in \eqref{eq:dim of annihilating dif}, given there over $F_v\cong\Q_p$, remains valid after extending scalars to the unramified quadratic field used here.

	In residue classes where the Jacobian matrix of the two	Coleman-integral functions is invertible modulo $p$, we compute their common zero by $p$-adic Newton iteration.
	We then recognize the symmetric functions of its coordinates as elements of $F$ and verify the resulting divisor exactly on $C$.
	
	The multiple-primes sweep in \texttt{multiprime-Xs3ns5.m} computes integrals over $F_v\cong\Q_p$ at eight places above the split primes
	\[
	p\in\{19,29,31,41,59,61,71,79\}.
	\]
	At each place, it searches pairs of distinct $\mathbb F_p$-rational residue discs satisfying the rank-two condition. Combining the results recovers $35$ distinct affine sporadic divisors.
	
	More precisely, the log \texttt{example-8.4-multiprime.log} records the minimal polynomials
\[f_i(X)\in F[X], \qquad 1\leq i\leq35,\]
of the $x$-coordinates of the recovered points. They are irreducible and quadratic; they are listed in \cref{tab:sporadic-divisors}, and we use the numbering there.
	For each $i$, let $\alpha_i$ be a root of $f_i(X)$, put $F_i\coloneq F(\alpha_i)$, and let $P_i\in C(F_i)$ be the recovered affine point with $x(P_i)=\alpha_i$.
	Writing $\tau_i$ for the nontrivial element of $\Gal(F_i/F)$, we denote the corresponding divisor by
	\[
	\Delta_i\coloneq P_i+P_i^{\tau_i}\in C^{(2)}(F).
	\]
	
	Our implementation allows us to replace this sweep by a	computation at a single place.
	Since $19$ splits in $F$, choose $v_{19}\mid19$, so that $F_{v_{19}}\cong\Q_{19}$, and let $K=\Q_{19^2}$ be its unramified quadratic extension.
	The reduction of an $F$-rational effective divisor of degree $2$ has one of the following forms:
	\[
	\widetilde D=\widetilde R_1+\widetilde R_2,
	\qquad
	\widetilde R_1,\widetilde R_2\in \mathcal C_{19}(\mathbb F_{19}),
	\]
	or
	\[
	\widetilde D=\widetilde R+\widetilde R^{(19)},
	\qquad
	\widetilde R\in
	\mathcal C_{19}(\mathbb F_{19^2})
	\setminus\mathcal C_{19}(\mathbb F_{19}).
	\]
	We call these residue classes \emph{split} and \emph{Frobenius-conjugate}, respectively.
	Computing integrals over $K$ allows us to search pairs of distinct residue discs of both types at the same prime.
	
	For example, consider $\Delta_4=P_4+P_4^{\tau_4}$, for which the minimal polynomial of $x(P_4)$ is
	\[
	f_4(X)=X^2-\frac{3}{2}X+\frac{63}{2}.
	\]
	Its discriminant is
	\[
	-\frac{495}{4}=\left(\frac{3}{2}\right)^2(-55),
	\]
	so $F_4=F(\sqrt{-55})$.
	Since $-55\equiv2\pmod{19}$ and $2$ is not a square in $\mathbb F_{19}$, $\Delta_4$ reduces to a Frobenius-conjugate class. Its constituent points are therefore accessible over $K$, although they are not $\Q_{19}$-rational.
	
	The single-prime sweep over $K$ recovers the same $35$ affine sporadic divisors $\Delta_1,\ldots,\Delta_{35}$: $12$ lie in split residue classes and $23$ in Frobenius-conjugate residue classes.
	Each recovered divisor is recognized over $F$ and verified exactly on $C$.
	Thus, computing Coleman integrals over an unramified quadratic extension allows us to recover all $35$ affine sporadic divisors using only one prime instead of eight.
	The proof that these exhaust the affine sporadic divisors uses the Mordell--Weil sieve together with symmetric Chabauty; see \cref{prop:quadratic points over Q(sqrt 5)}.
\end{example}

\appendix

\section{Quadratic points on $X(s3,ns5)$ over $\Q(\sqrt{5})$}
\label{appendix:a sc and mw}

In this appendix, we will give the complete classification of $X(s3,ns5)^{(2)}(\Q(\sqrt 5))$. The main tool to show the completeness is the Mordell--Weil sieve together with Siksek's symmetric Chabauty \cite{siksek2009chabauty}.

In this section, we fix a number field $F$.

\begin{art}[Annihilating differentials]\label{vanishing differentials}
	Let $C$ be a smooth projective curve over $F$ with $g\coloneq g(C)$, and let $J(C)$ be the Jacobian of $C$. Let $v$ be a place of $F$ above a rational prime $p$ where $C$ has good reduction, and let $F_v$ denote the completion. The \emph{space of annihilating differentials} $V\subset H^0(C_{F_v},\Omega_{C_{F_v}/F_v}^1)$ is the $F_v$-subspace of holomorphic differentials that annihilate (the $p$-adic closure) $\overline{J(C)(F)}\subset J(C)(F_v)$ under the Coleman pairing; concretely,
\begin{align}\label{eq:vanishing differential}
	V\coloneq\left\{\omega\in H^0\bigl(C_{F_v},\Omega^1_{C_{F_v}/F_v}\bigr) \;\middle|\; \int_0^D\omega=0\text{ for all }D\in J(C)(F)\right\}.
\end{align}
When the Chabauty condition $\rk J(C)(F) < g$ holds, we have $\dim_{F_v} V \geq g-\rk J(C)(F) \ge 1$, and $V$ is non-zero. 
\end{art}
\begin{art}[Trace-zero annihilating differentials]\label{trace zero vanishing differentials}
	Let $\varphi\colon C\to C'$ be a finite morphism of smooth curves over $F$ of degree $d$.
	\begin{enumerate}
		\item The \emph{trace map of function fields}
		\[
		\operatorname{Tr}_{F(C)/F(C')} \colon F(C) \to F(C')
		\]
is the $F(C')$-linear map that sends $z\in F(C)$ to the trace of multiplication by $z$ on the finite-dimensional $F(C')$-vector space $F(C)$. Equivalently, in characteristic zero it is the sum of the images of $z$ under the $F(C')$-embeddings into a normal closure.

		\item The \emph{trace map for differential 1-forms}
\[
\operatorname{Tr}_\varphi \colon \varphi_*\Omega^1_{C/F} \to \Omega^1_{C'/F}
\]
is defined via the function-field trace, as in \cite[\S 2.6, before Corollary 2.14]{box2023cubic}: every rational differential can be written as $\eta=z\cdot\varphi^*\alpha\in \Omega^1_{F(C)/F} = F(C)\otimes_{F(C')} \varphi^*\Omega^1_{F(C')/F}$ with $z\in F(C), \alpha\in  \Omega^1_{F(C')/F}$, and we put
\[
\operatorname{Tr}_\varphi\big(z \cdot \varphi^*\alpha\big)
\coloneq \operatorname{Tr}_{F(C)/F(C')}(z)\cdot \alpha.
\]
This is well defined, and $\Tr_\varphi$ maps regular differentials to regular differentials. Restricting to global sections gives a linear map on holomorphic differentials:
\[
\operatorname{Tr}_\varphi \colon H^0(C,\Omega^1_{C/F}) \to H^0(C',\Omega^1_{C'/F}).
\]
		\item Let $v$ be a place of $F$ above a rational prime $p$ where $C, C'$ have good reduction. The space of \emph{trace-zero annihilating differentials} $V_\varphi\subset H^0(C_{F_v},\Omega_{C_{F_v}/F_v}^1)$ is the intersection of the space of vanishing differentials $V$ of $C$ (see \eqref{eq:vanishing differential}) and the kernel of the push-forward map 
		\[\mathrm{Tr}_{\varphi,v}\colon H^0(C_{F_v},\Omega^1_{C_{F_v}/F_v})\to H^0(C'_{F_v},\Omega^1_{C_{F_v}'/F_v}).\]
	\end{enumerate}
	Notice that for any $\omega\in H^0(C',\Omega^1_{C'/F})$, we have $\Tr_\varphi(\varphi^*\omega)=d\cdot \omega$, so $\Tr_\varphi$ is surjective, and $$H^0(C,\Omega^1_{C/F}) = \varphi^*H^0(C',\Omega^1_{C'/F}) \oplus \Ker(\Tr_\varphi).$$
	If $\varphi$ is of degree $d=2$, then we have an involution $\sigma\colon C\to C$ (i.e. $\sigma^2=\id_C$) over $C'$ induced by the nontrivial element in $\Gal(F(C)/F(C'))$. Then $\varphi^*\Tr_\varphi = \id_C^*+\sigma^*$ on $H^0(C,\Omega_{C/F}^1)$. By linear algebra, one has
	\begin{align}\label{eq:description of ker tr}
		\Ker(\Tr_\varphi) = H^0(C,\Omega_{C/F}^1)^{-}\coloneq \left\{\omega\in H^0(C,\Omega_{C/F}^1) \;\middle|\; \sigma^*\omega=-\omega\right\}.
	\end{align}
\end{art}

\begin{art}
Let $\widetilde C$ be a smooth curve over a field $k$, let $R\in\widetilde C(k)$, and let $\omega\in H^0(\widetilde C,\Omega^1_{\widetilde C/k})$. For a uniformizer $t_R$ at $R$, write $\omega=\sum_{j\geq0}a_jt_R^j\,dt_R$. For an integer $h\geq1$, define
\[a(\omega,t_R,h)\coloneq(-a_0,a_1,\dots,(-1)^ha_{h-1}).\]

\end{art}

We are ready to state Siksek's symmetric Chabauty.

\begin{art}[Siksek's symmetric-Chabauty criterion]
	
	Let $C$ be a smooth projective curve over $F$ satisfying the Chabauty condition $\rk J(C)(F)<g(C)$, $m\in \N_{\geq 1}$, $v$ a place of $F$ above a rational prime $p$, and $\varphi\colon C\to C'$ be a finite morphism of projective curves over $F$ of degree $d\geq 1$. Let $\ell\in\N$ such that $e\coloneq m-d\ell\geq 0$. Let $D_{\mathrm{res}}\in C^{(e)}(F)$, $D_{\mathrm{pull}}\in \varphi^*C'^{(\ell)}(F)$ and set
	\[
	D\coloneq D_{\mathrm{res}}+D_{\mathrm{pull}}\in C^{(m)}(F).
	\]
	After collecting multiplicities, write $D=n_1R_1'+\cdots+n_kR_k'$ with $R_1',\dots, R_k'\in C(\overline{F})$ distinct geometric points and $n_i\geq 1$. Assume the following statements hold.
	\begin{enumerate}
		\item\label{MW i} $C$ and $C'$ have good reductions at $v$.
		\item\label{MW ii} No point in the support of $D_{\mathrm{pull}}$ has ramification degree under $\varphi$ divisible by $p$.
		\item\label{MW iii} Let $N$ be the maximum of the ramification indices of $p$ in the fields $F(R_i',R_j')$, for $1\leq i,j\leq k$, at places above $v$. Assume that $p\geq N+2$. Here $F(R_i',R_j')$ is obtained by adjoining the coordinates of the two support points.
		
		\item Let $\omega_1,\dots, \omega_r\in \mathcal V_\varphi\coloneq H^0(\mathcal{C},\Omega^1_{\mathcal C/F_v^\circ})\cap V_\varphi$ be an $F_v^\circ$-basis of $\mathcal V_\varphi$, where $\mathcal C$ is a smooth $F_v^\circ$-model of $C$, and $V_\varphi$ is the space of trace-zero annihilating differentials. Assume the matrix
		\begin{align}\label{eq:symmetric chabauty matrix}
			\widetilde{\mathscr A}_D\coloneq 
			\begin{bmatrix}
				a(\widetilde{\omega_1}, t_{\widetilde{R_1'}}, n_1)&a(\widetilde{\omega_1}, t_{\widetilde{R_2'}}, n_2)&\cdots & a(\widetilde{\omega_1}, t_{\widetilde{R_k'}}, n_k)\\
				a(\widetilde{\omega_2}, t_{\widetilde{R_1'}}, n_1)&a(\widetilde{\omega_2}, t_{\widetilde{R_2'}}, n_2)&\cdots & a(\widetilde{\omega_2}, t_{\widetilde{R_k'}}, n_k)\\
				\vdots& \vdots& \ddots&\vdots\\
				a(\widetilde{\omega_r}, t_{\widetilde{R_1'}}, n_1)&a(\widetilde{\omega_r}, t_{\widetilde{R_2'}}, n_2)&\cdots & a(\widetilde{\omega_r}, t_{\widetilde{R_k'}}, n_k)
			\end{bmatrix} 
		\end{align}
		has rank $m-\ell$, where $\widetilde{\omega_i}$ is the reduction of $\omega_i$ modulo $v$, $\widetilde{R_j'}$ is the reduction of $R_j'$ modulo a fixed place of $F(R_1',\dots,R_k')$ above $v$, and $t_{\widetilde{R_j'}}$ is a uniformizer at the reduction $\widetilde{R_j'}$ for each $j$.
	\end{enumerate}
	Then (see \cite[Theorem 2.6]{box2023cubic}) every $P\in C^{(m)}(F)$ in the mod $F_v^{\circ\circ}$ residue disc of $D$ is contained in 
	\[D_{\mathrm{res}}+\varphi^*C'^{(\ell)}(F).\]

In our case, i.e. the proof of
\cref{prop:quadratic points over Q(sqrt 5)},
the conditions \ref{MW i}, \ref{MW ii}, \ref{MW iii}
are easy to check. Moreover, we will show that
$V=V_\varphi=\Ker\Tr_{\varphi,v}$.

After fixing an integral basis $\omega_1,\dots,\omega_r$,
we define a similar matrix directly from the reduced divisor.
Over $\overline{\widetilde{F_v}}$, write
\[
\widetilde D=\sum_{j=1}^s m_j\widetilde R_j
\]
with distinct geometric points $\widetilde R_j$ and
$m_j\geq1$. Choose a uniformizer $t_{\widetilde R_j}$
at each $\widetilde R_j$, and put
\[
\widetilde{\mathscr A}(\widetilde D)\coloneq
\begin{bmatrix}
	a(\widetilde\omega_1,t_{\widetilde R_1},m_1)
	&\cdots&
	a(\widetilde\omega_1,t_{\widetilde R_s},m_s)\\
	\vdots&\ddots&\vdots\\
	a(\widetilde\omega_r,t_{\widetilde R_1},m_1)
	&\cdots&
	a(\widetilde\omega_r,t_{\widetilde R_s},m_s)
\end{bmatrix}.
\]
We call this the \emph{symmetric Chabauty matrix} of
$\widetilde D$ with respect to the chosen differentials
and uniformizers.

If distinct points in the support of $D$ have distinct
reductions, then $\widetilde{\mathscr A}_D$ agrees with
$\widetilde{\mathscr A}(\widetilde D)$ up to a permutation
of columns, using the same uniformizers.
Otherwise, the two matrices may differ: the rank condition
in the criterion above concerns $\widetilde{\mathscr A}_D$.
\end{art}

\begin{proposition}\label[pprop]{prop:quadratic points over Q(sqrt 5)}
	Let $F=\Q(\sqrt5)$ and $C=X(s3,ns5)$, with the model given in
	\eqref{eq:quartic model of Xs3ns5}.
	Let $\varphi\colon C\to E$ be the degree-$2$ quotient map
	associated with the involution
	\[
	\sigma([X:Y:Z])
	=
	[X+8Y-8Z:X+3Y+2Z:-2X+4Y+Z],
	\]
	where $E=X(ns3,ns5)$ has model $y^2+y=x^3+1$.
	Then $C(F)=\{Q_1,\ldots,Q_8\}$, where, with $s=\sqrt5$,
	\begin{gather*}
		Q_1=(1,0),\quad Q_2=(7,-3),\quad Q_3=\bigl(-\tfrac12,-\tfrac12\bigr),\quad Q_4=(1:0:0),\\
		Q_5=\bigl(\tfrac{s-1}{2},2-s\bigr),\quad Q_6=\bigl(\tfrac{3s+9}{2},-\tfrac{s+1}{2}\bigr),\quad Q_7=\bigl(\tfrac{9-3s}{2},\tfrac{s-1}{2}\bigr),\quad Q_8=\bigl(-\tfrac{s+1}{2},s+2\bigr),
	\end{gather*}
	and $\sigma(Q_{2i-1})=Q_{2i}$ for $1\leq i\leq4$. Moreover,
	\[
	C^{(2)}(F)
	=
	\mathscr K_2
	\cup
	\varphi^*\bigl(E(F)\bigr),
	\qquad
	\mathscr K_2
	\coloneq
	\left\{
	Q_i+Q_j\;\middle|\;1\leq i\leq j\leq 8
	\right\}
	\cup
	\{\Delta_1,\ldots,\Delta_{35},D_\infty\},
	\]
	where $\Delta_1,\ldots,\Delta_{35}$ are the affine sporadic
	divisors recovered in \cref{ex:app-singleprime} and listed in \cref{tab:sporadic-divisors}, and
	\[
	D_\infty
	=
	(1:1+\sqrt2:0)+(1:1-\sqrt2:0).
	\]
\end{proposition}
\begingroup
\small
\renewcommand{\arraystretch}{1.25}
\begin{longtable}{r|>{$}l<{$}|>{$}l<{$}}
	\caption{The $35$ affine sporadic divisors
		$\Delta_i=P_i+P_i^{\tau_i}$ of \cref{ex:app-singleprime},
		numbered as in the list \texttt{mps} of
		\texttt{appendix-A-common.m}. Here $f_i$ is the minimal
		polynomial of $x=x(P_i)$ over $F=\Q(s)$, $s^2=5$, and
		$y=y(P_i)$, the only root of $f(x(P_i),Y)$ in $F_i$,
		is given in terms of $x$.}
	\label{tab:sporadic-divisors}\\
	$i$ & f_i(X) & y \\ \hline
	\endfirsthead
	
	$i$ & f_i(X) & y \\ \hline
	\endhead
	
	1&X^2-11X-47&-3\\
	2&X^2-X-2s-2&-\tfrac{s+1}{2}\\
	3&X^2-X+2s-2&\tfrac{s-1}{2}\\
	4&X^2-\tfrac{3}{2}X+\tfrac{63}{2}&1-\tfrac{1}{3}x\\
	5&X^2-\tfrac{59}{7}X+\tfrac{25}{7}&-\tfrac{1}{3}+\tfrac{1}{3}x\\
	6&X^2-\tfrac{27}{28}X+\tfrac{1}{4}&1-2x\\
	7&X^2-\tfrac{23}{7}X+\tfrac{47}{28}&-\tfrac{1}{4}+\tfrac{1}{2}x\\
	8&X^2+(s+6)X+8s+4&\tfrac{3s+7}{2}+\tfrac{s+1}{2}x\\
	9&X^2+(6-s)X-8s+4&\tfrac{7-3s}{2}+\tfrac{1-s}{2}x\\
	10&X^2+(6s-4)X-12s+19&-s+1+\tfrac{s+1}{2}x\\
	11&X^2-(6s+4)X+12s+19&s+1+\tfrac{1-s}{2}x\\
	12&X^2+\tfrac{11s-47}{8}X-4s+9&\tfrac{s-3}{2}+(s+3)x\\
	13&X^2+\tfrac{14s-27}{3}X-6s+14&-\tfrac{s+3}{2}-(s+2)x\\
	14&X^2-\tfrac{11s+47}{8}X+4s+9&-\tfrac{s+3}{2}+(3-s)x\\
	15&X^2-\tfrac{14s+27}{3}X+6s+14&\tfrac{s-3}{2}+(s-2)x\\
	16&X^2+\tfrac{7s+17}{2}X-62s-146&-\tfrac{7s+13}{2}+\tfrac{s+1}{2}x\\
	17&X^2+\tfrac{33s-68}{2}X-33s+69&-\tfrac{13s+3}{22}+\tfrac{2s+3}{11}x\\
	18&X^2+\tfrac{17-7s}{2}X+62s-146&\tfrac{7s-13}{2}+\tfrac{1-s}{2}x\\
	19&X^2-\tfrac{33s+68}{2}X+33s+69&\tfrac{13s-3}{22}+\tfrac{3-2s}{11}x\\
	20&X^2-\tfrac{2s+19}{2}X+\tfrac{s+25}{4}&s-2+(2-s)x\\
	21&X^2+\tfrac{2s-19}{2}X+\tfrac{25-s}{4}&-s-2+(s+2)x\\
	22&X^2+\tfrac{3s-7}{6}X-\tfrac{6s+10}{3}&\tfrac{1-s}{2}+\tfrac{s-1}{2}x\\
	23&X^2-\tfrac{3s+7}{6}X+\tfrac{6s-10}{3}&\tfrac{s+1}{2}-\tfrac{s+1}{2}x\\
	24&X^2+\tfrac{5s-17}{2}X+\tfrac{67s-149}{2}&-s+2\\
	25&X^2-\tfrac{5s+17}{2}X-\tfrac{67s+149}{2}&s+2\\
	26&X^2+\tfrac{7s-176}{79}X+\tfrac{76s+256}{79}&\tfrac{7-s}{22}+\tfrac{3s+1}{22}x\\
	27&X^2+\tfrac{97-51s}{62}X+\tfrac{147s-327}{62}&\tfrac{s-3}{2}+(s+3)x\\
	28&X^2+\tfrac{51s+97}{62}X-\tfrac{147s+327}{62}&-\tfrac{s+3}{2}+(3-s)x\\
	29&X^2-\tfrac{7s+176}{79}X+\tfrac{256-76s}{79}&\tfrac{s+7}{22}+\tfrac{1-3s}{22}x\\
	30&X^2-\tfrac{137s+223}{82}X+\tfrac{59s+603}{82}&-\tfrac{13s+3}{22}+\tfrac{2s+3}{11}x\\
	31&X^2+\tfrac{137s-223}{82}X+\tfrac{603-59s}{82}&\tfrac{13s-3}{22}+\tfrac{3-2s}{11}x\\
	32&X^2+\tfrac{479-283s}{24}X+\tfrac{181s-179}{24}&\tfrac{7s+16}{11}-\tfrac{s+7}{11}x\\
	33&X^2+\tfrac{283s+479}{24}X-\tfrac{181s+179}{24}&\tfrac{16-7s}{11}+\tfrac{s-7}{11}x\\
	34&X^2-\tfrac{793s+1663}{82}X+\tfrac{991s+2123}{82}&-\tfrac{s+4}{11}+\tfrac{3-2s}{11}x\\
	35&X^2+\tfrac{793s-1663}{82}X+\tfrac{2123-991s}{82}&\tfrac{s-4}{11}+\tfrac{2s+3}{11}x\\
\end{longtable}
\endgroup
\begin{proof}
	Let $\mathscr T_0\subseteq C^{(2)}(F)$ be the union listed in
	the proposition. The divisors in $\mathscr T_0$ have been
	verified on $C$, so it remains to prove
	$C^{(2)}(F)\subseteq\mathscr T_0$.
	We use symmetric Chabauty and the Mordell--Weil sieve;
	no numerical Coleman integration is needed for this
	completeness argument.
	
	Let $\mathcal C$ and $\mathcal E$ be the projective models
	over $\Z$ defined by the homogenizations of
	\eqref{eq:quartic model of Xs3ns5} and $y^2+y=x^3+1$,
	respectively. Both have smooth reduction away from $3$ and $5$.
	The involution $\sigma$ is induced by
	\begin{equation}\label{eq:M}
		M=
		\begin{pmatrix}
			1&8&-8\\
			1&3&2\\
			-2&4&1
		\end{pmatrix},
	\end{equation}
	and its quotient map $\varphi\colon C\to E$ reduces to
	$\varphi_p\colon\mathcal C_p\to\mathcal E_p$ at these primes.
	
	We use the primes
	\[
	\mathcal P=\{19,29,31,41,59,61,71,79,89\}.
	\]
	Each splits in $F$. For every $p\in\mathcal P$, fix a place
	$v_p\mid p$, so that $F_{v_p}\cong\Q_p$, and write
	$\mathcal J_p\coloneq J(\mathcal C_p)$, the Jacobian of $\mathcal C_p$.
	Denote the reduction maps by
	\[
	\red_p^{(2)}\colon C^{(2)}(F)\to
	\mathcal C_p^{(2)}(\mathbb F_p),
	\qquad
	\red_p\colon J(C)(F)\to\mathcal J_p(\mathbb F_p).
	\]

By \cite[Theorem~1.1]{bastianelli2014gonality}, the \emph{$F$-gonality} of $C$
\[
\operatorname{gon}_F(C)\geq \operatorname{gon}_\C(C) = 4-1=3>2,
\]
so the degree-$2$ Abel--Jacobi map
\[
C^{(2)}(F)\longrightarrow\Pic^2(C)(F)
\]
is injective. Similarly, for a prime $p\not=3,5$, by \cite[Corollary~2.4]{pellikaan1996on} (by taking $j=1, i=1$ in the corollary), \(
\operatorname{gon}_{\mathbb F_p}(\mathcal{C}_p)\geq \operatorname{gon}_{\overline{\mathbb F_p}}(\mathcal{C}_p) >2,
\) and 	\(
\mathcal C_p^{(2)}(\mathbb F_p)\longrightarrow\Pic^2(\mathcal C_p)(\mathbb F_p)
\)
is injective.

	\begin{enumerate}
		\item \textbf{The Mordell--Weil groups.}
		
		The computation \texttt{verify-rank-E.m} gives
		\begin{equation}\label{eq:E(F)}
			E(F)=\Z P_E\oplus\langle T_E\rangle,
			\qquad \ord(T_E)=3,
		\end{equation}
		where
		\[
		P_E=\left(\frac{-3s+5}{2},\,3s-8\right),
		\qquad
		T_E=\left(0,\frac{-s-1}{2}\right).
		\]
		By the proof of
		\cite[Proposition~7.4.4]{caraiani2025on},
		\begin{equation}\label{eq:isogeny}
			J(C)\sim_\Q E\times E_{75b}\times E_{75c}.
		\end{equation}
		where $E_{75b}$ (resp. $E_{75c}$) is an elliptic curve in the isogeny class with Cremona label $75b$ (
resp. $75c$). By the twist decomposition, see \cite[Exercise 10.16]{silverman2009arithmetic}, and according to \cite{lmfdb}, we have
\[\rk E_{75b}(F) = \rk E_{75b}(\Q)+\rk E^{(5)}_{75b}(\Q) =  \rk E_{75b}(\Q)+\rk E_{75a}(\Q)=0.\]
Similarly, $\rk E_{75c}(F) = 0$.
Hence
\[
\operatorname{rk}J(C)(F)=\rk E(F)+\rk E_{75b}(F)+\rk E_{75c}(F)=1.
\]
		
Let $O_E$ denote the identity element of $E$. We fix the translation and sign of the quotient map
$\varphi\colon C\to E$ by requiring
\begin{align}\label{eq:phi normalisation}
	\begin{split}
		&\varphi\bigl((1,0)\bigr)=O_E,\\
		&\varphi^*(P_E)=\Bigl\{x^2-(6s+34)x-18s-26=0,\ y=\tfrac{(s-1)x-2s-2}{8}\Bigr\}.
	\end{split}
\end{align}
Explicitly, away from the line $X+2Y-Z=0$, we have $\varphi(X:Y:Z)=(x_\varphi,y_\varphi)$ with
\begin{align*}
	x_\varphi&=-\frac{X^2+3XY+Y^2-4XZ-2YZ+2Z^2}{(X+2Y-Z)^2},\\
	y_\varphi&=-\bigl(X^3+9X^2Y+21XY^2+12Y^3-3X^2Z\\
	&\qquad-18XYZ-12Y^2Z+6YZ^2+Z^3\bigr)\big/(X+2Y-Z)^3,
\end{align*}
see \texttt{appendix-A-covering.m}. Since $\sigma((1,0))=(7,-3)$, we may put
\begin{align}\label{eq:D0}
	D_0\coloneq(1,0)+(7,-3)=\varphi^*(O_E)\in\varphi^*\bigl(E(F)\bigr)\subset C^{(2)}(F),
\end{align}
and define
\begin{align}\label{eq:Dspor}
	D_{\mathrm{spor}}\coloneq [\Delta_{20} - D_0], \quad D_{\mathrm{fam}}\coloneq [\varphi^*P_E-D_0]=\varphi^*\bigl([P_E-O_E]\bigr) \in J(C)(F),
\end{align}
where $\Delta_{20}=\bigl\{x^2-\frac{2s+19}{2}x+\frac{s+25}{4}=0,\ y=(s-2)(1-x)\bigr\}$ is the divisor with $i=20$ in \cref{tab:sporadic-divisors}.
Also put
\begin{align}\label{eq:t10}
	t_{10}\coloneq \left[(7,-3) - (1,0)\right], \quad D_z\coloneq D_{\mathrm{spor}}+ D_{\mathrm{fam}}.
\end{align}

The computation \texttt{appendix-A-generators.m} verifies that
$t_{10}$ and $D_z$ have orders $10$ and $120$.
Reduction at suitable good primes gives
\[
\#J(C)(F)_{\mathrm{tors}}\leq240,
\]
while the reductions of $t_{10}$ and $D_z$ generate a
group of order $240$. Moreover, the reduction of $t_{10}$
has order $2$ modulo the subgroup generated by $D_z$.
It follows that
\[
T\coloneq J(C)(F)_{\mathrm{tors}}
=\langle5t_{10},D_z\rangle
\cong\Z/2\Z\oplus\Z/120\Z.
\]

In particular, from the discussion above, the symmetric Chabauty condition for degree-$2$ points (see \cite[\S 2.3]{box2023cubic}) holds:
\[g(C)-\operatorname{rk}J(C)(F) = 2 >1.\]

		\item \textbf{Trace-zero annihilating differentials.}

		Let $p\in\mathcal P$,
		\[
		V_p
		=
		\left\{
		\omega\in H^0(C_{\Q_p},\Omega^1_{C_{\Q_p}/\Q_p})
		\;\middle|\;
		\int_\gamma\omega=0
		\text{ for every }\gamma\in J(C)(F)
		\right\}
		\]
		be the full space of annihilating differentials.
		Let
		\[
		\Tr_{\varphi,p}\colon
		H^0(C_{\Q_p},\Omega^1_{C_{\Q_p}/\Q_p})
		\longrightarrow
		H^0(E_{\Q_p},\Omega^1_{E_{\Q_p}/\Q_p})
		\]
		be the base change to \(\Q_p\) of the trace map associated with the
		degree-\(2\) morphism
		\(
		\varphi\colon C\longrightarrow E.
		\)
		The space of trace-zero annihilating differentials is, by definition,
		\[
		V_{p,\varphi}\coloneq
		V_p\cap\Ker(\Tr_{\varphi,p}).
		\]
		
		\begin{itemize}
			\item In the present example, every trace-zero differential is
			annihilating. Consequently,
			\[
			\boxed{
				V_p=V_{p,\varphi}
				=
				\Ker(\Tr_{\varphi,p})
				=
				\Ker(\Tr_\varphi)\otimes_{\Q}\Q_p.
			}
			\]
			
			Indeed, the Mordell--Weil computation gives
			\[
			\rk J(C)(F)=\rk E(F)=1,
			\]
			then the isogeny \eqref{eq:isogeny} implies that
			\[
			J(C)(F)/\varphi^*E(F)
			\]
			is finite. Choose \(N\geq1\) annihilating this quotient. For every
			\(\gamma\in J(C)(F)\), there is then a point
			\(D\in E(F)=J(E)(F)\) such that
			\(
			N\gamma=\varphi^*D.
			\)
			Let \(\omega\in\Ker(\Tr_{\varphi,p})\). Compatibility of Coleman
			integration with pullback and trace (see \cite[Proposition~2.5~(ii)]{box2023cubic}) gives
			\[
			\int_{\varphi^*D}\omega
			=
			\int_D\Tr_{\varphi,p}(\omega).
			\]
			Hence
			\[
			N\int_\gamma\omega
			=
			\int_{N\gamma}\omega
			=
			\int_{\varphi^*D}\omega
			=
			\int_D\Tr_{\varphi,p}(\omega)
			=
			0.
			\]
			Since \(N\neq0\) in \(\Q_p\), it follows that
			\[
			\int_\gamma\omega=0.
			\]
			This holds for every \(\gamma\in J(C)(F)\), and therefore
			\[
			\Ker(\Tr_{\varphi,p})\subseteq V_p.
			\]
	To check the dimension, since $P_E\in E(F)$ is non-torsion, its $p$-adic logarithm is nonzero, so there exists $\eta\in H^0(E_{\Q_p},\Omega^1_{E_{\Q_p}/\Q_p})$ such that $\int_{P_E}\eta\neq0$. Put $D=\varphi^*P_E\in J(C)(F)$.
	By functoriality of integration and $\varphi_*\varphi^*=[2]$, we have
	\[\int_D\varphi^*\eta
	=\int_{\varphi_*D}\eta
	=2\int_{P_E}\eta\neq0.
	\]
	Thus $\omega\mapsto\int_D\omega$ is a nonzero
	$\Q_p$-linear functional on
	$H^0(C_{\Q_p},\Omega^1_{C_{\Q_p}/\Q_p})$. Then
	\[\dim_{\Q_p} V_p\leq g-1=2.\]
Combining this with $\dim_{\Q_p}V_p\geq g(C)-r=2$, we obtain
	\begin{align}\label{eq:dim of annihilating dif}
		\dim_{\Q_p}V_p=2.
	\end{align}
	We also have
	\begin{align*}
				\dim_{\Q_p}\Ker(\Tr_{\varphi,p}) &= \dim_{\Q_p} H^0(C_{\Q_p},\Omega^1_{C_{\Q_p}/\Q_p})-
				\dim_{\Q_p}H^0(E_{\Q_p},\Omega^1_{E_{\Q_p}/\Q_p})\\&=3-1=2
			\end{align*}
			Hence, in this particular example,
			\[
			V_p=\Ker(\Tr_{\varphi,p}).
			\]
			
			\item We now give a geometric description of \(V_p\), or more precisely, of $\Ker(\Tr_{\varphi})$. 
			Since \(C\subset\P_\Q^2\) is a smooth plane quartic, the adjunction formula
			gives
			\[
			\Omega_{C/\Q}^1
			\cong
			\left(
			\Omega^2_{\P_\Q^2/\Q}\otimes\OO_{\P_\Q^2}(C)
			\right)\big|_C
			\cong
			\OO_C(-3+4)
			\cong
			\OO_C(1).
			\]
			The restriction map
			\[
			H^0(\P^2_\Q,\OO_{\P^2_\Q}(1))
			\longrightarrow
			H^0(C,\OO_C(1))
			\]
			is injective, since \(C\) is not contained in a line. Moreover, both
			spaces have dimension \(3\):
			\[
			\dim_\Q H^0(\P_\Q^2,\OO_{\P_\Q^2}(1))=3
			\]
			and
			\[
			\dim_\Q H^0(C,\OO_C(1))
			=
			\dim_\Q H^0(C,\Omega_{C/\Q}^1)
			=
			g(C)
			=
			3.
			\]
			The restriction map is therefore an isomorphism. Combining it with the adjunction
			isomorphism gives
			\[
			\boxed{H^0(C,\Omega_{C/\Q}^1)
				\cong
				H^0(\P_\Q^2,\OO_{\P_\Q^2}(1))
				=
				\left\{
				aX+bY+cZ
				\;\middle|\;
				a,b,c\in\Q
				\right\}.}
			\]
			This isomorphism can be made explicit: 
			\begin{align}
				H^0(\P_\Q^2,\OO_{\P_\Q^2}(1))
				&\xrightarrow{\;\sim\;}
				H^0(C,\Omega_{C/\Q}^1),\notag\\
				aX+bY+cZ
				&\longmapsto
				\omega_{aX+bY+cZ}
				=
				\frac{ax+by+c}{f_y(x,y)}\,dx
				=
				-\frac{ax+by+c}{f_x(x,y)}\,dy,
				\label{eq:iso of differentials}
			\end{align}
			
			where 
			\(
			f_x\coloneq\frac{\partial f}{\partial x},
			f_y\coloneq\frac{\partial f}{\partial y}.
			\) 
			Let 	\(
			\ell=aX+bY+cZ
			\in 	H^0(\P_\Q^2,\OO_{\P_\Q^2}(1)).\)
			The zero divisor of \(\omega_\ell\) is the hyperplane section cut out
			by the line
			\[
			L_\ell=\{\ell=0\};
			\]
			namely,
			\[
			\operatorname{div}(\omega_\ell)
			=
			C\cdot L_\ell.
			\]
			Consequently, for any closed point $R\in C$,
			\[
			\operatorname{ord}_R(\omega_\ell)
			=
			I_R(C,L_\ell).
			\]
			where \(I_R(C,L_\ell)\) is the intersection multiplicity at $R$.
			In particular,
			\[
			\omega_\ell(R)=0
			\iff
			R\in L_\ell,
			\]
			and \(\omega_\ell\) vanishes to order at least \(2\) at \(R\) if and only if \(L_\ell\) is tangent to \(C\) at \(R\).

			We now describe $\Ker(\Tr_\varphi)$ in terms of hyperplanes.
			Recall that the involution $\sigma\colon C\longrightarrow C$ is induced on the canonical plane by the matrix $M$, see \eqref{eq:M}.
			We regard the coefficients of a linear form
			\[
			\ell=aX+bY+cZ
			\]
			as the row vector \((a,b,c)\). Then
			\[
			\ell M
			=
			(a+b-2c)X
			+
			(8a+3b+4c)Y
			+
			(-8a+2b+c)Z.
			\]
			
			Since
			\[f\bigl(M(X,Y,Z)^{\mathsf t}\bigr)=625f(X,Y,Z) \qquad\text{and}\qquad \det(M)=-125,\]
			by \eqref{eq:iso of differentials}, the action of \(\sigma^*\) on regular differentials is
			\[\sigma^*\omega_\ell=\frac{\det(M)}{625}\omega_{\ell M}=-\frac{1}{5}\omega_{\ell M}.\]
			
			On the other hand, by \eqref{eq:description of ker tr},
			\[
			\Ker(\Tr_\varphi)
			=
			H^0(C,\Omega_{C/\Q}^1)^{-}\coloneq \left\{\omega\in H^0(C,\Omega_{C/\Q}^1) \;\middle|\; \sigma^*\omega=-\omega\right\}.
			\]
			
			Let \(\omega=\omega_\ell\) with
			\(
			\ell=aX+bY+cZ.
			\)
			Using the formula for \(\sigma^*\), we obtain
			\begin{align*}
				\omega_\ell\in\Ker(\Tr_\varphi)
				&\iff
				-\frac{1}{5}\omega_{\ell M}=-\omega_\ell\\
				&\iff
				\omega_{\ell M}=5\omega_\ell\\
				&\iff
				\ell M=5\ell.
			\end{align*}
			In terms of the coefficients \(a,b,c\), the last condition is
			\[
			\begin{cases}
				a+b-2c=5a,\\
				8a+3b+4c=5b,\\
				-8a+2b+c=5c.
			\end{cases}
			\]
			These three equations are equivalent to the single relation
			\[
			b=4a+2c.
			\]
			Consequently,
			\begin{align*}
				\Ker(\Tr_\varphi)
				&=
				\left\{
				\omega_{aX+bY+cZ}
				\;\middle|\;
				b=4a+2c
				\right\}\\
				&=
				\left\{
				\omega_{aX+(4a+2c)Y+cZ}
				\;\middle|\;
				a,c\in\Q
				\right\}\\
				&=
				\mathrm{Span}_{\Q}
				\left\{
				\omega_{X+4Y},
				\omega_{2Y+Z}
				\right\}.
			\end{align*}
			
			The two linear forms
			\begin{align}\label{eq:ell1 ell2}
				\ell_1=X+4Y,
				\qquad
				\ell_2=2Y+Z
			\end{align}
			have the unique common zero
			\[
			p_0\coloneq (-4:1:-2)\in\P^2_\Q(\Q)\setminus C.
			\]
			Moreover,
			\[
			(aX+bY+cZ)(p_0)
			=
			-4a+b-2c.
			\]
			Thus
			\[
			b=4a+2c
			\iff
			\ell(p_0)=0.
			\]
			We therefore obtain the geometric description
			\[
			\boxed{
				\Ker(\Tr_\varphi)
				=
				\left\{
				\omega_\ell
				\;\middle|\;
				L_\ell\text{ passes through }p_0
				\right\}.
			}
			\]
			In other words, the trace-zero differentials correspond exactly to
			the pencil of lines through \(p_0\). For the primes considered here,
			these trace-zero differentials are also annihilating, and hence
			\[
			\boxed{V_p
				=
				\Ker(\Tr_\varphi)\otimes_{\Q}\Q_p
				=
				\mathrm{Span}_{\Q_p}
				\left\{
				\omega_{X+4Y},
				\omega_{2Y+Z}
				\right\}.}
			\]
			
			Define the integral
			trace-zero annihilating lattice by
			\[
			\mathcal V_p
			\coloneq
			V_p
			\cap
			H^0\left(
			\mathcal C_{\Z_p},
			\Omega_{\mathcal C_{\Z_p}/\Z_p}^1
			\right).
			\]
			where $\mathcal C_{\Z_p}=\mathcal C\otimes_\Z\Z_p$.
			Then
			\[
			\omega_{\ell_1},\omega_{\ell_2}
			\]
			extend to regular relative differentials on \(\mathcal C_{\Z_p}\) and
			form a \(\Z_p\)-basis of \(\mathcal V_p\). Indeed, under the
			integral adjunction identification, they correspond to the integral
			linear forms
			\[
			\ell_1=X+4Y,
			\qquad
			\ell_2=2Y+Z,
			\]
			whose reductions remain linearly independent over \(\mathbb F_p\).
			
			We write
			\[
			\widetilde V_p
			\coloneq
			\mathcal V_p\otimes_{\Z_p}\mathbb F_p.
			\]
			Then
			\[
			\widetilde V_p
			=
			\Span_{\mathbb F_p}
			\left\{
			\widetilde\omega_{\ell_1},
			\widetilde\omega_{\ell_2}
			\right\},
			\]	
			
			\item We now give a uniform geometric interpretation of the possible ranks
			of the symmetric Chabauty matrix in
			\eqref{eq:symmetric chabauty matrix}. Since the homogenization of $f$ in \eqref{eq:quartic model of Xs3ns5} takes the value $-75$ at $p_0=(-4:1:-2)$, the reduction $\widetilde p_0$ of $p_0$ satisfies
			\[\widetilde p_0\in\P^2_{\mathbb F_p}(\mathbb F_p)\setminus\mathcal C_{p}(\mathbb F_p)\]
			for every prime $p\neq3,5$, in particular for every $p\in\mathcal P$.
			Let
			\(
			\widetilde D\in \mathcal C_{p}^{(2)}(\mathbb F_p)
			\)
			be a degree-\(2\) residue divisor, and let
			\(\widetilde{\mathscr A}(\widetilde D)\) be the symmetric Chabauty matrix (see \eqref{eq:symmetric chabauty matrix}) of $\widetilde D$ with
			respect to
			$\widetilde \omega_{\ell_1},\widetilde \omega_{\ell_2}$, where \(\ell_1,\ell_2\) are given in
			\eqref{eq:ell1 ell2}.  Then
			\[
			\boxed{\rk\widetilde{\mathscr A}(\widetilde D)\in\{1,2\}.}
			\]
			After base change to \(\overline{\mathbb F_p}\), write
			\[
			\widetilde D_{\overline{\mathbb F_p}}
			=
			\widetilde R_1+\widetilde R_2,
			\]
			where we allow
			\(\widetilde R_1=\widetilde R_2\). Then
			\[
			\boxed{
				\rk\widetilde{\mathscr A}(\widetilde D)=1
				\iff
				\widetilde p_0,\widetilde R_1,\widetilde R_2
				\text{ are collinear}.
			}
			\]
			When \(\widetilde R_1=\widetilde R_2\), the collinearity condition
			means that the line joining \(\widetilde p_0\) and
			\(\widetilde R_1\) is tangent to
			\(\mathcal C_{\overline{\mathbb F_p}} \coloneq \mathcal C_p\otimes_{\mathbb F_p}\overline{\mathbb F_p}\) at
			\(\widetilde R_1\). Consequently,
			\(
			\rk\widetilde{\mathscr A}(\widetilde D)=2
			\)
			if and only if no such line contains
			\(\widetilde D_{\overline{\mathbb F_p}}\), with its multiplicity.
			
			Indeed, put
			\[
			k=\mathbb F_p,
			\qquad
			\widetilde C=\mathcal C_{p},
			\]
			and let \(\overline{k}=\overline{\mathbb F_p}\). Since $\rk\widetilde{\mathscr A}(\widetilde D)$ is unchanged by extending the ground field, we may compute
			\(
			\rk\widetilde{\mathscr A}(\widetilde D)
			\)
			after base change from \(k\) to \(\overline{k}\).
			
			Regard
			\(
			\widetilde D_{\overline{k}}
			\)
			as a length-\(2\) closed subscheme of
			\(\widetilde C_{\overline{k}}\). Consider the restriction map
			\[
			\operatorname{ev}_{\widetilde D}\colon
			\widetilde V_p\otimes_k\overline{k}
			\longrightarrow
			H^0\left(
			\widetilde D_{\overline{k}},
			\left.
			\Omega_{\widetilde C_{\overline{k}}/\overline{k}}^1
			\right|_{\widetilde D_{\overline{k}}}
			\right),
			\qquad
			\widetilde\omega
			\longmapsto
			\left.\widetilde\omega\right|_{\widetilde D_{\overline{k}}}.
			\]
			Since \(\widetilde D\) has degree \(2\), the target is a
			two-dimensional \(\overline{k}\)-vector space.
			
			To see the relation with the symmetric Chabauty matrix, write
			\[
			\widetilde D_{\overline{k}}
			=
			\sum_R m_RR
			\]
			and choose a local parameter \(t_R\) at each point in its support.
			For \(i=1,2\), expand
			\[
			\widetilde\omega_{\ell_i}
			=
			\left(
			\sum_{n\geq0}a_{i,R,n}t_R^n
			\right)dt_R.
			\]
			With respect to the local basis
			\[
			\left\{
			t_R^n\,dt_R
			\;\middle|\;
			R\in\Supp(\widetilde D_{\overline{k}}),
			\ 0\leq n<m_R
			\right\},
			\]
			the restriction map is represented, up to transpose, by the
			coefficients
			\[
			a_{i,R,n},
			\qquad
			0\leq n<m_R.
			\]
			These agree, up to alternating signs in the columns, with the entries of the symmetric Chabauty matrix in
			\eqref{eq:symmetric chabauty matrix}. Therefore
			\[
			\rk\widetilde{\mathscr A}(\widetilde D)
			=
			\rk\operatorname{ev}_{\widetilde D}.
			\]
			
			It follows that
			\begin{align*}
				\rk\widetilde{\mathscr A}(\widetilde D)<2
				&\iff
				\Ker(\operatorname{ev}_{\widetilde D})\neq0\\
				&\iff
				\text{there exists }
				0\neq\widetilde\omega_\ell
				\in\widetilde V_p\otimes_k\overline{k}
				\text{ such that }
				\left.\widetilde\omega_\ell\right|_
				{\widetilde D_{\overline{k}}}=0.
			\end{align*}
			The last condition is equivalent to
			\[
			\widetilde D_{\overline{k}}
			\leq
			\operatorname{div}(\widetilde\omega_\ell).
			\]
			
			By the geometric description of the trace-zero annihilating space,
			\[
			\widetilde V_p\otimes_k\overline{k}
			=
			\left\{
			\widetilde\omega_\ell
			\;\middle|\;
			\ell(\widetilde p_0)=0
			\right\}.
			\]
			Moreover,
			\[
			\operatorname{div}(\widetilde\omega_\ell)
			=
			\widetilde C_{\overline{k}}\cdot L_\ell,
			\qquad
			L_\ell=\{\ell=0\}.
			\]
			Consequently,
			\[
			\rk\widetilde{\mathscr A}(\widetilde D)<2
			\]
			if and only if there is a line \(L\) through \(\widetilde p_0\) such
			that
			\[
			\widetilde D_{\overline{k}}
			\leq
			\widetilde C_{\overline{k}}\cdot L.
			\]
			After writing
			\[
			\widetilde D_{\overline{k}}
			=
			\widetilde R_1+\widetilde R_2,
			\]
			this is precisely the stated collinearity condition. When
			\(\widetilde R_1=\widetilde R_2\), the scheme-theoretic containment
			means that
			\[
			I_{\widetilde R_1}
			\bigl(\widetilde C_{\overline{k}},L\bigr)\geq2,
			\]
			which is equivalent to \(L\) being tangent to
			\(\widetilde C_{\overline{k}}\) at \(\widetilde R_1\).
			
			It remains to exclude rank \(0\). If
			\[
			\rk\widetilde{\mathscr A}(\widetilde D)=0,
			\]
			then both basis differentials vanish on
			\(\widetilde D_{\overline{k}}\):
			\[
			\left.\widetilde\omega_{\ell_1}\right|_
			{\widetilde D_{\overline{k}}}
			=
			\left.\widetilde\omega_{\ell_2}\right|_
			{\widetilde D_{\overline{k}}}
			=
			0.
			\]
			It would follow that
			\[
			\Supp(\widetilde D_{\overline{k}})
			\subseteq
			\widetilde C_{\overline{k}}
			\cap L_{\ell_1}\cap L_{\ell_2}.
			\]
			However,
			\[
			L_{\ell_1}\cap L_{\ell_2}
			=
			\{\widetilde p_0\},
			\]
			while, by assumption,
			\[
			\widetilde p_0\notin\widetilde C.
			\]
			Hence
			\[
			\widetilde C_{\overline{k}}
			\cap L_{\ell_1}\cap L_{\ell_2}
			=
			\emptyset,
			\]
			contradicting the fact that \(\widetilde D\) is an effective divisor
			of degree \(2\). Thus
			\[
			\rk\widetilde{\mathscr A}(\widetilde D)\neq0.
			\]
			
			Since \(\widetilde{\mathscr A}(\widetilde D)\) is a \(2\times2\) matrix, we conclude
			that
			\[
			\rk\widetilde{\mathscr A}(\widetilde D)\in\{1,2\},
			\]
			and therefore
			\[
			\rk\widetilde{\mathscr A}(\widetilde D)=1
			\iff
			\widetilde p_0,\widetilde R_1,\widetilde R_2
			\text{ are collinear},
			\]
			with the tangency convention when
			\(\widetilde R_1=\widetilde R_2\).
		\end{itemize}
		
		We classify the divisors in $\mathcal C_{p}^{(2)}(\mathbb F_p)$ as follows: $\widetilde{D}\in \mathcal C_{p}^{(2)}(\mathbb F_p)$ is of
		\begin{itemize}
			\item type \emph{(a)} if $\rk\widetilde{\mathscr{A}}(\widetilde{D})=2$;
			\item type \emph{(b)} if $\widetilde{D}\in \varphi_p^*\mathcal E_{p}(\mathbb F_p)$ (in particular, $\rk\widetilde{\mathscr{A}}(\widetilde{D})=1$);
			\item type \emph{(c)} if $\widetilde{D}\not\in \varphi_p^*\mathcal E_{p}(\mathbb F_p)$ and $\rk\widetilde{\mathscr{A}}(\widetilde{D})=1$.
		\end{itemize}

			\item \textbf{The subgroup \(\Gamma\) and the index multiplier.}	

Define
\[
\iota_2\colon
C^{(2)}(F)\longrightarrow J(C)(F),
\qquad
D\longmapsto[D-D_0],
\]
where $D_0$ is defined in \eqref{eq:D0}. We put 
\[
\Gamma
\coloneq
\langle D_{\mathrm{spor}}\rangle
\oplus
J(C)(F)_{\mathrm{tors}} = \langle D_{\mathrm{spor}}, 5t_{10}, D_{\mathrm{fam}}\rangle
\subseteq J(C)(F),
\]
where $D_{\mathrm{spor}}, t_{10}, D_{\mathrm{fam}}$ are defined in \eqref{eq:Dspor} and \eqref{eq:t10}.
We claim that
\[
\boxed{[J(C)(F):\Gamma]\mid 2.}
\]
Indeed, suppose that
$J(C)(F)$ is generated by a divisor $D_g$ and $J(C)(F)_{\mathrm{tors}}$.
There exist \(\alpha\in\Z\) and
\(D_{\mathrm{tor}}\in J(C)(F)_{\mathrm{tors}}\) such that
\[
D_{\mathrm{fam}}
=
\alpha D_g+D_{\mathrm{tor}}.
\]

Consider the pullback and pushforward homomorphisms
\[
\varphi^*\colon E(F)\longrightarrow J(C)(F),
\qquad
\varphi_*\colon J(C)(F)\longrightarrow E(F)
\]
induced by the degree-two quotient map
\(
\varphi\colon C\longrightarrow E.
\)
They satisfy
\[
\varphi_*\varphi^*=[2].
\]

Recall from \eqref{eq:E(F)} that
\[
E(F)
=
\mathbb Z P_E+E(F)_{\mathrm{tors}}.
\]
Therefore, there exist
\(\beta\in\mathbb Z\) and
\(P_{\mathrm{tor}}\in E(F)_{\mathrm{tors}}\) such that
\[
\varphi_*D_g
=
\beta P_E+P_{\mathrm{tor}}.
\]
It follows that
\begin{align}
	\varphi_*D_{\mathrm{fam}}
	&=
	\alpha\varphi_*D_g+\varphi_*D_{\mathrm{tor}}
	\nonumber\\
	&=
	\alpha\beta P_E
	+
	\bigl(
	\alpha P_{\mathrm{tor}}
	+\varphi_*D_{\mathrm{tor}}
	\bigr).
	\label{eq:varphi-Dfam-1}
\end{align}
The term in parentheses is torsion.

On the other hand, by \eqref{eq:Dspor},
\begin{equation}\label{eq:pullback-generator}
	D_{\mathrm{fam}}
	=
	\varphi^*\bigl([P_E-O_E]\bigr)
	\in\Gamma.
\end{equation}
Applying \(\varphi_*\) and using
\(\varphi_*\varphi^*=[2]\), we obtain
\begin{equation}
	\varphi_*D_{\mathrm{fam}}
	=
	2P_E.
	\label{eq:varphi-Dfam-2}
\end{equation}

Comparing
\eqref{eq:varphi-Dfam-1} and
\eqref{eq:varphi-Dfam-2},
we obtain
\[
\alpha\beta=2.
\]
In particular,
\[
\alpha\mid 2.
\]

Finally, since
\(
D_{\mathrm{fam}}
=
\alpha D_g+D_{\mathrm{tor}}\in\Gamma
\)
and
\(
J(C)(F)_{\mathrm{tors}}\subseteq\Gamma,
\)
we have
\[
\alpha D_g
=
D_{\mathrm{fam}}-D_{\mathrm{tor}}
\in\Gamma.
\]
Therefore
\[
\alpha J(C)(F)\subseteq\Gamma,
\]
\[
[J(C)(F):\Gamma]\mid |\alpha|.
\]
Since \(\alpha\mid2\), we conclude that
\[
[J(C)(F):\Gamma]\mid2.
\]

In the following, we denote
\[
\jmath\colon
C^{(2)}(F)\longrightarrow\Gamma,
\qquad
D\longmapsto2[D-D_0].
\]

			\item \textbf{The finite quotient \(\Gamma/H_{\mathcal P}\).}
			
			Define
			\[
			H_{\mathcal P}
			\coloneq
			\bigcap_{p\in\mathcal P}
			\Ker(\mathrm{red}_{p}|_\Gamma).
			\]
			Let
			\[
			q_H\colon\Gamma\longrightarrow \Gamma/H_{\mathcal P}
			\]
			be the quotient map. For any $p\in \mathcal P$, because \(H_{\mathcal P}\subseteq\Ker(\mathrm{red}_{p})\), reduction induces
			\[
			\overline{\mathrm{red}_{p}}\colon
			\Gamma/H_{\mathcal P}\longrightarrow J(\mathcal C_{p})(\mathbb F_p).
			\]
			We have the injective homomorphism	
\begin{align}\label{eq:red on quotient of Gamma:first-definition}
				\overline\red\colon \Gamma/H_{\mathcal P} \hookrightarrow \prod_{p\in\mathcal P}J(\mathcal C_p)(\mathbb F_p), \qquad \gamma \mapsto (\overline{\mathrm{red}_{p}}(\gamma))_{p\in\mathcal P}.
			\end{align}
			Via this injection, in our code, we identify $\Gamma/H_{\mathcal P}$ with a subgroup of $\prod_{p\in\mathcal P}J(\mathcal C_p)(\mathbb F_p)$.

			\item \textbf{The local divisor sets $\mathcal M_p$.}

Define
\[H_{\mathcal P}\coloneq\bigcap_{p\in\mathcal P}\Ker(\mathrm{red}_{p}|_\Gamma).\]
Let
\[
q_H\colon\Gamma\longrightarrow \Gamma/H_{\mathcal P}
\]
be the quotient map. For any $p\in \mathcal P$, because \(H_{\mathcal P}\subseteq\Ker(\mathrm{red}_{p})\), reduction induces
\[\overline{\mathrm{red}_{p}}\colon \Gamma/H_{\mathcal P}\longrightarrow J(\mathcal C_{p})(\mathbb F_p).\]
We have the injective homomorphism	
\begin{align}\label{eq:red on quotient of Gamma}
	\overline\red\colon \Gamma/H_{\mathcal P} \hookrightarrow \prod_{p\in\mathcal P}J(\mathcal C_p)(\mathbb F_p), \qquad \gamma \mapsto (\overline{\mathrm{red}_{p}}(\gamma))_{p\in\mathcal P}.
\end{align}
Via this injection, in our code, we identify $\Gamma/H_{\mathcal P}$ with a subgroup of $\prod_{p\in\mathcal P}J(\mathcal C_p)(\mathbb F_p)$.

\item \textbf{The local divisor sets \(\mathcal M_{p}\).}

For every \(p\in\mathcal P\), let
\[
\iota_{2,p}\colon
\mathcal C_{p}^{(2)}(\mathbb F_p)
\longrightarrow
J(\mathcal C_{p})(\mathbb F_p),
\qquad
\widetilde D
\longmapsto
[\widetilde D-\red_p^{(2)}(D_{0})]
\]
be the Abel--Jacobi map which is injective because of $\operatorname{gon}_{\mathbb F_p}(\mathcal{C}_p) >2$. We set $\jmath_{p}\coloneq 2\iota_{2,p}$. Write
\[
\mathcal G_{p}
\coloneq
\Im\left(
\red_p\colon \Gamma
\longrightarrow
\mathcal J_{p}(\mathbb F_p)
\right),
\]
\[\mathcal G_p^{\mathrm{rob}}\coloneq\left\{y\in \mathcal J_{p}(\mathbb F_p) \;\middle |\; 2y\in \mathcal G_p\right\}.\]
Since $[J(C)(F):\Gamma]\mid2$, we have $\mathrm{red}_p(J(C)(F))\subset \mathcal G_p^{\mathrm{rob}}$.
Define
\[
\mathcal R_{p}
\coloneq
\left(
\iota_{2,p}
\bigl(
\mathcal C_{p}^{(2)}(\mathbb F_p)
\bigr)
\cap
\mathcal G_{p}^{\mathrm{rob}}
\right)
\setminus
\iota_{2,p}\left(
\varphi_{p}^*
\bigl(\mathcal E_{p}(\mathbb F_p)\bigr)
\right).
\]
We then put
\[
\mathcal M_{p}
\coloneq
\left\{
\widetilde D
\in
\mathcal C_{p}^{(2)}(\mathbb F_p)
\;\middle|\;
\iota_{2,p}(\widetilde D)\in\mathcal R_{p}
\right\}.
\]

We claim that the containment required by the Mordell--Weil sieve holds:
\begin{equation}\label{eq:xs3ns5-local-mws-containment}
	\boxed{\operatorname{red}_{p}^{(2)}
		\bigl(
		C^{(2)}(F)\setminus\mathscr T_0
		\bigr)
		\subseteq
		\mathcal M_{p}.}
\end{equation}
Indeed, if there is $D = P+P'\in C^{(2)}(F)\setminus\mathscr T_0$ such that $\widetilde D\coloneq \operatorname{red}_{p}^{(2)}(D)\not\in \mathcal M_{p}$, i.e. $\iota_{2,p}(\widetilde D)\not\in\mathcal R_{p}$, equivalently, $\iota_{2,p}(\widetilde D)\in \iota_{2,p}\left(
\varphi_{p}^*\bigl(\mathcal E_{p}(\mathbb F_p)\bigr)
\right)$, so $\widetilde D \in \varphi_{p}^*\bigl(\mathcal E_{p}(\mathbb F_p)\bigr)$ by the injectivity of $\iota_{2,p}$.
In particular, $\rk\widetilde{\mathscr{A}}(\widetilde D)=1$. Let $F'$ be an extension of $F$ of degree at most $2$ such that $P\in C(F')$, and $w$ a place of $F'$ over $v_p$. Let $\mathrm{red}_w\colon C(F'_w)\to \mathcal C_p(\widetilde{F'_w})$ be the reduction map, and $\widetilde{P}\coloneq \mathrm{red}_w(P), \widetilde{P'}\coloneq \mathrm{red}_w(P')\in \mathcal C_p(\widetilde{F'_w})$. Because $\widetilde D \in \varphi_{p}^*\bigl(\mathcal E_{p}(\mathbb F_p)\bigr)$, 
\[\widetilde{P}=\sigma_p(\widetilde{P'}) = \mathrm{red}_w(\sigma(P')),\]
where $\sigma_p\colon \mathcal C_p\to \mathcal C_p$ is the reduction of the involution $\sigma\colon C\to C$  induced by $M$, see \eqref{eq:M}.

For any $\omega\in \Ker(\Tr_\varphi)\otimes_\Q\Q_p=V_p$, we have
\[
\int_{D_0}^{\varphi^*(\varphi(P))}\omega
=\int_{O_E}^{\varphi(P)}\Tr_{\varphi,p}(\omega)=0 ,
\]
by \cite[Proposition~2.5~(ii)]{box2023cubic}, because $D_0=\varphi^*(O_E)$.
Together with $\int_{D_0}^{D}\omega=0$ we have
\[\int_{\varphi^*(\varphi(P))}^D\omega = \int_{\varphi^*(\varphi(P))}^{D_0}\omega+\int_{D_0}^D\omega=0.\] 
Hence
\begin{align}\label{int P sigma P'}
	0=\int_{\sigma(P)}^{P'}\omega=\int_P^{\sigma(P')}\sigma^*\omega=-\int_P^{\sigma(P')}\omega.
\end{align}
The last equality follows from $\sigma^*\omega=-\omega$. 

Let $t$ be an integral local parameter centered at $P$, and $t_0=t(\sigma(P'))\in {F'_w}^{\circ\circ}$. For $\omega\in \mathcal{V}_p$, on this residue disc of $\widetilde P$, we write
\[
\omega
=
\left(
a_0+a_1t+a_2t^2+\cdots
\right)\,dt,
\]
with
\(
a_i\in {F'_w}^\circ.
\)
Since $\widetilde p_0\notin\mathcal C_p$, the lines $L_{\ell_1}$ and $L_{\ell_2}$ have no common point on $\mathcal C_p$, so $\widetilde\omega_{\ell_1}$ or $\widetilde\omega_{\ell_2}$ does not vanish at $\widetilde P$. Hence $a_0=a_0(\widetilde
P,\widetilde\omega)$ is a unit for $\omega=\omega_{\ell_1}$ or for $\omega=\omega_{\ell_2}$. We fix such an $\omega\in \mathcal V_p$.

Equation~\eqref{int P sigma P'} gives
\begin{align*}
	0
	=
	\int_{P}^{\sigma(P')}\omega
	=
	t_0
	\left(
	a_0+\frac{a_1}{2}t_0
	+\frac{a_2}{3}t_0^2+\cdots
	\right).
\end{align*}
For the primes under consideration, \(p\geq 19\), the  ramification index
\[
e_{F'_w}=e(F'_w/\Q_p)\leq 2<p-1.
\]
It follows that for every $i\geq 1$, $$v_{F'_w}\left(\frac{t_0^i}{i+1}\right) \geq i- e_{F'_w}\log_p(i+1)>0,$$
\[
\frac{a_i}{i+1}t_0^i\in {F'_w}^{\circ\circ}.
\]
Therefore
\[
a_0+\frac{a_1}{2}t_0
+\frac{a_2}{3}t_0^2+\cdots
\]
is a unit.
We conclude that \(t_0=0\), and hence
\[
P=\sigma(P').
\]
Equivalently,
\(
P'=\sigma(P).
\)
Thus
\[
D=P+\sigma(P)=\varphi^*\bigl(\varphi(P)\bigr).
\]

It remains to note that \(\varphi(P)\in E(F)\). Indeed, for every
\(\tau\in\operatorname{Gal}(\overline F/F)\), the \(F\)-rationality
of \(D\) gives
\[
\varphi^*\bigl(\tau(\varphi(P))\bigr)
=
\tau(D)
=
D
=
\varphi^*\bigl(\varphi(P)\bigr).
\]
Applying \(\varphi_*\) yields
\[
2[\tau(\varphi(P))]
=
2[\varphi(P)]
\]
as effective divisors on \(E\), so
\[
\tau(\varphi(P))=\varphi(P).
\]
Hence \(\varphi(P)\in E(F)\), and therefore
\[
D
\in
\varphi^*E(F)
\subseteq
\mathscr T_0,
\]
contrary to the choice of \(D\). This proves
\eqref{eq:xs3ns5-local-mws-containment}.

Define
\[
\mathcal Y_{p}
\coloneq
\jmath_{p}(\mathcal M_{p}) = 2\mathcal R_p.
\]
The second equality follows from the fact that $$\iota_{2,p}(\mathcal M_p) = \mathcal R_p\cap \iota_{2,p}(\mathcal C_p^{(2)}(\mathbb F_p))=\mathcal R_p.$$

We have the following diagram:
\[
\xymatrix@R=3em@C=3em{
	C^{(2)}(F)\setminus\mathscr T_0
	\ar@{^{(}->}[r]
	\ar[d]
	&
	C^{(2)}(F)
	\ar[r]^-{\jmath}
	\ar[d]^-{\operatorname{red}_{p}^{(2)}}
	&
	\Gamma
	\ar@{->>}[r]^-{q_H}
	\ar[d]^-{\red_p}
	&
	\Gamma/H_{\mathcal P}
	\ar[d]^-{\overline{\red_{p}}}
	\\
	\mathcal M_p
	\ar@{^{(}->}[r]
	&
	\mathcal C_p^{(2)}(\mathbb F_p)
	\ar[r]^-{\jmath_p}
	&
	\mathcal J_p(\mathbb F_p)
	\ar@{=}[r]
	&
	\mathcal J_p(\mathbb F_p).
}
\]
\item \textbf{The Mordell--Weil sieve.}

We define the set of classes surviving the Mordell--Weil sieve 
\begin{equation}\label{eq:xs3ns5-mws-survivors} 
	W_{\mathcal P} \coloneq \bigcap_{p\in\mathcal P} \overline{\red_p}^{-1} \bigl(\mathcal Y_{p}\bigr). 
\end{equation}
By the classical statement of the Mordell--Weil sieve, we have 
\begin{align}\label{D in WP}
	q_H\left(\jmath\left(C^{(2)}(F)\setminus \mathscr T_0\right)\right)\subset W_{\mathcal P}.
\end{align}

We compute $W_{\mathcal P}$ through the injection \eqref{eq:red on quotient of Gamma}, using the following congruence conditions.

For each \(p\in\mathcal P\), set
\[
c_p=\red_p(D_{\mathrm{spor}}),
\qquad
m_p=\operatorname{ord}(c_p),
\]
and let
\[
\mathcal Y_p
=
\jmath_p(\mathcal M_p)
=
\left\{
2c
\;\middle|\;
c\in\mathcal R_p
\right\}
\subseteq \mathcal J_p(\mathbb F_p).
\]
The local computation \texttt{appendix-A-sieve.m} gives
\[
\begin{array}{c|c|c|c}
	p
	&
	m_p
	&
	\#\mathcal R_p
	&
	\# \mathcal Y_{p}
	\\ \hline
	19 & 280 &  92 &  74\\
	29 &  40 & 344 & 289\\
	31 & 520 & 186 & 160\\
	41 & 840 & 372 & 295\\
	59 & 120 & 368 & 325\\
	61 &  40 & 108 &  86\\
	71 & 360 & 132 &  87\\
	79 & 840 & 198 & 161\\
	89 & 360 & 376 & 321
\end{array}
\]

The common modulus for the free part of \(\Gamma\) is
\[
m_{\mathcal P}
=
\operatorname{lcm}_{p\in\mathcal P}m_p
=
\operatorname{lcm}
(280,40,520,840,120,40,360,840,360)
=
32760.
\]
For a fixed torsion class \(\theta\in T\) and a prime
\(p\in\mathcal P\), define
\[
N_p(\theta)
\coloneq 
\left\{
n\in\mathbb Z/m_p\mathbb Z
\;\middle|\;
nc_p+\red_p(\theta)\in\mathcal Y_p
\right\}.
\]
Thus \(N_p(\theta)\) is the set of admissible congruence classes for the
coefficient of \(D_{\mathrm{spor}}\) at \(p\).

The sets \(N_p(\theta)\) are combined by the generalized Chinese remainder
theorem to produce
\[
N(\theta)
=
\left\{
n\in\mathbb Z/m_{\mathcal P}\mathbb Z
\;\middle|\;
n\bmod m_p\in N_p(\theta)
\text{ for every }p\in\mathcal P
\right\}.
\]
It follows that the set of surviving classes is
\[
W_{\mathcal P}
=
\left\{
\overline\red\bigl(q_H(nD_{\mathrm{spor}}+\theta)\bigr)=\left(nc_p+\red_p(\theta)\right)_{p\in\mathcal P}
\;\middle|\;
\theta\in T,\ n\in N(\theta)
\right\} \subset \prod_{p\in\mathcal P}J(\mathcal C_p)(\mathbb F_p)
\]
via the injection $\overline \red$ in \eqref{eq:red on quotient of Gamma}.
The number of distinct surviving classes is
\[
\boxed{\#W_{\mathcal P}=158.}
\]

\item \textbf{Identification of the surviving classes.}

By \texttt{appendix-A-rational-points.m}, a search finds the eight points $Q_1,\ldots,Q_8$ of \cref{prop:quadratic points over Q(sqrt 5)}. Recall the set $\mathscr K_2\subseteq\mathscr T_0$ of the $\#\mathscr K_2=36+1+35=72$ known divisors defined there.
For $p\in \mathcal P$, we put
\[
\mathscr K_{2,p}^{\mathrm{red}}
\coloneq
\left\{
\iota_{2,p}
\bigl(
\operatorname{red}_p^{(2)}(D)
\bigr)
\;\middle|\;
D\in\mathscr K_2
\right\} \subset \iota_{2,p}(\mathcal C_p^{(2)}(\mathbb F_p))\subset J(\mathcal C_p)(\mathbb F_p).
\]

For \(x\in W_{\mathcal P}\) and \(p\in\mathcal P\), define the
local fiber
\[
\mathscr F_p(x)
\coloneq
\left\{
\widetilde D\in\mathcal M_p
\;\middle|\;
\jmath_p(\widetilde D)
=
\overline{\red_{p}}(x)
\right\}.
\]


Now we are ready to define a notion that is used to complete the Mordell--Weil sieve for $C^{(2)}(F)$.

We say that the prime \(p\) \emph{settles} a class
\(x\in W_{\mathcal P}\) if every
\(\widetilde D\in\mathscr F_p(x)\) satisfies the following three
conditions:
\begin{enumerate}
	\item \(\widetilde D\) is of type \emph{(a)};
	\item
	\(
	\iota_{2,p}(\widetilde D)
	\in
	\mathscr K_{2,p}^{\mathrm{red}};
	\)
	\item if \(\widetilde D=2\widetilde R\) for some \(\widetilde R\), then the element of \(\mathscr K_2\) reducing to \(\widetilde D\) is of the form \(2Q\) with \(Q\in C(F)\).
\end{enumerate}
We claim that:
\begin{itemize}
	\item Let $x\in W_{\mathcal P}$. If there is a prime $p\in \mathcal P$ that settles $x$, then for any $D\in C^{(2)}(F)\setminus\mathscr T_0$ we have $q_{H}(\jmath(D))\not= x$. 
\end{itemize} 
In particular, if for every $x\in W_{\mathcal P}$ there is a prime $p\in \mathcal P$ that settles $x$, then for any $D\in C^{(2)}(F)\setminus\mathscr T_0$, we have $q_H(\jmath(D))\not\in W_{\mathcal P}$. Together with \eqref{D in WP}, we conclude that 
\(
C^{(2)}(F)\setminus\mathscr T_0=\emptyset.
\)
Hence,
\[C^{(2)}(F) = \mathscr T_0.\] 

To prove the claim, let $x\in W_{\mathcal P}$, and let $p\in\mathcal P$ be a prime that settles $x$. Suppose that
\(
D\in C^{(2)}(F)\setminus\mathscr T_0
\)
satisfies
\(
q_H\bigl(\jmath(D)\bigr)=x.
\)
By the definition above and the compatibility of
the Abel--Jacobi maps with reduction,
\[
\operatorname{red}_p^{(2)}(D)
\in
\mathscr F_p(x).
\]
Since \(p\) settles \(x\), there exists \(D_*\in\mathscr K_2\) such that
\[
\iota_{2,p}
\bigl(
\operatorname{red}_p^{(2)}(D)
\bigr)
=
\iota_{2,p}
\bigl(
\operatorname{red}_p^{(2)}(D_*)
\bigr).
\]
Since \(\iota_{2,p}\) is injective, it follows that
\[
\operatorname{red}_p^{(2)}(D)
=
\operatorname{red}_p^{(2)}(D_*).
\]
Hence \(D\) and \(D_*\) lie in the same residue disc. We apply the symmetric-Chabauty criterion with $m=2$, $\ell=0$ and the base point
$D_*$. The matrix $\widetilde{\mathscr A}_{D_*}$ defined by \eqref{eq:symmetric chabauty matrix},  equals $\widetilde{\mathscr A}(\widetilde D)$: if $\widetilde D=\widetilde R_1+\widetilde R_2$
with $\widetilde R_1\neq\widetilde R_2$, the two points of $D_*$ are
distinct and have these reductions; if $\widetilde D=2\widetilde R$, then
$D_*=2Q$ with $Q\in C(F)$ by condition (iii), and both matrices have the
columns $a(\widetilde\omega_i,t_{\widetilde R},2)$. Since $\widetilde D$ is
of type \emph{(a)}, this matrix has rank $2$, and the criterion gives
\[D=D_*.\]
But \(D_*\in\mathscr K_2\), contradicting \(D\notin \mathscr{K}_2\cup\varphi^*E(F)\). Therefore our claim holds.

			The computation \texttt{Xs3ns5-sieve.m} gives
			\[\begin{array}{c|c|c|c}
				p&\text{classes settled}&\max\#\mathscr F_p(x)&
				\text{fibers with an unknown element}
				\\ \hline
				19 & 158 & 2 & 0  \\
				29 & 156 & 2 & 0  \\
				31 & 138 & 2 & 16 \\
				41 & 130 & 3 & 20 \\
				59 & 132 & 2 & 8  \\
				61 & 158 & 2 & 0  \\
				71 & 106 & 3 & 28 \\
				79 & 158 & 2 & 0  \\
				89 & 150 & 3 & 8
			\end{array}
			\]
			In particular, each of the primes
			\[p=19,\ 61,\ 79\]
			individually settles every element of \(W_{\mathcal P}\).
			Therefore
			\[
			\boxed{
				C^{(2)}(F)=\mathscr T_0.
			}
			\]

Finally, let \(P\in C(F)\). By the classification just proved, \(P+Q_1\) belongs to \(\mathscr K_2\cup\varphi^*(E(F))\). If \(P+Q_1\in \mathscr K_2\), then \(P\) is one of \(Q_1,\ldots,Q_8\), since the remaining divisors in \(\mathscr K_2\) are irreducible quadratic divisors over \(F\). Otherwise, \(P+Q_1=\varphi^*(R)\) for some \(R\in E(F)\). Since \(Q_1\) lies in this fiber, \(R=\varphi(Q_1)=O_E\), and therefore
\[P+Q_1=\varphi^*(O_E)=Q_1+Q_2.\]
Thus \(P=Q_2\). Consequently, \(C(F)=\{Q_1,\ldots,Q_8\}\).			
		\end{enumerate}
	\end{proof}

	\section*{Acknowledgements}
We thank Veronika Ertl for her help with a question in rigid cohomology and Steffen M\"uller for helpful discussions. We also thank Alex Best and Enis Kaya for their contribution to a different branch of the code. Before them, also Nirvana Coppola, Nicolas M\"uller, and Simone Muselli were in the project started at the Arizona Winter School 2020 in the group of Jennifer Balakrishnan.

\section*{Declaration of AI use}

The authors began developing the Coleman-integration algorithm and its {\sffamily Magma} implementation in 2024.
In 2026, they used ChatGPT (OpenAI) and Claude (Anthropic) to assist with drafting and editing text, discussing mathematical arguments, and checking and revising code.

For \cref{appendix:a sc and mw}, Claude proposed the main approach and generated much of the computational code.
The authors subsequently reviewed and revised the mathematical arguments and code.

The authors take full responsibility for the mathematical results, {\sffamily Magma} code, and final content of the paper.

\providecommand{\bysame}{\leavevmode\hbox to3em{\hrulefill}\thinspace}
\providecommand{\MR}{\relax\ifhmode\unskip\space\fi MR }
\providecommand{\MRhref}[2]{%
	\href{http://www.ams.org/mathscinet-getitem?mr=#1}{#2}
}
\providecommand{\href}[2]{#2}

\end{document}